%% file: article.tex
\documentclass[a4paper]{scrartcl}
\input{packages.inc.tex}
\input{macros.inc.tex}

\newcommand{\myTitle}{Parametric Finite Element Discretization of the Surface Stokes Equations}
\newcommand{\mySubtitle}{Inf-sup stability and discretization error analysis}

\newcommand{\myAbstract}{%
We study a higher-order surface finite element (SFEM) penalty-based discretization of the tangential surface Stokes problem. Several discrete formulations are investigated which are equivalent in the continuous setting. The impact of the choice of discretization of the diffusion term and of the divergence term on numerical accuracy and convergence, as well as on implementation advantages, is discussed. We analyze the inf-sup stability of the discrete scheme in a generic approach by lifting stable finite element pairs known from the literature. A discretization error analysis in tangential norms then shows optimal order convergence of an isogeometric setting that requires only geometric knowledge of the discrete surface.
}
\newcommand{\myKeywords}{%
surface Stokes equation; inf-sup stability; surface finite element method; penalty method; a priori error estimates; higher order surface approximation%
}

\begin{document}

\title{\myTitle}
\subtitle{\mySubtitle}

\author{%
Hanne Hardering%
  \thanks{Technische Universit{\"a}t Dresden, %
          Institut f{\"u}r Numerische Mathematik, %
          D-01069 Dresden, Germany %
          (\url{hanne.hardering@tu-dresden.de}).}
~~and~~
Simon Praetorius%
  \thanks{Technische Universit{\"a}t Dresden, %
          Institut f{\"u}r Wissenschaftliches Rechnen, %
          D-01069 Dresden, Germany %
          (\url{simon.praetorius@tu-dresden.de}).}%
}

\maketitle

\paragraph*{Abstract.}
\myAbstract

\paragraph*{Keywords.}
\myKeywords

\input{01_introduction.tex}
\input{02_discrete_problem.tex}
\input{03_inf_sup.tex}
\input{04_geometric_errors.tex}
\input{05_error_analysis.tex}

\input{06_numerical_experiments.tex}

\section*{Acknowledgements}
The authors wish to thank Arnold Reusken for fruitful discussions and the German Research Foundation (DFG) for financial support within the Research Unit ``Vector- and Tensor-Valued Surface PDEs'' (FOR 3013) with project no. HA 10104/1-1 and PR 1705/1-1.

\input{09_appendix.tex}

\bibliographystyle{alpha}
\bibliography{references}

\end{document}

%% file: packages.inc.tex
\usepackage[left=2cm,right=2cm]{geometry}
\usepackage{amsthm}
\usepackage[dvipsnames]{xcolor}
\usepackage{graphicx}
\usepackage{epstopdf}
\usepackage{tensor}
\usepackage{colonequals}
\usepackage{verbatim}
\usepackage{mathtools}
\usepackage{subcaption}
\usepackage[numbers]{natbib}
\usepackage{appendix}
\usepackage{hyperref}
\usepackage[capitalize,nameinlink]{cleveref}
\hypersetup{
    colorlinks,
    allcolors={green!50!black},
    urlcolor={red!50!black}
}

\theoremstyle{plain}
\newtheorem{theorem}{Theorem}[section]
\newtheorem{lemma}[theorem]{Lemma}%[section]
\newtheorem{corollary}[theorem]{Corollary}%[section]
\newtheorem{remark}[theorem]{Remark}%[section]
\newtheorem{example}[theorem]{Example}%[section]

\theoremstyle{plain}
\newtheorem{problem}{Problem}
\newtheorem{property}{Property}
\newtheorem{definition}{Definition}

\crefname{problem}{problem}{problems}
\Crefname{problem}{Problem}{Problems}
\crefname{property}{property}{properties}
\Crefname{property}{Property}{Properties}
\crefname{figure}{figure}{figures}
\Crefname{figure}{Figure}{Figures}

\ifpdf
  \DeclareGraphicsExtensions{.eps,.pdf,.png,.jpg}
\else
  \DeclareGraphicsExtensions{.eps}
\fi

\newcounter{subeq}
\renewcommand{\thesubeq}{\theequation\alph{subeq}}
\newcommand{\newsubeqblock}{\setcounter{subeq}{0}\refstepcounter{equation}}
\newcommand{\mysubeq}{\refstepcounter{subeq}\tag{\thesubeq}}

%% file: macros.inc.tex
\usepackage{amsmath,amsfonts,amssymb,amsopn}
\usepackage{bm}
\usepackage{mathabx}
\usepackage{xargs}
\usepackage{xifthen}

\def\R{\mathbb{R}}

\DeclareMathOperator{\Id}{Id}
\DeclareMathOperator{\tr}{tr}

\DeclareMathSymbol{\shortminus}{\mathbin}{AMSa}{"39}

\newcommand{\transposed}[1]{#1^\mathrm{t}}

\newcommand{\G}{\Gamma}
\newcommand{\Gh}{\Gamma_h}

\newcommand{\g}{\omega}
\newcommand{\gh}{\omega_h}

\newcommand{\T}{\mathcal{T}}
\newcommand{\Tri}[1][\G]{\T({#1})}

\newcommand{\El}{T}
\newcommand{\ElRef}{T^\mathrm{ref}}

\newcommand{\TriH}{\T_h}
\newcommand{\TriBar}{\bar{\T}_h}
\newcommand{\TriHat}{\hat{\T}_h}

\newcommand{\leqC}{\lesssim}

\newcommand{\nn}{\bm{n}}
\newcommand{\nnh}{\bm{n}_h}
\newcommand{\nnhTilde}{\nn^\sharp_{h}}

\newcommand{\PP}{\bm{P}}
\renewcommand{\P}{P}
\newcommand{\PPh}{\PP_{\!h}}

\newcommand{\QQ}{\bm{Q}}

\newcommand{\QQh}{\QQ_{\!h}}

\newcommandx{\TangentBundle}[2][1={}]{\mathrm{T}_{#1}{#2}}

\newcommand{\kg}{{k_g}}
\newcommand{\kk}{{k_\text{\scalebox{0.8}{$K$}}}}

\newcommand{\ku}{{k_u}}

\newcommand{\Weingarten}{\mathcal{W}}

\newcommand{\K}{K}
\newcommand{\Kh}{\K_h}
\newcommand{\KhTilde}{\K^\sharp_{h}}

\newcommand{\Grad}{\bm{\nabla}}
\newcommand{\GradG}{\Grad_{\G}}
\newcommand{\GradGh}{\Grad_{\Gh}}

\newcommand{\gradG}{\nabla_{\G}}
\newcommand{\gradGh}{\nabla_{\Gh}}

\newcommand{\Div}{\operatorname{Div}}
\newcommand{\DivG}{\Div_{\G}}

\renewcommand{\div}{\operatorname{div}}
\newcommand{\divG}{\div_{\G}}
\newcommand{\divGh}{\div_{\Gh}}

\newcommand{\CurlG}{\operatorname{curl}_{\G}}

\newcommand{\LaplaceG}{\bm{\Delta}_{\G}}

\newcommand{\D}{\mathcal{D}}
\newcommand{\gradh}{\nabla}
\newcommand{\DG}{\D_{\G}}
\newcommand{\DGh}{\D_{\Gh}}

\newcommand{\Inner}[3]{\big({#1} \,,\, {#2}\big)_{#3}}
\newcommand{\Abs}[1]{\left\lvert{#1}\right\rvert}
\newcommand{\Norm}[2]{\lVert{#1}\rVert_{#2}}
\newcommand{\HalbNorm}[2]{\lvert{#1}\rvert_{#2}}
\newcommand{\Set}[1]{\big\{{#1}\big\}}
\newcommand{\Jump}[1]{\ldbrack{#1}\rdbrack}

\newcommand{\EnergyNorm}[2]{{\lvert\kern-0.25ex\lvert\kern-0.25ex\lvert{#1}\rvert\kern-0.25ex\rvert\kern-0.25ex\rvert}_{#2}}

\newcommand{\avg}[2]{\langle{#1}\rangle_{#2}}
\newcommand{\restr}[2]{\left.{#1}\right|_{#2}}

\newcommandx{\SobolevSpace}[3][1=s,2=p]%
  {\bm{W}^{#1,#2}({#3})}
\newcommandx{\SobolevSpaceTan}[3][1=s,2=p]%
  {\bm{W}_\mathrm{tan}^{#1,#2}({#3})}
\newcommandx{\SobolevSpaceAmb}[3][1=s,2=p]%
  {\bm{W}^{#1,#2}({#3})}
\newcommandx{\SobolevNormTan}[4][1=s,2=p]%
  {\Norm{#3}{\SobolevSpaceTan[#1][#2]{#4}}}
\newcommandx{\SobolevNormAmb}[4][1=s,2=p]%
  {\Norm{#3}{\SobolevSpaceAmb[#1][#2]{#4}}}
\newcommandx{\SobolevNorm}[4][1=s,2=p]%
  {\Norm{#3}{\SobolevSpace[#1][#2]{#4}}}

\newcommandx{\HSpace}[2][1={1}]%
  {\bm{H}^{#1}({#2})}
\newcommandx{\HSpaceTan}[2][1={1}]%
  {\bm{H}_\mathrm{tan}^{#1}({#2})}
\newcommandx{\HSpaceAmb}[3][1={1},2={3}]%
  {[H^{#1}({#3})]^{#2}}
\newcommandx{\HNormTan}[3][1={1}]%
  {\Norm{#2}{\HSpaceTan[#1]{#3}}}
\newcommandx{\HNormAmb}[3][1={1}]%
  {\Norm{#2}{\HSpaceAmb[#1]{#3}}}
\newcommandx{\HNorm}[3][1={1}]%
  {\Norm{#2}{\HSpace[#1]{#3}}}

\newcommandx{\LSpace}[3][1={2},2={}]%
  {\bm{L}^{#1}_{#2}(#3)}
\newcommandx{\LSpaceTan}[2][1={2}]%
  {\bm{L}_\mathrm{tan}^{#1}(#2)}
\newcommandx{\LSpaceAmb}[3][1={2},2={3}]%
  {[L^{#1}(#3)]^{#2}}
\newcommandx{\LNormTan}[3][1={2}]%
  {\Norm{#2}{\LSpaceTan[#1]{#3}}}
\newcommandx{\LNormAmb}[3][1={2}]%
  {\Norm{#2}{\LSpaceAmb[#1]{#3}}}
\newcommandx{\LNorm}[3][1={2}]%
  {\Norm{#2}{\LSpace[#1]{#3}}}
\newcommandx{\LTwoNorm}[3][3={}]%
  {\Norm{#1}{#2}}

\newcommandx{\LSpaceAvg}[3][1={2},2={}]%
{L^{#1}_{0{#2}}(#3)}

\newcommand{\LInfNorm}[2]{\LNorm[\infty]{#1}{#2}}

\newcommandx{\Vh}[1][1=h]{\bm{V}_{\!\!{#1}}}

\newcommand{\VelocitySpace}[1][h]{\bm{V}_{#1}}

\newcommand{\VelocitySpaceZero}[1][h]{\bm{V}_{0,#1}}
\newcommand{\VelocitySpaceFlatZero}[1][h]{\bar{\bm{V}}_{0,#1}}
\newcommand{\PressureSpace}[1][h]{Q_{#1}}
\newcommand{\PressureSpaceAvg}[1][h]{Q_{0,#1}}

\newcommand{\NormZero}[2][]{\Norm{#2}{0,h\ifthenelse{\isempty{#1}}{}{,#1}}}
\newcommand{\NormOne}[2][]{\Norm{#2}{1,h\ifthenelse{\isempty{#1}}{}{,#1}}}

\newcommand{\GNorm}[2][]{\Norm{#2}{\star,h\ifthenelse{\isempty{#1}}{}{,#1}}}
\newcommand{\ahNorm}[2][]{\HNormTan[1]{\PPh{#2}}{\ifthenelse{\isempty{#1}}{\Gh}{#1}}}
\newcommand{\AhNorm}[2][]{\EnergyNorm{#2}{A_h\ifthenelse{\isempty{#1}}{}{,#1}}}

\newcommand{\FSpace}{\mathbb{P}}
\newcommand{\PolynomialSpace}[1]{\FSpace_{#1}}

\newcommand{\RefText}{\mathrm{ref}}
\newcommand{\Broken}{\mathrm{b}}

\newcommand{\I}{I}

\newcommand{\IRef}{\I^\mathrm{ref}}
\newcommand{\II}{\mathcal{I}}

\newcommand{\IIVh}{\II_{\VelocitySpace[h]}}
\newcommand{\IIQh}{\II_{\PressureSpace[h]}}

\newcommandx{\PDQ}[4][1={\Weingarten},2={\PP},3={n}]{\sum_{r=1}^{k+l}{#1}\bigotimes_{r,l+1}{#2}\big[{#4}\underset{r}{\cdot}{#3}\big]}

\newcommand{\rhsConst}{\Theta}

%% file: 01_introduction.tex
\section{Introduction} We consider the tangential surface Stokes equations posed on a closed two-dimensional smooth hypersurface $\G\subset\R^3$:
\begin{problem}\label{prob:Problem1}
   Find a tangential velocity vector field $\bm{u}\colon\G\to\TangentBundle{\G}$ and a scalar pressure field $p\colon\G\to\R$ with $\int_\G p\,\textrm{d}s=0$, such that
   \begin{subequations}\label{eq:stokes-problem}
   \begin{align}
     \label{eq:stokes-problem-1}
     -\frac{1}{2}\DivG\big(\GradG\bm{u} + \transposed{\GradG}\bm{u}\big) + \bm{u} - \gradG p &= \bm{f} &\text{on }\G\,, \\
     \label{eq:stokes-problem-2}
     \divG\bm{u} &= 0 &\text{on }\G\,.
   \end{align}
   \end{subequations}
   for some external forcing $\bm{f}\colon\G\to\TangentBundle{\G}$, where $\GradG$ denotes the covariant derivative of a vector field, $\DivG$ the tangential surface divergence of a tensor field, and $\divG$ the surface divergence of a vector field.
\end{problem}

The tensor divergence, $\DivG$, has to be understood as the $L^2(\G)$-adjoint operator to the covariant derivative\footnote{The definition of the tensor divergence in the surface Stokes problem is different to some other definitions where it is defined as the trace of the gradient, cf. \cite{JankuhnEtAl2018Incompressible,Fries2018HigherOrder,PSW2021Navier}.} $\GradG$. For convenience, we introduce the symmetric gradient $\bm{E}_\G(\bm{u}) \colonequals \frac{1}{2}\big(\GradG\bm{u} + \transposed{\GradG}\bm{u}\big)$, where the transposed operator $\transposed{\GradG}\bm{u}=\transposed{(\GradG\bm{u})}$ has to be understood as the transposed tensor in a coordinate system.

Note that we have introduced a mass term, representing some friction with the surface, to avoid technical difficulties in the overall discussion with the null-space of the symmetric gradient operator, the Killing vector-fields. We want to focus on the analysis of the discretization of the problem ignoring such additional complications. The influence of a zero-order term is also investigated in \cite{PSW2021Navier,BonitoEtAl2020Divergence} upon others.

An equivalent formulation introduced by intrinsic modeling in \cite{Scriven1960Dynamics} and favored in some groups for the simulation of surface flow, e.g., \cite{ReutherVoigt2018Solving,RV2015Interplay}, can be related to \eqref{eq:stokes-problem} by using the relation $\DivG\transposed{\GradG}\bm{u} = \gradG\divG\bm{u} + K\bm{u}$ with $K$ the surface Gaussian curvature, see \cite{JankuhnEtAl2018Incompressible,PSW2021Navier} for some basic calculus arguments.
% %
This allows to rewrite the surface Stokes problem in the following form:

\begin{problem}\label{prob:Problem2}
   Find a tangential velocity vector field $\bm{u}\colon\G\to\TangentBundle{\G}$ and scalar pressure field $p\colon\G\to\R$ with $\int_\G p\,\textrm{d}s=0$, such that
   \begin{subequations}\label{eq:stokes-problem2}
     \begin{align}
       \label{eq:stokes-problem2-1}
       -\frac{1}{2}(\LaplaceG\bm{u} + K\bm{u}) + \bm{u} - \gradG p &= \bm{f} &\text{on }\G\,, \\
       \label{eq:stokes-problem2-2}
       \divG\bm{u} &= 0 &\text{on }\G\,,
     \end{align}
   \end{subequations}
   with $\LaplaceG=\DivG\GradG$ the Bochner--Laplacian.
\end{problem}

The numerical approaches for discretizing these problems are manifold, ranging from purely discrete methods comparable to a surface staggered grid approach \cite{NRV2017Discrete}, immersed interface methods \cite{HL2017stabilized,OlshanskiiEtAl2018Finite}, spectral methods for spherical or radial surfaces \cite{GA2018Hydrodynamic}, to surface finite element methods with discontinuous but tangential \cite{BonitoEtAl2020Divergence,LLS2020Divergence,DemlowNeilan2023PenaltyFree} or continuous but non-tangential \cite{Fries2018HigherOrder,ReutherVoigt2018Solving,RNV2020Numerical,BrandnerEtAl2022FiniteElement} velocity, and reformulations as stream-function formulation \cite{NVW2012Finite,Reu2018Stream,BrandnerReusken2019,BrandnerEtAl2022FiniteElement}. Some groups focus on geometrically unfitted finite element schemes, such as the Trace finite element methods (TraceFEM) \cite{JankuhnReusken2019Higher,OY2019Penalty,JankuhnEtAl2021Error}.

We consider a $C^0$-surface approximation using an elementwise polynomial parametrization and a $H^1$-conforming but non-tangential vector field approximation in a penalty approach.
Both formulations \Cref{prob:Problem1,prob:Problem2} are discretized using componentwise higher-order surface finite elements as in \cite{Demlow2009Higher,DE2013Finite} for scalar functions and in \cite{HLL2020Analysis,HP2022Tangential} for vector fields. Tangentiality of the surface velocity is enforced by a penalty approach, but in contrast to other groups, see, e.g., \cite{Fries2018HigherOrder,JankuhnReusken2019Higher}, we focus on an isogeometric setting in which only geometric quantities from the discrete surface itself are used in the discretization.

Surface finite elements generally mirror the approximation and local properties of their Euclidean counterparts, although geometric errors must be taken into account, see \Cref{sec:geometric_errors}. The key challenge we address is the global property of discrete inf-sup stability, which is essential for discussing the well-posedness and error estimates in finite element analysis, see \Cref{sec:error_estimates}.
In this paper we adapt a technique of Stenberg \cite{Stenberg1984Analysis,Stenberg1990Technique} to translate the global property into a family of local estimates on macroelement patches. The global estimate can be recovered from the local ones if they are independent of the patches themselves.
We prove that this approach also works on curved macroelements by lifting the local inf-sup estimates from the flat patches. This allows to obtain inf-sup stability for many finite element pairs that are based on stable methods for the Euclidean space without needing individual verifications for each case, see \Cref{sec:inf_sups}.

All of this is complemented by numerical experiments in \Cref{sec:numerical-experiments} for a range of stable elements, the higher order Taylor--Hood element, the MINI element, a conforming Crouzeix--Raviart element and a pair of quadratic velocity and piecewise constant pressure approximations. These discrete spaces fit in the analytic framework and show the expected behavior for an experiment of flow on a spherical surface.

%% file: 02_discrete_problem.tex
\section{Discretization of the Surface Stokes Problem}
The starting point for the discretization of \Cref{prob:Problem1,prob:Problem2} is the variational formulation. A discrete variational form requires an approximation of the geometry as well as of the functions and vector fields. We consider approximations of the smooth surface $\G$ by piecewise polynomial surfaces $\Gh$ via an elementwise interpolation. On these discrete surfaces, a surface finite element discretization of functions and vector fields in (broken) Sobolev spaces is introduced, which is defined by lifting from piecewise flat reference surfaces $\hat{\G}_h$. Finally, four discrete problem formulations are presented.

% --------------------------------------------------
\subsection{Variational Form of the Continuous Problem}
In the following, we consider the surface Stokes equations in a variational form. Since both \Cref{prob:Problem1,prob:Problem2} are equivalent on a smooth surface $\G$, we combine them into one formulation.

\begin{problem}\label{prob:VariationalProblem}
	Find a tangential vector field $\bm{u}\in\HSpaceTan[1]{\G}$ and a pressure
	field $p\in\LSpaceAvg[2]{\G}$ such that
	\begin{subequations}\label{eq:variational-form}
		\begin{align}
			\label{eq:variational-form-1}
			a(\bm{u},\bm{v}) + b(\bm{v},p) &= \Inner{\bm{f}}{\bm{v}}{\G} &&
			\forall\bm{v}\in\HSpaceTan[1]{\G}\,, \\
			\label{eq:variational-form-2}
			b(\bm{u},q) &= 0 && \forall q\in \LSpaceAvg[2]{\G}\,,
		\end{align}
	\end{subequations}
	with bilinear forms defined as
	\begin{align}
		\newsubeqblock
		\mysubeq a(\bm{u},\bm{v})
		\colonequals&\, \Inner{\bm{E}_\G(\bm{u})}{\bm{E}_\G(\bm{v})}{\G} +
		\Inner{\bm{u}}{\bm{v}}{\G} \\
		\mysubeq
		=&\, \frac{1}{2}\Inner{\GradG\bm{u}}{\GradG\bm{v}}{\G} + \frac{1}{2}\Inner{\divG\bm{u}}{\divG\bm{v}}{\G} -
		\frac{1}{2}\Inner{K\bm{u}}{\bm{v}}{\G} + \Inner{\bm{u}}{\bm{v}}{\G}\,, \\
		b(\bm{u},p) \colonequals&\, \Inner{p}{\divG\bm{u}}{\G}\,,
	\end{align}
	for tangential vector fields $\bm{u},\bm{v}\in\HSpaceTan[1]{\G}$ and scalar
	field $p\in\LSpaceAvg[2]{\G}$.
\end{problem}

Here, $\bm{u}\in\HSpaceTan[1]{\G}$ denotes tangential vector fields, with $\bm{u}\cdot\bm{n}=0$ almost everywhere, for $\bm{n}$ the surface surface normal field of $\G$. The space $\LSpaceAvg[2]{\G}$ contains square-integrable functions with zero mean over the surface $\G$.

\begin{remark}
	For pressure fields $p\in H^1(\G)\cap\LSpaceAvg[2]{\G}$ on a smooth and closed surface
	$\G$, the divergence term $b(\bm{u},p)$ can be written in the equivalent form
	$b(\bm{u},p) = -\Inner{\gradG{p}}{\bm{u}}{\G}$.
\end{remark}

The following lemma in combination with the coercivity of the bilinear form
$a(\cdot,\cdot)$ implies well-posedness of the \Cref{prob:VariationalProblem}
and existence of a unique solution.

\begin{lemma}[Continuous inf-sup condition]\label{lem:continuous-inf-sup}
	For a closed and compact $C^2$-surface $\G$ there exists a constant $\beta>0$, such that
	\begin{align}
		\sup_{\bm{v}\in\HSpaceTan[1]{\G}} \frac{b(\bm{v},q)}{\HNormTan[1]{\bm{v}}{\G}}
		&\geq \beta \LTwoNorm{q}{\G}
		&&\text{for all }q\in\LSpaceAvg[2]{\G}\,.
	\end{align}
\end{lemma}
\begin{proof}
	See, e.g., \cite[Lemma 2.1]{OlshanskiiEtAl2018Finite}.
\end{proof}

\begin{lemma}[Regularity of the continuous problem]\label{lem:regularity}
	 We assume that $\G$ is $C^2$-smooth, compact and closed. The \Cref{prob:VariationalProblem} is well-posed. Denote by $(\bm{u},p)$ the (unique) solution of this problem.
	If $\bm{f}\in\LSpaceTan[2]{\G}$, then
		$\bm{u}\in\HSpaceTan[2]{\G}$ and $p\in H^1(\G)\cap\LSpaceAvg[2]{\G}$ and
		there exists a constant $C_1>0$ such that,
		\begin{align}
			\HNormTan[2]{\bm{u}}{\G} + \Norm{p}{H^1(\G)} &\leq C_1 \Norm{\bm{f}}{\G}\,.
		\end{align}
\end{lemma}
\begin{proof}
	See, e.g., \cite[Lemma 2.1]{OlshanskiiEtAl2021InfSup}.
\end{proof}

%In the following sections we introduce, analyze, and evaluate a numerical discretization of \Cref{prob:VariationalProblem}.

%For the analysis of the finite element method for the surface Strokes equations, especially for the establishment of discrete inf-sup stability properties, we additionally introduce a partitioning of the discrete surface grid into so-called ``macroelements'', i.e., connected patches of surface elements that can be flattened into two-dimensional reference patches. We will then construct a proper lifting of stability properties of functions defined on these two-dimensional flat macroelements back to the piecewise polynomial discrete surface.

%With this framework in place, we define surface finite element discretizations of these mappings by lifting finite element spaces from a reference patch and giving the properties that are preserved by the lifting. We then use four equivalent variational formulations of \Cref{prob:Problem1,prob:Problem2} to present four different discretizations, all of which can be found in the literature. After briefly introducing the norms and notation that we need, we then give first estimates for the differences of the four discrete problems we consider in the rest of the paper.

% -----------------------------------------------------------------------------
\subsection{Vector Fields on the Surface Discretizations}
In the surface finite element method \cite{DE2013Finite,Demlow2009Higher}, the smooth surface $\G$ is approximated by a piecewise polynomial representation $\Gh$ that is described by an elementwise parametrization over a reference element $\ElRef$. The topology is thereby fixed by a piecewise flat triangulation $\hat{\G}_h$ of the surface.

Let $U_{\delta}(\G)$ be a $\delta$-neighborhood of $\G$ such that the closest-point projection $\pi\colon U_\delta(\G)\to\G$ is well defined. We assume that there exists a shape-regular sequence of affine grids $\{\hat{\G}_h\}_{h>0}$ approximating $\G$ with all $\hat{\G}_h\subset U_\delta(\G)$ and vertices on the smooth surface. For each $h$ let $\Tri[\hat{\G}_h]=\Set{\hat{\El}\mid\hat{\El}=F_{\hat{\El}}(\ElRef)}$ denote the triangulation of $\hat{\G}_h$, with $F_{\hat{\El}}$ affine mappings from the reference element $\ElRef$ to $\hat{\El}$.
To simplify the discussion, we assume that this sequence of triangulations is quasi-uniform, with $h$ being the uniform bound on element size. We also assume that all grids are conforming. The notation with the $\hat{\cdot}$ symbol is used throughout this document to denote piecewise flat domains or fields on these domains.

A piecewise polynomial surfaces approximation is constructed as parametrization over the elements of the picewise flat surface $\hat{\G}_h$. We denote by $\kg\geq 1$ the polynomial order of the surface parametrization. Let $\pi_{h,\kg}\circ F_{\hat{\El}}\colonequals \IRef_{\kg}[\pi\circ F_{\hat{\El}}]$ for $\hat{\El}\in\Tri[\hat{\G}_h]$ denote the elementwise Lagrange interpolation of the closest-point projection $\pi$ of order $\kg$.
A parametric discrete surface $\G_{h,\kg}$ is defined as the union of images of the elements of $\hat{\G}_h$ under $\pi_{h,\kg}$, i.e.,
\[
	\G_{h,\kg}\colonequals \bigcup_{\hat{\El}\in\Tri[\hat{\G}_h]} \pi_{h,\kg}(\hat{\El})\,.
\]
As before, $\Tri[\G_{h,\kg}]=\Set{\El\mid\El=\pi_{h,\kg}(\hat{\El}),\;\hat{\El}\in\Tri[\hat{\G}_h]}$ denotes the triangulation of $\G_{h,\kg}$.

To simplify the notation, we always denote the polynomial order of the surface by $\kg$ and thus write $\pi_h\colonequals \pi_{h,\kg}$ and $\Gh\colonequals \G_{h,\kg}$ in the following.

% --------------------------------------------------
% \subsection{Representation of Surface Vector Fields}
% TODO: Hier wird eigentlich nicht viel zu vector felder Representation gesagt
While it is possible to have continuous and tangential vector fields on $\G$, this does not hold for the piecewise polynomial surface $\Gh$. In this paper we want to preserve the continuity while giving up the tangentiality. Therefore, we introduce vector fields as mapping from $\Gh\to\R^3$. To relate these fields to tangential vector fields, we need tangential projections.

Denote by $\PP=\Id-\nn\otimes\nn$ the smooth surface tangential
projection operator, $\QQ=\Id-\PP$ the surface normal projection on
$\G$, for $\nn$ the (outward) surface normal vector field.
Let $\nnh$ denote the discrete surface normal vector field, defined
inside the elements of a triangulation of $\Gh$, formally glued together. A
corresponding discrete tangential projection is introduced as
$\PPh=\Id-\nnh\otimes\nnh$, and, correspondingly,
$\QQh=\Id-\PPh$, on $\Gh$.

A projected surface derivative $\GradG$ of vector fields
$\bm{u}\colon\G\to\R^3$ is introduced as the projection of the componentwise
Euclidean derivative, $\GradG\bm{u} \colonequals
\PP(\restr{\D\bm{u}^e}{\G})\PP$, with $\D$ the derivative w.r.t.
Cartesian coordinates and $\bm{u}^e$ a smooth extension of $\bm{u}$ in the
surrounding of $\G$. This derivative can be written intrinsically using
componentwise surface derivatives $\DG$, i.e., $\GradG\bm{u} =
\PP(\DG\bm{u})$. For tangential fields this derivative coincides with the
surface covariant derivative. For scalar fields $p\colon\G\to\R$, all the
surface derivatives coincide\footnote{We understand the derivative of a scalar field as a row vector
	such that the projection matrix must be multiplied from the right. In some
	references, the derivative is understood as the gradient as a column vector instead. This leads to
	the transposed notation $\PP(\restr{\D p^e}{\G})$.}, that is, $\gradG p = \DG p
= (\restr{\D p^e}{\G})\PP$.
Analogously, a discrete surface derivative is introduced elementwise as
$\GradGh\bm{u}_h \colonequals \PPh(\DGh\bm{u}_h)$ for vector fields
$\bm{u}_h\colon\Gh\to\R^3$. Higher order derivatives can be defined in a similar
manner by projecting all components of the Euclidean derivative.

On surface subdomains $\g\subseteq\G$ and $\gh\subseteq\Gh$ we denote by $\LSpace[2]{\g}$ and $\LSpace[2]{\gh}$ the usual space of square integrable tensor fields with corresponding $L^2$-inner product $\Inner{\bm{u}}{\bm{v}}{\g_{(h)}}$ and $L^2$-norm $\Norm{\cdot}{\g_{(h)}}$ defined in terms of the pointwise Frobenius inner-product and Frobenius norm, respectively. For scalar fields we will use light symbols, e.g., $L^2(\g)$, instead. We denote by $\LSpaceAvg[2]{\g}\colonequals\Set{p\in L^2(\g)\mid\avg{p}{\g}=0}$ the space of scalar fields $p$ with zero mean value $\avg{p}{\g}\colonequals\frac{1}{|\g|}\int_{\g} p\,\text{d}\g$.

We introduce the Sobolev spaces $\HSpaceTan[s]{\g}$ and norms for tangential vector fields $\bm{u}\colon\g\to\R^3$ with $\bm{u}\cdot\nn=0$ on $\g$ and the embedded Sobolev spaces $\HSpace[s]{\gh}=\HSpaceAmb[s]{\gh}$ for vector fields $\bm{u}_h\colon\gh\to\R^3$ on the discrete surface $\gh$, with associated norms
\begin{align}
	\HNormTan[s]{\bm{u}}{\g}^2 &\colonequals \sum_{j=0}^s	\Norm{\GradG^j\bm{u}}{\g}^2\,, &
	\HNorm[s]{\bm{u}_h}{\gh}^2 &\colonequals \sum_{j=0}^s	\Norm{\DGh^j\bm{u}_h}{\gh}^2\,.
\end{align}
The Sobolev spaces $\SobolevSpaceTan[s][p]{\g}$ and $\SobolevSpace[s][p]{\gh}$ are defined analogously, with $\SobolevSpaceTan[s][2]{\g}=\HSpaceTan[s]{\g}$ and $\SobolevSpace[s][2]{\gh}=\HSpace[s]{\gh}$.

For tensor fields or derivatives that are defined elementwise without requiring additionaly continuity across element faces, we introduce broken Sobolev spaces, see also \cite{EG2021FiniteI}. In order to highlight this, we replace the domain argument by its corresponding triangulation.
\begin{definition}\label{def:broken-sobolev-spaces}
	For $\TriH=\Tri[\Gh]$ a triangulation of $\Gh$ and $\T$ a triangulation of $\G$ defined by mapping the elements of $\TriH$ using $\pi$, we denote by
	\begin{align*}
		\HSpaceTan[s]{\T} &\colonequals\Set{\bm{u}\in\LSpace[2]{\G,\R^3}\mid\bm{u}\cdot\bm{n}=0,\;\restr{\bm{u}}{\El}\in\HSpaceTan[s]{\El},\forall\El\in\T}\,,\\
		\HSpace[s]{\TriH} &\colonequals\Set{\bm{u}_h\in\LSpace[2]{\Gh,\R^3}\mid\restr{\bm{u}_h}{\El}\in\HSpace[s]{\El},\forall\El\in\TriH}
	\end{align*}
	the \textbf{broken Sobolev spaces} with norms defined elementwise, i.e.,
	\begin{align*}
		\HNormTan[s]{\bm{u}}{\T}^2 &\colonequals \sum_{\El\in\T} \HNormTan[s]{\restr{\bm{u}}{\El}}{\El}^2\,, &
		\HNorm[s]{\bm{u}_h}{\TriH}^2 &\colonequals \sum_{\El\in\TriH} \HNorm[s]{\restr{\bm{u}_h}{\El}}{\El}^2\,.
	\end{align*}
\end{definition}

To compare scalar and vector fields of the smooth surface $\G$ and on the discrete surface $\Gh$, we introduce extension and lifting of fields by composition with the surface closest-point projection $\pi$, i.e., for $u\colon\G\to\R$ we introduce the extension $u^e\colonequals u\circ\pi\colon U_\delta(\G)\to\R$ and vice versa for $u_h\colon\Gh\to\R$ we write for the surface lifting $u_h^\ell\colonequals u_h\circ\restr{\pi}{\Gh}^{-1}\colon\G\to\R$. The extension of the lifting is still denoted with an ${}^e$ by $u_h^e\colonequals (u_h^\ell)^e$. For vector fields extensions and liftings are defined analogously.

We introduce the surface Weingarten maps $\Weingarten\colonequals -\GradG\nn$ and elementwise $\Weingarten_h\colonequals -\GradGh\nnh$ to express the Gaussian curvature for embedded two-dimensional surfaces as $K = \frac{1}{2}(\operatorname{tr}(\Weingarten)^2 - \operatorname{tr}(\Weingarten^2))$ and $K_h = \frac{1}{2}(\operatorname{tr}(\Weingarten_h)^2 - \operatorname{tr}(\Weingarten_h^2))$.
Apart from this intrinsic Gaussian curvature $\Kh$, we introduce another approximation of the continuous Gaussian curvature $\K$ on $\Gh$, denoted by $\KhTilde$. For this field we assume that the following estimate is fulfilled,
\begin{align}
	\Abs{\Inner{(\K^e- \KhTilde)\PPh \bm{v}_h}{\PPh\bm{w}_h}{\Gh}}
	&\leqC h^{\kk} \HNormTan{\PPh\bm{v}_h}{\Gh}\HNormTan{\PPh\bm{w}_h}{\Gh} \label{eq:order-assumptions-K}
\end{align}
for $\bm{v}_h, \bm{w}_h \in \HSpace{\Gh}$ and some order $\kk+1\geq\kg$. We write $a\leqC b$ as an abbreviation for $a\leq C b$ with $C$ depending just on the property of the surface but not on the grid size $h$. The normal and curvature fields can be constructed by, e.g., extracting the geometric information from another discrete surface $\G_{h,\kk}$ with a possibly different approximation order than $\kg$, or by interpolating the continuous surface fields on $\Gh$ in some higher-order function space.

\begin{remark}\label{rem:gauss-curv-approx}
	On a continuous piecewise polynomial surface approximation $\Gh$ of order $\kg$, we have $\kk\geq\kg-1$ for the elementwise surface Gaussian curvature $\KhTilde=\Kh$. This follows from the estimate $\sup_{\El\in\Tri[\Gh]}\LInfNorm{\K^e - \Kh}{\El}\leqC h^{\kg-1}$, cf. \cite{Demlow2009Higher}.
	Numerical experiments indicate that for a piecewise Lagrange parametrization of the surface with even $\kg$, we have $\kk \geq \kg$, cf. \cite{ZavalaniEtAl2024Note}.
\end{remark}

% --------------------------------------------------
\subsection{Construction of Surface Finite Elements}\label{sec:sfem}
The spaces $H^s(\G)$, $s=0,1$, are approximated using surface finite elements \cite{Demlow2009Higher}, i.e., standard finite element spaces lifted to the discrete approximation $\Gh$ of $\G$. Before introducing conforming approximations, we define broken discrete spaces. We follow the definitions given in \cite{EG2021FiniteI}.

Let $\TriHat\subseteq\Tri[\hat{\G}_h]$ denote a subset of the piecewise flat surface triangulation and $\hat{\g}\colonequals \bigcup_{\hat{\El}\in\TriHat} \hat{\El}$ its domain. We introduce the (broken) scalar finite element space
\begin{align}\label{eq:piecewise-flat-broken-finite-element-space}
\hat{S}^\Broken(\TriHat)=\hat{S}^\Broken(\TriHat;\FSpace) &\colonequals\Set{v\in
	L^\infty(\TriHat,\R)\mid \restr{v}{\hat{\El}} \circ
	F_{\hat{\El}}\in \FSpace\text{ for all
	}\hat{\El}\in\TriHat}\,,
\end{align}
constructed from a local reference finite element $(\ElRef,\FSpace,\Sigma)$. The space $\FSpace$ denotes a finite dimensional function space on $\ElRef$, where $\PolynomialSpace{r}\subseteq\FSpace\subset W^{r+1,\infty}(\ElRef)$ with $\PolynomialSpace{r}$ the set of polynomials of order at most $r$, for some $r\geq 1$. The maximal $r$ for which $\PolynomialSpace{r}\subseteq\FSpace$ is called the order of the finite element space. The global space for the special case $\FSpace=\PolynomialSpace{r}$ is abbreviated by $\hat{S}_r^\Broken(\TriHat)$.

Associated to the local reference finite element is a canonical interpolation $\IRef_{\FSpace}$ defined on the domain $D^\RefText$ of the linear forms in $\Sigma\ni\sigma\colon D^\RefText\to\R$, and a corresponding global interpolation operator $\hat{\I}_{h,\FSpace}\colon\hat{D}\to\hat{S}^\Broken(\TriHat)$ with $\hat{\I}_{h,\FSpace}[v]\circ F_{\hat{\El}} \colonequals \IRef_\FSpace[v\circ F_{\hat{\El}}]$, for $\hat{\El}\in\TriHat$ and $\hat{D}=D^\RefText\circ F_{\hat{\El}}$ the set of functions lifted from $\El^\RefText$ to $\hat{\El}$. Here again, we use the abbreviation $\IRef_{r}$ and $\hat{\I}_{h,r}$ for $\FSpace=\PolynomialSpace{r}$.

The broken finite element spaces $\hat{S}^\Broken$ and the interpolation operators $\hat{\I}_{h,\FSpace}$ are lifted to the associated surfaces $\gh=\pi_h(\hat{\g})$, and $\g=\pi(\hat{\g})$ via the mappings $\pi_h$ and $\restr{\pi}{\hat{\g}}$, respectively. The resulting finite element spaces are denoted by $S_h^\Broken$ and $S^\Broken$, and the interpolation operators by  $\I_{h,\FSpace}$ and $\I_{\FSpace}$.

The spaces $\hat{S}^\Broken$, $S_h^\Broken$, and $S^\Broken$ are broken spaces without continuity requirements across element faces. Following the notation of the Sobolev spaces, we remove the $\cdot^\Broken$ and replace the triangulation parameter by the domain parameter to denote the continuous spaces, e.g.,
\begin{align*}
	\hat{S}(\hat{\g}) &\colonequals \Set{v\in\hat{S}^\Broken(\Tri[\hat{\g}])\mid\Jump{v}_e=0,\;\forall e\in\mathcal{E}^\mathrm{int}(\Tri[\hat{\g}])}\,,
\end{align*}
with $\mathcal{E}^\mathrm{int}(\T)\colonequals\Set{e=\partial\El_1\cap\partial\El_2\mid\El_1,\El_2\in\T}$ the set of all interior edges of the triangulation $\T$ and $\Jump{v}_e\colonequals \restr{v}{\El_1} - \restr{v}{\El_2}$ the jump of $v$ across the edge $e$ between the elements $\El_1$ and $\El_2$. Analogously, we define the continuous spaces $S_h(\gh)$ and $S(\g)$.

We formulate two assumptions on the discrete spaces, the classical interpolation and inverse estimate properties. These assumptions allow us later to analyze properties of the discrete formulation of \Cref{prob:DiscreteProblem} and are typical for a large class of finite element spaces.

\begin{property}[Interpolation property]\label{prop:interpolation}
	For a piecewise flat subset $\hat{\g}\subseteq\hat{\G}_h$ with corresponding finite element space $\hat{S}^\Broken(\TriHat;\FSpace)$ of order $r\geq 1$ there exists an interpolation operator $\hat{\II}_{h,\FSpace}\colon L^1(\hat{\g})\to\hat{S}^\Broken(\TriHat;\FSpace)$ with the following approximation properties:

	For all $p\in[1,\infty]$, $0\leq l\leq r+1$, and all $0\leq m\leq l$ it holds
	\begin{equation}\label{eq:interpolation-local-approximation}
		\HalbNorm{v-\hat{\II}_{h,\FSpace} v}{W^{m,p}(\hat{\El})} \leqC h_{\hat{\El}}^{l-m}\HalbNorm{v}{W^{l,p}(\hat{\El})}\,,\quad\forall v\in W^{l,p}(\hat{\El}), \hat{\El}\in\TriHat\,.
	\end{equation}
\end{property}

Note that this interpolation operator is not the canonical $\hat{\I}_{h,\FSpace}$ from the definition of the local finite element spaces, but denotes a different, stable interpolation operator.
In \cite[Sec. 3]{ErnGuermond2017Quasi} the $L^1$-stable interpolation operator $\II_{\El}^\sharp$ is introduced, which satisfies the approximation properties. In \cite[Sec. 18.4]{EG2021FiniteI} the broken $L^2$-orthogonal projection $\II_{\El}^\Broken$ with similar properties is defined. These are elementwise projections onto the local finite element space $\FSpace$.

\begin{property}[Inverse estimate property]\label{prop:inverse-estimates}
	On a piecewise flat subset $\hat{\g}\subseteq\hat{\G}_h$ the finite element space $\hat{S}^\Broken(\TriHat,\FSpace)$ with local reference finite element $(\ElRef,\FSpace,\Sigma)$ be such that $\FSpace\subset W^{l,\infty}(\ElRef)$ for some $l\geq 0$. Let $0\leq m\leq l$ and $1\leq p,q\leq\infty$. If $h$ is small enough then it holds
  \begin{equation}
    \Norm{v_h}{W^{l,p}(\hat{\El})} \leqC h_{\hat{\El}}^{m-l+\frac{2}{p}-\frac{2}{q}} \Norm{v_h}{W^{m,q}(\hat{\El})}\,,\quad\forall v_h\in\hat{S}^\Broken(\TriHat,\FSpace)\,,\;\hat{\El}\in\TriHat\,.
  \end{equation}
\end{property}

See, e.g., \cite[Lemma 12.1]{EG2021FiniteI} for a proof of this property for many finite elements. These inverse estimates can also be lifted to elements of the parametric surface $\Gh$ using the boundedness of the discrete parametrization $\pi_h$, see, e.g., \cite[Lemma 4.3]{HP2022Tangential}.

\begin{definition}[Pair of discrete spaces]\label{def:mixed-spaces}
	We define the \textbf{pair of discrete spaces} $(\VelocitySpace, \PressureSpace)$ componentwise as surface finite element spaces that satisfy \Cref{prop:interpolation,prop:inverse-estimates},
	\[
		(V_h,Q_h)\colonequals (V_h(\Gh),Q_h(\Gh))\colonequals (S_h(\G_h;\FSpace_V), S_h^\Broken(\Tri[\G_h];\FSpace_Q))\subset H^1(\G_h)\times L^2(\G_h)
	\]
	and denote by $\VelocitySpace\colonequals[V_h]^3$ the \textbf{velocity space} and by $\PressureSpace$ the \textbf{pressure space}.
%	The pair of spaces is of \textbf{order} $k\geq 1$ if the local finite element space $\FSpace_V$ is at least of order $k$ and $\FSpace_Q$ is at least of order $k-1$.
	The pair of spaces is of \textbf{order} $\ku =\min\{r_v, r_q+1\}\geq 1$, where $r_v$ and $r_q$ denote the orders of the local finite element spaces $\FSpace_V$ and $\FSpace_Q$, respectively.

	For a subdomain $\gh\subset\Gh$ with boundary $\partial\gh$ the velocity space with zero boundary trace will be denoted by $\VelocitySpaceZero(\gh)\colonequals\Set{\bm{v}\in\VelocitySpace(\gh)\mid\restr{\bm{v}}{\partial\gh}=\bm{0}}$.
\end{definition}

Associated to a pair of discrete spaces $(\VelocitySpace,\PressureSpace)$ is a pair of interpolation operators, which will be denoted by	$(\IIVh, \IIQh)$. These can be defined by a lifting of the local interpolation operators of \Cref{prop:interpolation} to $\Gh$:

\begin{definition}[Global interpolation operators]
	We define two interpolation operators of order $r$ on $\Gh$ in terms of the associated interpolation operators from \Cref{prop:interpolation} on the piecewise flat surface $\hat{\G}_h$:
	\begin{align}
		\II^\Broken_{h,\FSpace}&\colon L^1(\Gh)\to S^\Broken_h(\Gh;\FSpace);& \restr{\II^\Broken_{h,\FSpace}[u]}{\El}
			&\colonequals\hat{\II}_{h,\FSpace}[\restr{u}{\El}\circ\pi_h]\circ\pi_h^{-1},\;\El\in\Tri[\Gh] \\
		\II_{h,\FSpace}&\colon L^1(\Gh)\to S_h(\Gh;\FSpace);& \II_{h,\FSpace}[u]
			&\colonequals\hat{\mathcal{J}}^\text{av}_{h,\FSpace}\big[\hat{\II}_{h,\FSpace}[u\circ\pi_h]\big]\circ\pi_h^{-1}
	\end{align}
	with $S^\Broken_h$ and $S_h$ discrete spaces of order $r$, and $\hat{\mathcal{J}}^\text{av}_{h,\FSpace}$ an associated averaging operator over overlapping DOFs in neighboring elements, as in \cite[Sec. 4.2]{ErnGuermond2017Quasi}.
\end{definition}

For the space $\VelocitySpace$, we denote by $\IIVh\colon L^1(\Gh,\R^3)\to [S_h(\Gh;\FSpace_V)]^3$ with $\IIVh[\bm{v}]_i\colonequals \II_{h,\FSpace_V}[v_i]$ the componentwise interpolation using the global continuous interpolation operator. If the space $\PressureSpace$ is continuous, we identify $\IIQh\colonequals \II_{h,\FSpace_Q}$, otherwise, $\IIQh\colonequals \II^\Broken_{h,\FSpace_Q}$.

Let $\El\in\Tri[\Gh]$ be an element of the parametrized surface $\Gh$. We denote by $\g_{\El}\subset\Gh$ the element-neighborhood of $\El$, w.r.t. overlapping DOFs in the function space $S_h(\Gh)$, and $\g^\Broken_{\El}\equiv\El$ the trivial neighborhood associated to the broken function space $S_h^\Broken(\Gh)$. The following approximation estimates follow classical properties found for Scott-Zhang, Cl\'ement, or Ern-Guermond quasi-interpolation operators, cf. \cite{EG2021FiniteI}, where estimation on an element into its neighborhood are employed.

\begin{lemma}[Approximation properties of the global interpolation operators]\label{lem:interpolation-properties}
	For the interpolation operators $\II^\Broken_{h,\FSpace}$ and $\II_{h,\FSpace}$ the following approximation properties hold:

	For all $p\in[1,\infty]$, $0\leq l\leq r+1$, and all $0\leq m\leq l$ it holds
	\begin{equation}\label{eq:interpolation-local-approximation-Itilde}
		\HalbNorm{u-\II^{(\Broken)}_{h,\FSpace}[u]}{W^{m,p}(\El)} \leqC h_{\El_h}^{l-m}\HalbNorm{u}{W^{l,p}(\Tri[\g^{(\Broken)}_{\El}])}\,,\quad\forall u\in W^{l,p}(\g^{(\Broken)}_{\El}),\;\El\in\Tri[\g^{(\Broken)}_{\El}]\subset\Tri[\Gh]\,.
	\end{equation}
\end{lemma}
\begin{proof}
	The estimates for the broken interpolation operator $\II^\Broken_{h,\FSpace}$ follow from \Cref{prop:interpolation} and the boundedness of the lifting operators $\pi_h$ and $\pi_h^{-1}$, see \cite[Lemma 2.2]{HP2022Tangential}. Using the triangle inequality and considering the case $l=m$, we can additionally conclude that $\II^\Broken_{h,\FSpace}$ is stable in $W^{m,p}$, i.e., $\HalbNorm{\II^{\Broken}_{h,\FSpace}[u]}{W^{m,p}(\El)} \leqC \HalbNorm{u}{W^{m,p}(\El)}$.

	The estimates for the continuous interpolation operator $\II_{h,\FSpace}$ follow by argumentation similar to \cite[Lemma 5.1 and 5.2]{ErnGuermond2017Quasi} for their quasi-interpolation operator $\II_h^\text{av}$ plus lifting of the operator to functions on $\Gh$ by $\pi_h$ similar as above.
\end{proof}

% -----------------------------------------------------------------------------
\subsection{Four Different Discrete Problems}\label{sec:discrete_problem_formulation}
To formulate the discrete variational problem we follow the notation of
\cite{BrandnerEtAl2022FiniteElement,HP2022Tangential}. While the velocity vector
field is embedded in Euclidean space, but the bilinear forms are applied only to
the projected (tangential) component of this field. This requires adding a
treatment of the normal component, here a penalty term denoted by $s_h$, to
enforce that the velocity field is nearly tangential.

While in the continuous setting of \Cref{prob:VariationalProblem} several
formulations are equivalent, this is not necessarily the case in the discrete
setting. Therefore, we denote discretizations of the bilinear form
$a(\cdot,\cdot)$ by $a_i(\cdot,\cdot)$ and discretizations of $b(\cdot,\cdot)$
by $b_j(\cdot,\cdot)$, $i,j\in\{1,2\}$. This results in four different formulations, which are
summarized in the following general discrete problem:

\begin{problem}\label{prob:DiscreteProblem}
	For $i,j=1,2$ let a pair of discrete spaces $(\VelocitySpace,\PressureSpace)$ be such that $\VelocitySpace\subset[H^1(\Gh)]^3$ and $\PressureSpace\subset H^{j-1}(\Gh)$.
	Find $\bm{u}_h\in \VelocitySpace$ and $p_h\in \PressureSpace\cap \LSpaceAvg{\Gh}$ such that
	\begin{subequations}\label{eq:discrete-variational-form}
		\begin{align}
			\label{eq:discrete-variational-form-1}
			a_i(\bm{u}_h,\bm{v}_h) + s_h(\bm{u}_h,\bm{v}_h) + b_j(\bm{v}_h,p_h) &=
			\Inner{\bm{f}\circ\pi}{\bm{v}_h}{\Gh} && \forall\bm{v}_h\in \VelocitySpace\,, \\
			\label{eq:discrete-variational-form-2}
			b_j(\bm{u}_h,q_h) &= 0 && \forall q_h\in \PressureSpace\cap \LSpaceAvg{\Gh}\,,
		\end{align}
	\end{subequations}
	with discrete bilinear forms defined as
	\begin{align}
		a_1(\bm{u},\bm{v})
		&\colonequals
		\Inner{\bm{E}_{\Gh}(\PPh\bm{u})}{\bm{E}_{\Gh}(\PPh\bm{v})}{\Gh} +
		\Inner{\PPh\bm{u}}{\PPh\bm{v}}{\Gh}, \\
		a_2(\bm{u},\bm{v})
		&\colonequals
		\frac{1}{2}\Inner{\GradGh\PPh\bm{u}}{\GradGh\PPh\bm{v}}{\Gh}\!\! -
		\frac{1}{2}\Inner{\KhTilde\PPh\bm{u}}{\PPh\bm{v}}{\Gh}\!\! +
		\Inner{\PPh\bm{u}}{\PPh\bm{v}}{\Gh}, \\
		b_1(\bm{u},p)
		&\colonequals \Inner{p}{\divGh\PPh\bm{u}}{\Gh},\\
		b_2(\bm{u},p)
		&\colonequals - \Inner{\GradGh p}{\bm{u}}{\Gh},\text{ and} \\
		s_h(\bm{u},\bm{v})
		&\colonequals
		\eta h^{-1}\Inner{\bm{u}\cdot\nnh}{\bm{v}\cdot\nnh}{\Gh},\label{eq:sh_def}
	\end{align}
	for $\bm{u},\bm{v}\in\VelocitySpace$ and $p\in \PressureSpace$, with $\eta>0$ and $\KhTilde$ an approximation of the Gaussian curvature of $\G$.
\end{problem}

\begin{remark}
	The penalty term contains the constant $\eta$ as a scaling factor. This allows to control the strength of enforcement of the tangentiality. In the analysis of the method, this constant prefactor does not play a role and we thus fix it the following to $\eta\equiv 1$ for simplicity and drop the constant from the penalty term.
\end{remark}

\begin{remark}
	Instead of using the discrete normal $\nnh$ in \eqref{eq:sh_def}, one can also introduce a better approximation $\nnhTilde$ of the continuous normal with $\LInfNorm{\nnhTilde-\nn}{\Gh}\leqC h^{\kg+1}$ as it is for example done in \cite{HLL2020Analysis,HP2022Tangential} for the Poisson equation. In that case an alternative definition $\tilde{s}_h(\bm{u},\bm{v})\colonequals
	h^{-2}\Inner{\bm{u}\cdot\nnhTilde}{\bm{v}\cdot\nnhTilde}{\Gh}$ leads to optimal error estimates. The analysis can be done analogously to the one in \Cref{sec:error_estimates}, but is excluded here for brevity. In general this choice will lead to slightly better error estimates for the normal part of the solution, and a better tangential $L^2$-estimate in the case $\kg=1$, but it needs the construction of the higher order normal approximation.
\end{remark}

% -----------------------------------------------------------------------------
\begin{definition}[Combined forms and energy norms]
For $\bm{u},\bm{v}\in\HSpaceTan[1]{\G}$ and $p,q\in L^2(\G)$ we denote by $\mathcal{B}$ the combined bilinear form,
\begin{align}
  \mathcal{B}\big((\bm{u},p),(\bm{v},q)\big)
  &\colonequals a(\bm{u},\bm{v}) + b(\bm{v}, p) + b(\bm{u}, q)\;,
\end{align}
and corresponding energy norm $\EnergyNorm{(\bm{u},p)}{\mathcal{B}}^2 \colonequals \HNormTan[1]{\bm{u}}{\G}^2 + \Norm{p}{\G}^2$.

Similarly, for $\bm{u}_h,\bm{v}_h\in\VelocitySpace[h]$ and $p_h,q_h\in \PressureSpace$ we denote by $\mathcal{B}_{ij}$ the combined discrete bilinear form,
\begin{align}
  \mathcal{B}_{ij}\big((\bm{u}_h,p_h),(\bm{v}_h,q_h)\big)
%    &\colonequals A_i(\bm{u}_h,\bm{v}_h) + b_j(\bm{v}_h, p_h) + b_j(\bm{u}_h, q_h) \\
    &\colonequals a_i(\bm{u}_h,\bm{v}_h) + s_h(\bm{u}_h,\bm{v}_h) + b_j(\bm{v}_h, p_h) + b_j(\bm{u}_h, q_h)\;,
\end{align}
with corresponding discrete energy norm $\EnergyNorm{(\bm{u}_h,p_h)}{}^2\colonequals \EnergyNorm{\bm{u}_h}{A_h,\Gh}^2 + \Norm{p_h}{\Gh}^2$, with
\begin{align}
	\EnergyNorm{\bm{u}_h}{A_h,\Gh}^2
	  &\colonequals \ahNorm[\Gh]{\bm{u}_h}^2 + \EnergyNorm{\bm{u}_h}{s_h}^2\,,\\
%  \ahNorm[\gh]{\bm{u}}^2
%	  &\colonequals \HNormTan[1]{\PPh\bm{u}}{\gh}^2 = \Norm{\GradGh\PPh \bm{u}}{\gh}^2 + \Norm{\PPh \bm{u}}{\gh}^2\,,\\
  \EnergyNorm{\bm{u}_h}{s_h}^2
    &\colonequals s_h(\bm{u}_h,\bm{u}_h)\,.
\end{align}
On some domain $\gh\subseteq\Gh$ we additionally introduce weighted norms
\begin{align}
	\NormZero[\gh]{\bm{u}}
    &\colonequals \Norm{\PPh\bm{u}}{\gh} + h^{-1}\Norm{\QQh\bm{u}}{\gh}\,, \\
	\NormOne[\gh]{\bm{u}}
    &\colonequals \HNormTan[1]{\PPh\bm{u}}{\gh} + h^{-1}\Norm{\QQh\bm{u}}{\gh}\,, \\
	\GNorm[\gh]{\bm{u}}
    &\colonequals \HNorm[1]{\bm{u}}{\gh} + h^{-1}\Norm{\QQh\bm{u}}{\gh}\,.
\end{align}
If the domain is omitted, it refers to the whole surface $\Gh$, e.g., $\EnergyNorm{\cdot}{A_h}\colonequals\EnergyNorm{\cdot}{A_h,\Gh}$.
\end{definition}

Note that the discrete energy norms are independent of the formulations of the discrete bilinear forms $a_i$ and $b_j$ and contain only the $H^1$-norm of the projected fields. The weighted norms arise naturally when dealing with embedded vector fields due to the nonconformity of the tangent spaces.
%We have generalized the definition of the norms to subdomains $\gh$ of the surfaces, since this is necessary to introduce a patchwise inf-sup stability, which follows in \Cref{sec:inf_sups}.

The (weighted) norms are related to each other. The following lemma summarizes some results shown in \cite{HP2022Tangential}.
\begin{lemma}[Relation between norms]\label{lem:relation-between-norms}
  For $\bm{u}_h\in\VelocitySpace[h]$ with $h$ small enough, we have
  \begin{align}
    \AhNorm{\bm{u}_h}&\leqC \GNorm{\bm{u}_h} \leqC \NormOne{\bm{u}_h} \leqC \ahNorm{\bm{u}_h} + h^{-\frac{1}{2}}\EnergyNorm{\bm{u}_h}{s_h}\,.
  \end{align}
  For $\bm{u}\in\HSpace[1]{\G}$ we can estimate
  \begin{equation}
    \HNormTan{\PP\bm{u}}{\G} \leqC \HNorm[1]{\bm{u}}{\G} \leqC \GNorm{\bm{u}^e}\,.
  \end{equation}
  For $\bm{u}\in\HSpaceTan[1]{\G}$ we can estimate
  \begin{equation}
    \GNorm{\bm{u}^e} \leqC \HNormTan[1]{\bm{u}}{\G}\,.
  \end{equation}
\end{lemma}
\begin{proof}
  The estimates follow similar to \cite[Remark 2.12 and Lemma 4.7]{HP2022Tangential} with the $\GNorm{\cdot}$ defined slightly different, as the estimate in \cite[(2.15)]{HP2022Tangential}.
\end{proof}

% -----------------------------------------------------------------------------
%\subsection{Difference in the Discrete Variational Forms}\label{sec:differences_in_discrete_forms}
While in the continuous setting the two representations of the bilinear form $a(\bm{u},\bm{v})$ and the divergence term $b(\bm{u},p)$ are equivalent, this does not hold in the discrete setting for $a_1,a_2$ and $b_1,b_2$ due to the nonconformity of the discrete surface and thus the noncontinuity of the projected fields. This requires integration by parts at the element level introducing additional jump terms.

\begin{lemma}[Trace inequalities]\label{lem:trace-inequalities}
	Assume that $\El\in\Tri[\Gh]$ is a (curved) element of a shape-regular
	triangulation and $e\in\partial\El$ is a (curved) facet of the element. Let
	$1\leq m$, $1< p,q<\infty$, and $0\leq s\leq m-1/p$ be integers with $s-1/q\leq
	m-2/p$.
	%For any scalar-valued $v\in W^{m,p}(T)$, the function $\restr{v}{e}\in
	%W^{s,q}(e)$, and we have
	%\begin{align}
	%    \SobolevNorm[s][q]{v}{e} &\leqC h^{1/q} h^{-2/p} h^{-s}
	%(\SobolevNorm[0][p]{v}{T} + h^m\HalbNorm{v}{\SobolevSpace[m][p]{T}})\,.
	%\end{align}
	For any $\bm{v}\in[W^{m,p}(\El)]^3$, the trace
	$\restr{\bm{v}}{e}\in[W^{s,q}(e)]^3$, and we have
	\begin{align}\label{eq:trace-inequality}
		\SobolevNorm[s][q]{\bm{v}}{e} &\leqC h^{\frac{1}{q} - \frac{2}{p} - s}
		(\Norm{\bm{v}}{\LSpace[p]{\El}} + h^m\HalbNorm{\bm{v}}{\SobolevSpace[m][p]{\El}})\,.
	\end{align}
\end{lemma}
\begin{proof}
	For each scalar component of $\bm{v}=[v_i]$ the trace inequality holds, see
	\cite[Lemma 2.4]{Bernardi1989Optimal}. The extension to vector-valued fields
	follows by norm equivalence.
\end{proof}

\begin{lemma}[Difference in the divergence terms]\label{lem:difference-discrete-b}
	For $\bm{v}_h\in\HSpaceAmb[1]{\Gh}$ and $q_h\in H^1(\Gh)$ we can characterize the difference between the two bilinear forms $b_1$ and $b_2$ by
	\begin{align}
		\Abs{b_1(\bm{v}_h,q_h) - b_2(\bm{v}_h,q_h)} &\leqC h^{\kg}\left(h\Norm{\gradh
			q_h}{\Gh}+\Norm{q_h}{\Gh}\right)
		\GNorm{\bm{v}_h}\,.\label{eq:difference-discrete-b}
		%		\\
		%		&\quad\cdot
		%		\left(h^{\kg}\Norm{\D\PP\bm{v}_h}{\Gh}+h^{\kg-1}\Norm{\PP\bm{v}_h}{\Gh}
		%		+\Norm{\D\QQ\bm{v}_h}{\Gh}
		%		+h^{-1}\Norm{\QQ\bm{v}_h}{\Gh}\right)\,.\label{eq:difference-discrete-b}
	\end{align}
	% 	\Abs{b_1(\bm{v}_h,q_h) - b_2(\bm{v}_h,q_h)}
	% 	\leqC h^{\kg+1}\Norm{q_h}{h,1}
	% \left(h^{\kg}\Norm{\PP\bm{v}_h}{h,1} + \Norm{\QQ\bm{v}_h}{h,1}\right)
	%		\left(h^{\kg}\Norm{D\PP\bm{v}_h}{\Gh}+h^{\kg-1}\Norm{\PP\bm{v}_h}{\Gh}
	%		+\Norm{D\QQ\bm{v}_h}{\Gh}
	%		+h^{-1}\Norm{\QQ\bm{v}_h}{\Gh}\right).
	If additionally $\bm{v}_h\cdot \bm{n}=0$, we have
	\begin{align}
		\Abs{b_1(\bm{v}_h,q_h) - b_2(\bm{v}_h,q_h)} &\leqC
		h^{2\kg-1}\left(h\Norm{\gradh q_h}{\Gh}+\Norm{q_h}{\Gh}\right)
		\HNormTan{\bm{v}_h^{\ell}}{\G}\;.
	\end{align}
%%	If $v_h\in \VelocitySpace$, we have
%%	\begin{align}
%		\Abs{b_1(\bm{v}_h,q_h) - b_2(\bm{v}_h,q_h)} &\leqC h^{\kg}\left(h\Norm{\gradh
%			q_h}{\Gh}+\Norm{q_h}{\Gh}\right)
%		\NormZero[\Gh]{\bm{v}_h}\;.
%	\end{align}
	Further, we have for all constant maps $c_q$
	\begin{align}\label{eq:difference-discrete-b-constq}
		\Abs{b_1(\bm{v}_h,c_q)} &\leqC h^{\kg+1}|c_q|
		\GNorm{\bm{v}_h}\,.%
	\end{align}
%	and in particular for averages of maps $q_{h}\in L^2(\Gh)$
%	\begin{align}
%		\Abs{b_1(\bm{v}_h,\avg{q_h}{\Gh})} &\leqC h^{\kg+1}\Norm{q_h}{\Gh}
%		\left(\HNorm{\bm{v}_h}{\Gh}
%		+h^{-1}\Norm{\QQ\bm{v}_h}{\Gh}\right)\,.%\label{eq:difference-discrete-b}
%	\end{align}
	If additionally $\bm{v}_h\in \VelocitySpace$
	\begin{align}
		\Abs{b_1(\bm{v}_h,c_q)} &\leqC h^{\kg+1}|c_q|
		\NormZero[\Gh]{\bm{v}_h}\,.%\label{eq:difference-discrete-b}
	\end{align}
\end{lemma}
\begin{proof}
	Denote by $\TriH=\Tri[\Gh]$ the triangulation of $\Gh$ with set of (internal) edges $\mathcal{E}_h\colonequals\mathcal{E}(\TriH)$. Writing out the difference $b_1-b_2$ and employing elementwise integration by
	parts, we have
	\begin{align*}
		b_1(\bm{v}_h,q_h) - b_2(\bm{v}_h,q_h)
		& = (q_h,\divGh \PPh \bm{v}_h)_{\Gh} + (\gradGh q_h, \PPh \bm{v}_h)_{\Gh}
%		& = \sum_{e\in\mathcal{E}_{h}}(q_h, \Jump{\PPh \bm{v}_h\cdot \nu_e})_{e}
		= \sum_{e\in\mathcal{E}_{h}}(q_h, \bm{v}_h\cdot \Jump{\nu_e})_{e},
	\end{align*}
	where $\Jump{\nu_{e}} = \nu_{e,T^+} + \nu_{e,T^-}$ denotes the jump of the outer
	co-normals on the edge $e$ across the elements. It is bounded by
	$\Abs{\Jump{\nu_e}}\leqC h^{\kg}$. Additionally, we have the property
	$\Abs{\PP\Jump{\nu_e}}\leqC h^{2\kg}$, see, e.g., \cite[Lemma 3.5]{ORX2014Stabilized} and \cite{JankuhnEtAl2021Error}, and therefore
	\begin{align*}
		\Abs{\bm{v}_h\cdot \Jump{\nu_e}}
	%	&\leqC h^{2\kg} \Abs{\bm{v}_h} +h^{\kg} \Abs{\QQ\bm{v}_h}
		\leqC h^{2\kg} \Abs{\PP\bm{v}_h} + h^{\kg} \Abs{\QQ\bm{v}_h}.
	\end{align*}
%
%	For an edge $e\in\mathcal{E}_{h}$, we can choose $\PPh^{+}$ to be the
%	tangential projection of one of the adjacent elements and write
%	\begin{align*}
%		\bm{v}_h\cdot \Jump{\nu_e} & = \PPh^{+}\bm{v}_h\cdot \Jump{\nu_e} +
%		\QQh^{+}\bm{v}_h\cdot \Jump{\nu_e}\\
%		& =  (\PPh^{+}-\PP)\bm{v}_h\cdot \Jump{\nu_e} + (\QQh^{+}-\QQ)\bm{v}_h\cdot \Jump{\nu_e}
%		+ \bm{v}_h\cdot \PP \Jump{\nu_e} + \QQ\bm{v}_h\cdot \Jump{\nu_e}\\
%		\Rightarrow\ \Abs{\bm{v}_h\cdot \Jump{\nu_e}} &\leqC h^{2\kg} \Abs{\bm{v}_h} +
%		h^{\kg} \Abs{\QQ\bm{v}_h}
%		\leqC h^{2\kg} \Abs{\PP\bm{v}_h} + h^{\kg} \Abs{\QQ\bm{v}_h}
%	\end{align*}
	Using trace inequalities, see \Cref{lem:trace-inequalities} with $s=0$, $m=1$, $q=p=2$, we thus obtain
	\begin{align*}
		\Abs{b_1(\bm{v}_h,q_h) - b_2(\bm{v}_h,q_h)}
%		&\leqC h^{\kg} \left(h^{\frac{1}{2}}\Norm{\gradh
%			q_h}{\Gh}+h^{-\frac{1}{2}}\Norm{q_h}{\Gh}\right)\\
%		&\quad\cdot
%		\left(h^{\kg+\frac{1}{2}}\Norm{\D\PP\bm{v}_h}{\Gh}
%		+h^{\kg-\frac{1}{2}}\Norm{\PP\bm{v}_h}{\Gh}
%		+h^{\frac{1}{2}}\Norm{\D\QQ\bm{v}_h}{\Gh}
%		+h^{-\frac{1}{2}}\Norm{\QQ\bm{v}_h}{\Gh}\right)\\
		&\leqC h^{\kg} \left(h\Norm{\gradh q_h}{\Gh}+\Norm{q_h}{\Gh}\right)\\
		&\quad\cdot
		\left(h^{\kg}\Norm{\D\PP\bm{v}_h}{\Gh}+h^{\kg-1}\Norm{\PP\bm{v}_h}{\Gh}
		+\Norm{\D\QQ\bm{v}_h}{\Gh}
		+h^{-1}\Norm{\QQ\bm{v}_h}{\Gh}\right),
	\end{align*}
	and hence the first two estimates.
	For constant maps $c_q$, we use the trace inequality
	\Cref{lem:trace-inequalities} with $s=0$, $m=1$, $q=1$, and $p=2$, and obtain
	\begin{align*}
		\Abs{b_1(\bm{v}_h,c_q)-b_2(\bm{v}_h,c_q)} &\leqC h^{\kg}|c_q|
		\left(h^{\kg+1}\Norm{\D\PP\bm{v}_h}{\Gh}+h^{\kg}\Norm{\PP\bm{v}_h}{\Gh}
		+h\Norm{\D\QQ\bm{v}_h}{\Gh}
		+\Norm{\QQ\bm{v}_h}{\Gh}\right).
	\end{align*}
	Using inverse estimates, if $v_h\in \VelocitySpace$ this yields the last two estimates.
\end{proof}

\begin{remark}\label{rem:difference-discrete-bilinear-forms}
	We can do a similar calculation using integration by parts and the trace theorem to characterize the difference between the two bilinear forms $a_1$ and $a_2$ for $\bm{v}_h\in\HSpaceAmb[1]{\Gh}$ and $\bm{w}_h\in\HSpaceAmb[2]{\Gh}$ by
	\begin{multline}\label{eq:difference-discrete-bilinear-forms-1}
		a_2(\bm{v}_h, \bm{w}_h) - a_1(\bm{v}_h, \bm{w}_h)\\
		 = \frac{1}{2}\Inner{(\Kh-\KhTilde)\PPh \bm{v}_h}{\PPh \bm{w}_h}{\Gh}
		- \frac{1}{2}\Inner{\divGh \PPh \bm{v}_h}{\divGh \PPh\bm{w}_h}{\Gh} + r(\bm{v}_h,\bm{w}_h)
	\end{multline}
	with
	$
		\Abs{r(\bm{v}_h,\bm{w}_h)} \leqC h^{\kg-1} \left(h \HalbNorm{\bm{v}_h}{\HSpace{\Gh}}+\Norm{\bm{v}_h}{\Gh} \right)\left(h\HalbNorm{\bm{w}_h}{\HSpace[2]{\Gh}}+\Norm{\bm{w}_h}{\HSpace{\Gh}} +
	h^{-1}\Norm{\QQh\bm{\bm{w}_h}}{\Gh} \right).
	$
	Using the continuous bilinear form, better estimates will be derived in \Cref{sec:geometric_errors}.
\end{remark}

%% file: 03_inf_sup.tex
% ============================================================================
\section{Analysis of Discrete Inf-Sup Stability}\label{sec:inf_sups}

Inf-sup stability is a crucial ingredient both for the well-posedness of the problem and for a priori error estimates. In this section, we will derive inf-sup conditions for the discrete problem for all pairs of spaces $(\VelocitySpace,\PressureSpace)$ that satisfy some basic properties including a localized stability property.

We extend the classical condition of Verfürth \cite{Verfurth1984Error} to the discrete surface $\Gh$ for both bilinear forms $b_1$ and $b_2$. For this we use the velocity norms $\NormOne{\cdot}$ that contain a weighted normal component.

\begin{definition}[Verfürth V-stability]\label{def:V-stability} Let $j=1,2$, and assume that $\PressureSpace\subset H^{j-1}(\Gh)$. The pair of spaces $(\VelocitySpace,\PressureSpace)$ is called \textbf{V-stable}, if there exists a constant $\beta_2 > 0$ independent of $h$ such that
	\begin{align}\label{eq:discrete-inf-sup-V}
	\sup_{0\neq\bm{v}\in\VelocitySpace}
	\frac{b_{j}\big(\bm{v},q)}{\NormOne{\bm{v}}} &\geq \beta_2 \Norm{q}{\Gh}\,, \qquad \forall q\in \PressureSpace\cap \LSpaceAvg{\Gh}\,.
	\end{align}
\end{definition}

We first note that V-stability for $b_1$, implies it for $b_2$, as long as $h$ is small enough.
Therefore, we restrict the remainder of this section to the analysis of $b_1$.

\begin{lemma}\label{lem:stability-bj}
	If \eqref{eq:discrete-inf-sup-V} holds for $j=1$, then \eqref{eq:discrete-inf-sup-V} holds for $j=2$.
\end{lemma}
\begin{proof}
	This follows from the estimate \Cref{lem:difference-discrete-b} for $H^1$-functions and inverse estimates.
\end{proof}

We construct the inf-sup condition patchwise, following the macroelement technique of Stenberg and Nicolaides \cite{Stenberg1984Analysis,BolandNicolaides1983Stability,Stenberg1990Technique}. The general idea is that the discrete surface $\Gh$ can be covered by patches built from unions of elements in such a way that each element patch $\gh$ can be lifted from a flat element patch $\bar{\g}$ of similar shape, see \Cref{def:macroelements}. On these flat patches we assume that a discrete inf-sup condition holds, which is valid for a large class of finite elements, see \Cref{sec:stability-assumptions}. We then lift these local estimates to the discrete surface and show V-stability, see \Cref{sec:lifted-inf-sup-stability}.

% -----------------------------------------------------------------------------
\subsection{Macroelement partitioning}\label{def:macroelements}
%Let us introduce the partitioning of the discrete surfaces into macroelements following some ideas of \cite{Stenberg1990Technique}. These macroelements will allow us to formulate localized conditions on our function spaces that must be fulfilled in order to establish V-stability on the whole surface.

We introduce macroelements as patches of elements in various surface parametrizations, see also \Cref{fig:macroelements} for a visualization.
%Thus, we consider a sequence of consecutive definitions starting from patches on the piecewise flat surfaces $\hat{\G}_h$.

\begin{definition}[Piecewise flat macroelement]
	A patch of elements $\hat{\g}\subset\hat{\G}_h$ connected via edges of the triangulation $\Tri[\hat{\G}_h]$, with $\hat{\g}=\bigcup_{j=1}^m \hat{\El}_j$, for $\hat{\El}_j\in\Tri[\hat{\G}_h]$, is called a \textbf{piecewise flat macroelement}.
\end{definition}

\begin{definition}[Flattening map]
	We consider a piecewise flat macroelement $\hat{\g}$. On each element $\hat{\El}_{j}\in\Tri[\hat{\g}]$ we have an associated constant normal vector $\hat{\bm{n}}_j\perp\hat{\El}_{j}$. We select one element from the triangulation, $\hat{\El}\in\Tri[\hat{\g}]$, and define a local orthonormal tangent basis, denoted by $\{\hat{\bm{t}}^1, \hat{\bm{t}}^2\}$, and normal vector $\hat{\bm{n}}=\hat{\bm{t}}^1\times\hat{\bm{t}}^2$. If the macroelement patch and the grid size $h$ are small enough, we can introduce local bases $\{\hat{\bm{t}}_j^1, \hat{\bm{t}}_j^2\}$ in all elements by an orthogonal projection $\hat{\bm{t}}_j^i \colonequals (I - \hat{\bm{n}}_j\otimes\hat{\bm{n}}_j)\hat{\bm{t}}^i\equalscolon\hat{\PP}_j\hat{\bm{t}}^i$, for $i=1,2$. This defines the map $\bar{\pi}\colon\hat{\g}\to\R^2$ as a \textbf{flattening map} of the macroelement $\hat{\g}$,
	\begin{align}\label{eq:flattening-map}
		\restr{\bar{\pi}}{\hat{\El}_j}(\hat{x}) &\colonequals \transposed{\begin{pmatrix}
			\hat{\bm{t}}^1 & \hat{\bm{t}}^2
		\end{pmatrix}} \begin{pmatrix}
			\hat{\bm{t}}_j^1 & \hat{\bm{t}}_j^2
		\end{pmatrix}\cdot\hat{x}
		= \begin{pmatrix}
			\langle\hat{\bm{t}}^1, \hat{\PP}_j\hat{\bm{t}}^1\rangle & \langle\hat{\bm{t}}^1, \hat{\PP}_j\hat{\bm{t}}^2\rangle \\
			\langle\hat{\bm{t}}^2, \hat{\PP}_j\hat{\bm{t}}^1\rangle & \langle\hat{\bm{t}}^2, \hat{\PP}_j\hat{\bm{t}}^2\rangle
		\end{pmatrix}\cdot\hat{x}\,,
	\end{align}
	with local coordinates $\hat{x} \in \hat{\El}_j$ associated to the local tangent basis of the element.
\end{definition}

\begin{lemma}[Properties of the flattening map]\label{lem:prop-flattening}
	Let $\bar{\mu}\colonequals\operatorname{det}(\D_{\hat{x}}\bar{\pi})$ denote the integration element defined elementwise on $\hat{\g}$. There exists a constant $C>0$ independent of the grid size $h$ such that almost everywhere
	\begin{align}
		\|\bar{\pi} - \Id\|&\leq C h\,,\text{ and}\label{eq:property-flattening-map-length}\\
		\Abs{\bar{\mu} - 1} &\leq C h^2\,,\;\text{ for }h<h_0\text{ small enough.}\label{eq:property-flattening-map}
	\end{align}
\end{lemma}
\begin{proof}
	See \Cref{proof:prop-flattening}.
\end{proof}

While there are multiple flattening maps possible, depending on the selection of the tangent element $\hat{\El}\in\Tri[\hat{\g}]$, if $h$ is small enough and the curvature of $\G$ is bounded any of the elements is a valid choice. To make the definition unique, we could choose the element closest to the barycenter of the macroelement.

\begin{definition}[Flat macroelement]
	A \textbf{flat macroelement} $\bar{\g}\subset\R^2$ is defined as the union of the images of the elements of $\hat{\g}$ under a piecewise affine flattening mapping $\bar{\pi}\colon\hat{\g}\to\R^2$, i.e., $\bar{\g}=\bigcup_{j=1}^m \bar{\El}_j$ with $\bar{\El}_j=\bar{\pi}(\hat{\El}_j)$ and $\hat{\El}_j$ an element of $\hat{\g}$.
\end{definition}

\begin{lemma}[Properties of the flat macroelement]\label{lem:prop-flat-macroelement}
	For $h<h_0$ small enough, the flat macroelement $\bar{\g}$ is conforming, non-degenerate, and the ratio of outer to inner diameter $h_{\bar{\El}} / \rho_{\bar{\El}}$ is bounded for all elements $\bar{\El}\in\Tri[\bar{\g}]$ if these properties hold also for the associated piecewise flat macroelement $\hat{\g}$.
\end{lemma}
\begin{proof}
	Let $\hat{\El}\in\Tri[\hat{\g}]$. \Cref{lem:prop-flattening} implies that the flattening map conserves the lengths of all edges of $\hat{\El}$ with $O(h)$ and the area of $\hat{\El}$ with $O(h^2)$. As the outer and inner diameter depend on the area and edge length, they are also preserved with $O(h)$.
	Thus, the elements of $\bar{\g}$ are non-degenerate.
\end{proof}

\begin{definition}[Parametric macroelement]
	A \textbf{parametric macroelement} $\g_{h,\kg}\subset\G_{h,\kg}$ is defined as the union of the images of the elements of $\hat{\g}$ under $\pi_{h,\kg}$, i.e., $\g_{h,\kg}=\bigcup_{j=1}^m \pi_{h,\kg}(\hat{\El}_j)$ with $\hat{\El}_j$ the elements of $\hat{\g}$. The macroelement $\g_{h,\kg}$ can also be parametrized over the flat macroelement $\bar{\g}$ via the elementwise mapping $F_{h,\kg}\colon\bar{\g}\to\g_{h,\kg}$, $F_{h,\kg}=\pi_{h,\kg}\circ\bar{\pi}^{-1}$.

	Similar as above, we write $\gh\colonequals \g_{h,\kg}$ and $F_h\colonequals F_{h,\kg}$ for simplicity of notation.
\end{definition}

\begin{lemma}[Properties of the parametric macroelement]\label{lem:prop-parametric-macroelement}
	Let $h<h_0$ small enough and let $\bar{\mathcal{E}}\colonequals\mathcal{E}^\text{int}(\Tri[\bar{\g}])$ denote the set of inner edges of the triangulation of $\bar{\g}$. For the parametrization $F_h$ and the associated integration element $\restr{\mu_h}{\bar{\El}}(\bar{x})\colonequals \sqrt{\operatorname{det}(\transposed{\D_{\bar{x}} F_h(\bar{x})}\D_{\bar{x}} F_h (\bar{x}))}$ for $\bar{x}\in\bar{\El}, \bar{\El}\in\Tri[\bar{\g}]$, $\bar{e}\in\bar{\mathcal{E}}$ it holds,

	\begin{subequations}
		\noindent\centering
    \begin{minipage}{0.495\textwidth}
			\begin{align}
				\LInfNorm{\Jump{\D_{\bar{x}} F_h}}{\bar{e}} &\leq C h\,, \label{eq:DFh-jump}\\
				\Norm{\mu_h - 1}{L^\infty(\bar{\El})} &\leq C h^2\,,\text{ and } \label{eq:muh-1}
			\end{align}
		\end{minipage}
		\hfill
		\begin{minipage}{0.495\textwidth}
			\begin{align}
				\Norm{F_h}{C^\infty(\bar{\El})} &\leq C,\;\forall \bar{\El}\in\Tri[\bar{\g}]\,, \label{eq:Fh-Cinfty}\\
				\LInfNorm{\D_{\bar{x}}\mu_h}{\bar{\El}} &\leq C h\,, \label{eq:Dmuh}
			\end{align}
		\end{minipage}\bigskip
	\end{subequations}

	\noindent for some constants $C>0$ independent of $h$.
\end{lemma}
\begin{proof}
	See \Cref{proof:prop-parametric-macroelement}.
\end{proof}

\begin{figure}
  \begin{center}
		\begin{subfigure}{0.43\textwidth}
    	\includegraphics{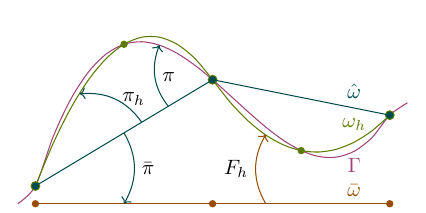}
			\caption{One-dimensional curve}
		\end{subfigure}
		\begin{subfigure}{0.56\textwidth}
			\includegraphics[width=\textwidth]{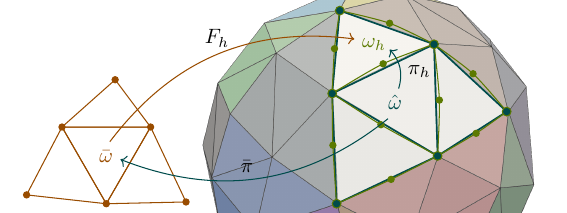}
			\caption{Two-dimensional sphere}
		\end{subfigure}
  \end{center}
  \caption{\label{fig:macroelements}Visualization of the various macroelements and the mappings between their elements from the piecewise flat surface $\hat{\g}$ denoted by $\pi_*$ and from the flat two-dimensional patch $\bar{\g}$ by $F_*$.}
\end{figure}

A collection of macroelements $\hat{\Omega}_h=\{\hat{\omega}\ldots\}$ that covers the surface $\hat{\G}_h$ with each element $\hat{\El}\in\Tri[\hat{\G}_h]$ contained in at least one and at most $L$ macroelements is called a \textbf{macroelement partitioning} of $\hat{\G}_h$. The associated collection $\Omega_h=\{\omega_h=\pi_h(\hat{\omega})\ldots\}$ is called a macroelement partitioning of $\Gh$.

Following \cite{Stenberg1984Analysis,Stenberg1990Technique} we call two (piecewise) flat macroelements equivalent if there exists a piecewise affine and continuous one-to-one mapping between them that preserves the triangulation. We call two parametric macroelements equivalent if they are associated to equivalent piecewise flat macroelements.
For the sequence of macroelement partitionings $\{\hat{\Omega}_h\}_h$ of the corresponding sequence $\{\hat{\G}_h\}_h$, which we consider in the following, we assume that the total number of equivalence classes is finite.

%------------------------------------------------------------------------------
\subsection{Stability Assumptions}\label{sec:stability-assumptions}
In this section we formulate assumptions about the discrete pairs $(\VelocitySpace,\PressureSpace)$ that will lead to the definition of stable pairs at the end of this section. These assumptions are given by the properties of the associated spaces on flattened macroelement patches. Therefore, we denote by $\hat{\Omega}_h=\{\hat{\g}\ldots\}$ a collection of macroelements covering $\hat{\G}_h$ with associated collection of parametric macroelements $\Omega_h=\{\pi_h(\hat{\g})\ldots\}$ and flattened macroelements $\bar{\Omega}_h=\{\bar{\pi}(\hat{\g})\ldots\}$.

Similar to \Cref{sec:sfem}, we denote by $\bar{S}$ the lifting of the finite element space $\hat{S}$ to a flat macroelement $\bar{\g}\in\bar{\Omega}_h$ with the piecewise affine flattening map $\bar{\pi}\colon\hat{\g}\to\bar{\g}$, with $\hat{\g}\in\hat{\Omega}_h$,
\begin{align}\label{eq:flat-broken-finite-element-space}
	\bar{S}^{\Broken}(\Tri[\bar{\g}];\FSpace)
		&\colonequals\Set{\bar{v}\in L^\infty(\Tri[\bar{\g}],\R)\mid \bar{v}\circ\bar{\pi} \in\hat{S}^{\Broken}(\Tri[\hat{\g}];\FSpace) }\,,\\
		\bar{S}(\bar{\g};\FSpace)
			&\colonequals\Set{\bar{v}\in C(\bar{\g},\R)\mid \bar{v}\circ\bar{\pi} \in\hat{S}(\Tri[\hat{\g}];\FSpace) }\,,
\end{align}
with the abbreviations $\bar{S}^{\Broken}_r(\Tri[\bar{\g}])=\bar{S}^{\Broken}(\Tri[\bar{\g}];\FSpace_r)$ and $\bar{S}_r(\bar{\g})=\bar{S}(\bar{\g};\FSpace_r)$ as before.

\begin{property}[Local super-approximation property]\label{prop:local-super-approximation}
	For a flat domain $\bar{\g}\subset\R^2$ with triangulation $\TriBar=\Tri[\bar{\g}]$ let the continuous $\bar{v}\in \bar{S}(\bar{\g};\FSpace)$ and broken $\phi\in H^1(\TriBar)$ with $\restr{\phi}{\bar{\El}}\in C^\infty(\bar{\El}), \bar{\El}\in\TriBar$ be given. Then there exists $\bar{w}\in\bar{S}^\Broken(\TriBar;\FSpace)$ such that
	\begin{align}
		\Norm{\phi\bar{v} - \bar{w}}{H^m(\bar{\El})} &\leq C h \Norm{\bar{v}}{H^m(\bar{\El})}\,,\quad\forall\bar{\El}\in\TriBar,\; m=0,1\,,
	\end{align}
  with $C=C(\Norm{\phi}{C^\infty(\bar{\El})})$. If $\bar{v}$ is zero along outer faces of $\bar{\g}$, so is $\bar{w}$.
\end{property}

This local super-approximation property is known to hold for many finite element spaces, cf. \cite{NS1974Interior,AX1995Local}. See also \cite[Assumption 7.1, 9.1]{Wah1991Local} for a discussion of this property. It is formulated here for a generic function $\phi$, but will later be used with components of the Jacobian of the parametrization $\D F_h$.

\begin{example}\label{rem:prop-local-super-approximation-example1}
	For continuous Lagrange functions $\bar{v}\in\bar{S}_r(\bar{\g})$ \Cref{prop:local-super-approximation} is satisfied by elementwise Lagrange interpolation, $\restr{\bar{w}}{\bar{\El}}=\bar{\I}_{h,r}[\restr{\phi\bar{v}}{\bar{\El}}]$, even for only piecewise smooth functions $\phi$.
\end{example}

\begin{example}\label{rem:prop-local-super-approximation-example2}
	For $\bar{v} = \bar{v}^{(1)}+\bar{v}^{(b)}\in \bar{S}(\bar{\g};\PolynomialSpace{1}\oplus\mathbb{B}_3)$, with $\mathbb{B}_3$ the element-bubble functions of order 3, \Cref{prop:local-super-approximation} is satisfied with $\restr{\bar{w}}{\bar{\El}} = \bar{\I}_{h,1}[\phi\bar{v}^{(1)}] + \avg{\phi}{\bar{\El}}v^{(b)}$. See also \cite{AX1995Local}.
\end{example}

The following property is the basic assumption we make to ensure that the mixed space $(\VelocitySpace,\PressureSpace)$ satisfies a V-stability condition on $\Gh$.
The property requires that when we consider individual patches $\gh$ of the surface $\Gh$, an inf-sup stability condition holds on their associated flat patches $\bar{\g}$, which can be verified by classical techniques.

\begin{property}[Stability property]\label{prop:stable-reference-spaces}
	There exists a sequence of macroelement partitionings $\{\Omega_h\}_{h\leq h_0}$ covering the corresponding discrete surfaces $\{\Gh\}_{h\leq h_0}$. For every $\gh\in\Omega_h$ there exists an associated flat macroelement $\bar{\g}$ with $\operatorname{diam}(\bar{\g})\leqC h$. The spaces $(\bar{V}(\bar{\g}),\bar{Q}(\bar{\g}))$ associated to the localized pair $(\restr{V_h}{\gh},\restr{Q_h}{\gh})$ through the lifting $F_h$, i.e.
	\begin{align*}
		\bar{V}(\bar{\g}) &= \Set{\bar{v}\in H^1(\bar{\g}) \mid \bar{v}\circ F_h^{-1}\in \restr{V_h}{\gh}} = \bar{S}(\bar{\g};\FSpace_V)\,, \\
		\bar{Q}(\bar{\g}) &= \Set{\bar{q}\in L^2(\bar{\g}) \mid \bar{q}\circ F_h^{-1}\in \restr{Q_h}{\gh}} = \bar{S}^{(\Broken)}(\Tri[\bar{\g}];\FSpace_Q)\,,
	\end{align*}
	satisfy the following stability property:

	For all $\bar{q}\in \bar{Q}(\bar{\g})\cap \LSpaceAvg[2]{\bar{\g}}$ there exists $\bar{\bm{v}}\in [\bar{V}(\bar{\g})]^2$ with $\restr{\bm{\bar{v}}}{\partial\bar{\g}}=\bm{0}$, such that
	\begin{align}\label{eq:inf-sup-flat-macroelement2}
		\Inner{\nabla\cdot\bar{\bm{v}}}{\bar{q}}{\bar{\g}}
		\geq \bar{\beta}' \Norm{\nabla\bar{\bm{v}}}{\bar{\g}} \Norm{\bar{q}}{\bar{\g}}
	\end{align}
  with $\bar{\beta}'$ independent of the macroelement $\bar{\g}$ and the gridsize $h$.
\end{property}

The stability condition \eqref{eq:inf-sup-flat-macroelement2} is known to hold for several pairs of finite element spaces if the flat macroelements contain enough elements and belong to a finite collection of equivalence classes. This is because, if $\bar{\omega}$ is large enough, there exist stability constants $\beta'_{\bar{\g}}$ for each macroelement $\bar{\g}$ such that \eqref{eq:inf-sup-flat-macroelement2} holds with these constants.
Since the elements belong to a shape-regular family of triangulations, the constants can be replaced by a uniform global constant $\bar{\beta}'$ that is independent of the specific macroelement and grid size \cite[Lemma 3.4]{Stenberg1990Technique}. The proof uses the fact that the inf-sup-constant is independent of the scaling of the elements, and that the rescaled macroelements of a single equivalence class form a compact set on which the inf-sup-constant is a continuous function and as such takes its minimum.

\begin{example}\label{rem:example-global-spaces}
	The pair $(\bar{S}_k(\bar{\g}),\bar{S}_{k-1}(\bar{\g}))$ of Lagrange spaces of order $k$ and $k-1$, called the Taylor-Hood element \cite{TH1973numerical}, is an example of a pair of spaces that satisfies \Cref{prop:stable-reference-spaces} for $k\geq 2$ if all $\bar{\g}$ contain at least three triangles, see. \cite{BrezziFalk1991Stability,Bof1994Stability}.
	Another example, $(\bar{S}(\bar{\g};\FSpace_1\oplus\mathbb{B}_3),\bar{S}_{1}(\bar{\g}))$, is called the MINI element, see \cite{ArnoldEtAl1984Stable}.
\end{example}

We combine the \Cref{prop:interpolation,prop:inverse-estimates} for general discrete spaces and the \Cref{prop:local-super-approximation,prop:stable-reference-spaces} into one definition.

\begin{definition}[Stable pair of discrete spaces]\label{def:stable-mixed-spaces}
	Let $(\VelocitySpace,\PressureSpace)$ denote a pair of discrete spaces as in \Cref{def:mixed-spaces} that satisfies \Cref{prop:interpolation,prop:inverse-estimates}. If, in addition, \Cref{prop:local-super-approximation,prop:stable-reference-spaces} are satisfied, we call this pair a \textbf{stable pair} of discrete spaces.
\end{definition}

% -----------------------------------------------------------------------------
\subsection{Lifted Inf-Sup Stability Analysis}\label{sec:lifted-inf-sup-stability}
% TODO: Einleitung küurzen
%Since we do not have a global smooth parametrization of the surface $\G$, we construct the inf-sup condition patchwise, following the macroelement technique of Stenberg and Nicolaides \cite{Stenberg1984Analysis,BolandNicolaides1983Stability,Stenberg1990Technique}. The general idea is that the discrete surface $\Gh$ can be covered by patches built from unions of elements in such a way that each element patch $\gh$ can be lifted from a flat element patch $\bar{\g}$ of similar shape, cf. \Cref{def:macroelements}. On these flat patches we assume that a discrete inf-sup condition holds, compare \Cref{prop:stable-reference-spaces}, i.e., for all local pressure fields there exists a two-dimensional continuous discrete vector field, which satisfies \eqref{eq:inf-sup-flat-macroelement2}.

In order to lift the stability condition \eqref{eq:inf-sup-flat-macroelement2}, we construct a continuous discrete vector field in $\VelocitySpace$ from a discrete flat vector field in $\big[\bar{V}(\bar{\g})\big]^2$. This is done in two steps. At first we lift the field into the tangent spaces of the elements of the discrete surface $\Gh$. The result is discontinuous but tangential. Then we apply a local averaging operator. The following lemma will verify that this leads to a suitable approximation, i.e., this construction has super-approximating qualities.

% In order to relate discrete vector fields in $\VelocitySpace$ with discrete maps in the associate flat discrete space $\bar{V}(\bar{\g})$, we will first show that interpolation of tangential vector fields with discrete coefficients
% TODO: Grunsätzliche Idee wie wir die Vektorfelder konstruieren in Textform motivieren

\begin{lemma}[Super-approximation]\label{lem:super-approximation}
  Let $\bar{\g}\subset\R^2$ be a flat macroelement with $\operatorname{diam}(\bar{\g})\leqC h$, $\gh=F_h(\bar{\g})\subset\Gh$ a parametric macroelement, $\bar{V}(\bar{\g})$ a discrete continuous space with \Cref{prop:local-super-approximation}, and let $\bar{\bm{u}}_h\in\big[\bar{V}(\bar{\g})\big]^2$ be a flat discrete (continuous) vector field with $\restr{\bar{\bm{u}}_h}{\partial\bar{\g}}=\bm{0}$.
  We set $\bm{u}\colon\bar{\g}\to\R^3$, $\bm{u}(\bar{x}) = \D F_h(\bar{x})\bar{\bm{u}}_h(\bar{x})$. Then there exists a continuous $\bm{u}_I\in\big[\bar{V}(\bar{\g})\big]^3$ with $\restr{\bm{u}_I}{\partial\bar{\g}}=\bm{0}$ such that
  \begin{align}
		\HNorm[m]{\bm{u}_I - \bm{u}}{\Tri[\bar{\g}],\R^3}
			&\leqC h \HNorm[m]{\bar{\bm{u}}_h}{\bar{\g},\R^2}\,,\text{ for }m=0,1 \label{eq:super-approx} \\
			\text{ and }\quad
			\NormOne[\gh]{\bm{u}_I\circ F_h^{-1}}
			&\leqC \HNorm{\bar{\bm{u}}_h}{\bar{\g},\R^2}\,. \label{eq:super-approx2}
  \end{align}
\end{lemma}
\begin{proof}
  Let $\phi\colonequals \D_i F_h^{j}$. It holds $\Norm{\phi}{C^\infty(\bar{\El})}\leq C$ with $C$ independent of $h$ on all $\bar{\El}\in\Tri[\bar{\g}]$, see \Cref{lem:prop-parametric-macroelement}. By \Cref{prop:local-super-approximation} with $\bar{v}\colonequals\bar{u}_h^i$ there exist $\bar{w}^{ij}\in\bar{S}^\Broken_h(\Tri[\bar{\g}];\FSpace_V)$ with $\restr{\bar{w}^{ij}}{\partial\bar{\g}}=0$ such that
	\begin{align}
		\HNorm[m]{\D_i F_h^{j}\bar{u}_h^i - \bar{w}^{ij}}{\bar{\El},\R}
			&\leqC h\HNorm[m]{\bar{u}_h^i}{\bar{\El},\R}\,. \label{eq:componentwise-property}
	\end{align}
	We construct the field $\bm{w}\in\big[\bar{S}^\Broken(\Tri[\bar{\g}];\FSpace_V)\big]^3$, with $w^j = \bar{w}^{1j} + \bar{w}^{2j}$, and introduce the smoothing $\bm{u}_I=\mathcal{J}^\text{av}_{0,\FSpace_V}[\bm{w}]$ using the averaging operator from \cite{ErnGuermond2017Quasi} that preserves zero boundary traces.
	By inserting $\bm{w}$ into \eqref{eq:super-approx} we estimate by triangle inequality
	\begin{align*}
		\HNorm[m]{\bm{u}_I - \bm{u}}{\bar{\El},\R^3}
			&\leq \HNorm[m]{\bm{w} - \bm{u}}{\bar{\El},\R^3} + \HNorm[m]{\bm{u}_I - \bm{w}}{\bar{\El},\R^3}\,.
	\end{align*}
	The estimate of the first term is given by summing up the componentwise estimate \eqref{eq:componentwise-property}, i.e.,
	\begin{align*}
		\HNorm[m]{\bm{w} - \D F_h \bar{\bm{u}}_h}{\bar{\El},\R^3} &\leqC h \HNorm[m]{\bar{\bm{u}}_h}{\bar{\El},\R^2}\,.
	\end{align*}
	For the second term we get by \cite[Lemma 4.2]{ErnGuermond2017Quasi}, the trace theorem and the local super-approximation \Cref{prop:local-super-approximation},
	\begin{align*}
		\HNorm[m]{\bm{u}_I - \bm{w}}{\bar{\El},\R^3}
			&\leqC h^{1/2 - m}\sum_{e\in\partial\bar{\El}} \Norm{\Jump{\bm{w}}}{e} % Lemma 4.2 Ern-Guermond 2017
			 \leq  h^{1/2 - m}\sum_{e\in\partial\bar{\El}} (\Norm{\Jump{\bm{w} - \bm{u}}}{e} + \Norm{\Jump{\bm{u}}}{e}) \\ % triangle
			&\leq  h^{1/2 - m}\sum_{e\in\partial\bar{\El}} (\Norm{\Jump{\bm{w} - \bm{u}}}{e} + \Norm{\Jump{\D F_h}}{L^\infty(e)}\Norm{\bar{\bm{u}}_h}{e}) \\ % Hölder
			&\leqC h^{-m} (\Norm{\bm{w} - \bm{u}}{\bar{\g}} + \Norm{\Jump{\D F_h}}{L^\infty(\mathcal{E}^\text{int}_{\bar{\g}})}\Norm{\bar{\bm{u}}_h}{\bar{\g}}) \\
			&\quad + h^{1 - m} (\HNorm{\bm{w} - \bm{u}}{\bar{\g}} + \Norm{\Jump{\D F_h}}{L^\infty(\bar{\mathcal{E}})}\HNorm{\bar{\bm{u}}_h}{\bar{\g}}) \\ % trace-theorem
			&\leqC h^{1-m} (1 + h^{-1}\Norm{\Jump{\D F_h}}{L^\infty(\bar{\mathcal{E}})})(\Norm{\bar{\bm{u}}_h}{\bar{\g}}+h \HNorm{\bar{\bm{u}}_h}{\bar{\g}})\,, % local super-approximation
	\end{align*}
	where $\bar{\mathcal{E}}\colonequals\mathcal{E}^\text{int}(\Tri[\bar{\g}])$ represents the set of inner edges of the triangulation of $\bar{\g}$. We have $\Norm{\Jump{\D F_h}}{L^\infty(\bar{\mathcal{E}})}\leqC h$, see \Cref{lem:prop-parametric-macroelement}, and since $\bar{\bm{u}}_h\in\bm{H}^1_0(\bar{\g})$ we have by Poincar\'e inequality \cite[(7.44)]{GilbargTrudinger1977Elliptic} with $\operatorname{diam}(\bar{\g})\leqC h$ the estimate $\Norm{\bar{\bm{u}}_h}{\bar{\g}} \leqC h \Norm{\D\bar{\bm{u}}_h}{\bar{\g}}$. So we get
	\begin{align*}
		\HNorm[m]{\bm{u}_I - \bm{w}}{\bar{\El},\R^3}
		&\leqC h^{2-m} \HNorm{\bar{\bm{u}}_h}{\bar{\g}}\,.
	\end{align*}
	For $m=0$ we get the assertion by inverse inequality.

	For inequality \eqref{eq:super-approx2} we set $\bm{u}_h\colonequals \bm{u}_I\circ F_h^{-1}$. Note that as  $\QQh\bm{u}=\bm{0}$ we have on $\El=F_h(\bar{\El})$
	\begin{align*}
		\NormOne[\El]{\bm{u}_h}
		& \leqC \Norm{\GradGh \bm{u}_h}{\El} + \Norm{\bm{u}_h}{\El} + (h^{-1}+\LInfNorm{\Weingarten_h}{\El})\Norm{\QQh\bm{u}_h}{\El}\\
		& \leqC \HNorm[1]{\bm{u}_I}{\bar{\El}} + h^{-1}\Norm{\QQh \bm{u}_I}{\bar{\El}}\\
		&\leqC \HNorm[1]{\bm{u}}{\bar{\El}}+\HNorm[1]{\bm{u}_I-\bm{u}}{\bar{\El}} + h^{-1}\Norm{\QQh(\bm{u}_I-\bm{u})}{\bar{\El}}\\
		&\leqC \SobolevNorm[1][\infty]{\D F_h}{\bar{\El}} \HNorm[1]{\bar{\bm{u}}_h}{\bar{\El}}+h \HNorm[1]{\bar{\bm{u}}_h}{\bar{\g}} + \Norm{\bar{\bm{u}}_h}{\bar{\g}}.
	\end{align*}
\end{proof}

When switching from $\g$ to $\gh$, we usually lose the normalization $\avg{p}{\g}=0$. The following lemma compensates for this.

\begin{lemma}[Norm of averaged pressure]\label{lem:averaged-norm}
  Let $\g$ be a macroelement with volume $|\g|\leq 1/2$. For $q\in L^2(\g)$ and $q-\avg{q}{\g}\in\LSpaceAvg[2]{\g}$ the $L^2$-norms are equivalent with $\LTwoNorm{q}{\g}\leqC 2\LTwoNorm{q-\avg{q}{\g}}{\g}\leqC 4 \LTwoNorm{q}{\g}$.
\end{lemma}
\begin{proof}
	This follows from $\Abs{\avg{q}{\g}}\leqC |\g|^{\frac{1}{2}}\Norm{q}{\g}$, and the triangle inequality.
%	\begin{align*}
%		\LTwoNorm{q - \avg{q}{\g}}{\g}
%		&\leq \LTwoNorm{q}{\g} + \LTwoNorm{\avg{q}{\g}}{\g}
%		\leq (1+|\g|) \LTwoNorm{q}{\g},
%		\intertext{and}
%		\LTwoNorm{q}{\g}
%		&\leq \LTwoNorm{q - \avg{q}{\g}}{\g} + \LTwoNorm{\avg{q}{\g}}{\g}
%		\leq \LTwoNorm{q - \avg{q}{\g}}{\g} + |\g|\LTwoNorm{q}{\g}.
%	\end{align*}
%	For $|\g|\leq \frac{1}{2}$, the equivalence follows.
\end{proof}

We  will now prove that we can lift a local stability on the flat macroelement to the parametrized macroelement.

\begin{lemma}[Lifted V-stability]\label{lem:lifted-V-stability}
	Let $(\VelocitySpace,\PressureSpace)$ be a stable pair of discrete spaces on $\Gh$ and $\gh$ a parametric macroelement associated to a flat macroelement $\bar{\g}$. We denote the restriction of the spaces to $\gh$ by
	%For all parametrized macroelements $\gh\subset\Gh$ associated to flat macroelements $\bar{\g}$ by $\gh=F_h(\bar{\g})$ with $\operatorname{diam}(\bar{\g})\leqC h$ let
	$(\VelocitySpace(\gh),\PressureSpace(\gh))\colonequals(\restr{\VelocitySpace}{\gh},\restr{\PressureSpace}{\gh})$. For all $q_h\in\PressureSpace(\gh)\cap \LSpaceAvg{\gh}$ there exists $\bm{v}_h\in\VelocitySpaceZero[h](\gh)$, such that
	\begin{align}\label{eq:inf-sup-parametric-macroelement2}
    \Inner{\divGh\PPh\bm{v}_h}{q_h}{\gh}
      \geq \beta'\NormOne[\gh]{\bm{v}_h}\Norm{q_h}{\gh}
	\end{align}
  for some $\beta' > 0$ dependent on $\bar{\beta}'$ of \Cref{prop:stable-reference-spaces} and $\max_{\bar{\El}\in\Tri[\bar{\g}]}\Norm{F_h}{C^\infty(\bar{\El})}$.
\end{lemma}

\begin{proof}
  Let $q_h\in\PressureSpace(\gh)\cap \LSpaceAvg{\gh}$, then $\bar{q}\colonequals q_h\circ F_h\in \bar{Q}(\bar{\g})$. For $\bar{p}\colonequals\bar{q}-\avg{\bar{q}}{\bar{\g}}$ there exists a continuous $\bar{\bm{v}}\in \big[\bar{V}(\bar{\g})\big]^2$ with $\restr{\bm{\bar{v}}}{\partial\bar{\g}}=\bm{0}$, where $\bar{V}$ is the associated flat componentspace to $\VelocitySpace$, and a constant $\bar{\beta}'>0$ independent of $h$ and the macroelement $\bar{\g}$, such that \eqref{eq:inf-sup-flat-macroelement2} holds. We introduce a discontinuous vector field $\bm{v}\colonequals\D F_h\,\bar{\bm{v}}$ that is tangential to the elements of $\gh$.
	By the super-approximation property \Cref{lem:super-approximation} there exists a continuous $\bm{v}_I\in\big[\bar{V}(\bar{\g})\big]^3$ with $\restr{\bm{v}_I}{\partial\bar{\g}}=\bm{0}$ such that \eqref{eq:super-approx} and \eqref{eq:super-approx2} hold. We introduce the lifting to $\gh$ via $\bm{v}_h\colonequals\bm{v}_I\circ F_h^{-1}\in \VelocitySpaceZero[h](\gh)$.
	We add and substract the tangential map $\bm{v}\circ F_h^{-1}$ on each element to the left hand side of \eqref{eq:inf-sup-parametric-macroelement2}, i.e., we write
	\begin{align*}
		\int_{\gh} q_h\divGh\PPh\bm{v}_h\,\text{d}\mu_h
		& = \sum_{\El\in\Tri[\gh]}	\left(\int_{\El} q_h\divGh\PPh(\bm{v}_h-\bm{v}\circ F_h^{-1})\,\text{d}\mu_h
		+ \int_{\El} q_h\divGh\bm{v}\circ F_h^{-1}\,\text{d}\mu_h\right).
	\end{align*}
	Using \Cref{lem:super-approximation} we obtain
	\begin{align*}
		\sum_{\El\in\Tri[\gh]}	\int_{\El} q_h\divGh\PPh(\bm{v}_h-\bm{v}\circ F_h^{-1})\,\text{d}\mu_h
%		& \leqC \Norm{q_h}{\gh}\sum_{\El\in\Tri[\gh]}	\HNormTan{\PPh\bm{v}_h-\bm{v}\circ F_h^{-1}}{\El}\\
		& \leqC \Norm{\bar{q}}{\bar{\g}} \HNorm{\bm{v}_I-\bm{v}}{\Tri[\bar{\g}]}
		 \leqC h \Norm{\bar{q}}{\bar{\g}} \HNorm{\bar{\bm{v}}}{\bar{\g}}.
	\end{align*}
	As $\bar{\bm{v}}$ are coordinates for $\bm{v}$ with respect to the parametrization $F_h$, we have
	\begin{align*}
		\int_{\El} q_h\divGh\bm{v}\circ F_h^{-1}\,\text{d}\mu_h
		& = \int_{\bar{\El}} \bar{q}\;\div(\mu_h\bar{\bm{v}})\,\text{d}\bar{x}
		= \int_{\bar{\El}} \bar{q}\;\mu_h\div\bar{\bm{v}}\,\text{d}\bar{x} + \int_{\bar{\El}} \bar{q}\;\D \mu_h\cdot \bar{\bm{v}}\,\text{d}\bar{x},
	\end{align*}
	where $\mu_h=\sqrt{\det \transposed{\D F_h}\D F_h}$.
	We have on each element $\Abs{\mu_h-1}\leqC h^{2}$ and $\Abs{\D\mu_h}\leqC h$, see \Cref{lem:prop-parametric-macroelement}.
	Thus, we obtain
	\begin{align*}
		\Abs{\int_{\gh} q_h \divGh \PPh\bm{v}_h\,\text{d}\mu_h - \int_{\bar{\g}} \bar{q}\,\div\bar{\bm{v}}\,\text{d}\bar{x}}
		& \leqC h \Norm{\bar{q}}{\bar{\g}}\HNorm{\bar{\bm{v}}}{\bar{\g}}.
%		\leqC h\;\Norm{\bar{q}}{\bar{\g}}\Norm{\nabla\bar{\bm{v}}}{\bar{\g}},
	\end{align*}
%	where the last inequality follows from $\bar{v}\in \VelocitySpaceFlatZero[h,\bar{\g}]$.
	Note that since $\restr{\bm{\bar{v}}}{\partial\bar{\g}}=\bm{0}$, by \Cref{prop:stable-reference-spaces}, using \Cref{lem:averaged-norm}, and the fact that for $\bar{\bm{v}}\in \VelocitySpaceFlatZero[h](\bar{\g})$ a Poincar\'e inequality for $h\leq h_0$ holds, we have that
	\begin{align*}
		\int_{\bar{\g}} \bar{q}\,\div\bar{\bm{v}}\,\text{d}\bar{x} & = \int_{\bar{\g}} \bar{p}\,\div\bar{\bm{v}}\,\text{d}\bar{x}
		\geq
		\bar{\beta}' \Norm{\bar{p}}{\bar{\g}}\Norm{\D\bar{\bm{v}}}{\bar{\g}}
		\geq \frac{\bar{\beta}'}{2}\Norm{\bar{q}}{\bar{\g}}\Norm{\D\bar{\bm{v}}}{\bar{\g}}
		\geq \frac{\bar{\beta}'}{4}\Norm{\bar{q}}{\bar{\g}}\HNorm{\bar{\bm{v}}}{\bar{\g}}
	\end{align*}
	follows. Thus for $h\leq h_0$ small enough we have the estimate
	\begin{align}\label{eq:lifted-V-stability-pf1}
		\int_{\gh} q_h \divGh \PPh\bm{v}_h\,\text{d}\mu_h
		& \geq \int_{\bar{\g}} \bar{p}\,\div\bar{\bm{v}}\,\text{d}\bar{x} - Ch\;\Norm{\bar{q}}{\bar{\g}}\HNorm{\bar{\bm{v}}}{\bar{\g}}
		\geq  \left(\frac{\bar{\beta}'}{4}-Ch\right)\Norm{\bar{q}}{\bar{\g}}\HNorm{\bar{\bm{v}}}{\bar{\g}}\notag \\
		& \geq \frac{\bar{\beta}'}{8}\Norm{\bar{q}}{\bar{\g}}\HNorm{\bar{\bm{v}}}{\bar{\g}}\;.
	\end{align}
 	By definition, we have for $q_{h}=\bar{q}\circ F_h^{-1}$ the estimate $\Norm{q_h}{\gh}\leqC \Norm{\bar{q}}{\bar{\g}}$, and for $\bm{v}_h$ by \Cref{lem:super-approximation} $\NormOne[\gh]{\bm{v}_h}\leqC \HNorm{\bar{\bm{v}}}{\bar{\g}}$.
\end{proof}

Once we have established local stability for all macroelements, we can show that a global stability holds for the whole grid. This follows the argumentation of \cite[Lemma 3.1]{Stenberg1990Technique} for possibly overlapping macroelements. We assume that there is a collection of macroelements $\Omega_h$ such that each element $\El\in\Tri[\Gh]$ is contained in at least one and in at most $L$ of these macroelements.

\begin{lemma}[From local to global stability]
	Let each $\Gh$ be covered by a macroelement partition $\Omega_h\colonequals\{\gh\subset\Gh\}$, and let $\Omega\colonequals \bigcup_{h\leq h_0}\Omega_h$.
	If a local stability estimate \eqref{eq:inf-sup-parametric-macroelement2} holds on all macroelements $\gh\in\Omega$, i.e., there exist $\beta'>0$ independent of $\gh$, such that
  \begin{equation}\label{eq:local-inf-sup}
    \sup_{0\neq \bm{v}_{\gh}\in\VelocitySpaceZero[h](\gh)}\frac{\Inner{\div_{\Gh}\PPh \bm{v}_{\gh}}{q_{\gh}}{\gh}}{\NormOne[\gh]{\bm{v}_{\gh}}}\geq \beta' \Norm{q_{\gh}}{\gh}\,,\quad \forall q_{\gh}\in \PressureSpace(\gh)\cap \LSpaceAvg{\gh}\,,
  \end{equation}
  then we obtain the global V-stability estimate \eqref{eq:discrete-inf-sup-V}.
%  : it exists $\beta>0$ independent of $h$, such that
%
%  \begin{equation}\label{eq:global-inf-sup}
%    \sup_{0\neq \bm{v}\in\VelocitySpace[h](\Gh)}\frac{\Inner{\div_{\Gh}\PPh \bm{v}}{q}{\Gh}}{\NormOne{\bm{v}}}\geq \beta \Norm{q}{\Gh}\,,\quad \forall q\in \PressureSpaceAvg[h](\Gh)\,.
%	\end{equation}
\end{lemma}

\begin{proof}
	Let $q\in\PressureSpace(\Gh)\cap \LSpaceAvg{\Gh}$ and $\gh\in\Omega_h$. We set $p_{\gh}\colonequals \restr{q}{\gh}$ and $q_{\gh}\colonequals p_{\gh} - \avg{q}{\gh}\in\PressureSpace(\gh)\cap \LSpaceAvg{\gh}$. By \eqref{eq:local-inf-sup} there exists $\bm{v}_{\gh}\in\VelocitySpaceZero[h](\gh)$, such that
	\[
	\Inner{\div_{\Gh}\PPh\bm{v}_{\gh}}{q_{\gh}}{\gh} \geq \beta' \Norm{q_{\gh}}{\gh}\NormOne[\gh]{\bm{v}_{\gh}}.
	\]
	We introduce the extended velocity fields $\bm{v}^0_{\gh}\in\VelocitySpace[h](\Gh)$ by $\restr{\bm{v}^0_{\gh}}{\gh}=\bm{v}_{\gh}$ and $\restr{\bm{v}^0_{\gh}}{\Gh\setminus\gh}=0$ since fields in $\VelocitySpaceZero[h](\gh)$ have boundary trace zero.
	Summing up the local velocities over all macroelements defines the global velocity field
	\[
	\bm{v}\colonequals\sum_{\gh\in\Omega_h}\bm{v}^0_{\gh}\in\VelocitySpace[h](\Gh)\,.
	\]
	We estimate the global norms of $\bm{v}$ and $q$ against the local ones by,
	\begin{align*}
		\NormOne{\bm{v}}\Norm{q}{\Gh}
		& \leq
		\sum_{\gh\in\Omega_h}\sum_{\gh'\cap\gh\neq\emptyset}\NormOne[\gh]{\bm{v}_{\gh}}\Norm{q}{\gh'}
		% \leq
		%L \sum_{\gh\in\Omega_h}\NormOne[\gh]{\bm{v}_{\gh}}\Norm{p_{\gh}}{\gh}
		 \leq
		2L \sum_{\gh\in\Omega_h}\NormOne[\gh]{\bm{v}_{\gh}}\Norm{q_{\gh}}{\gh},
	\end{align*}
	where we used \Cref{lem:averaged-norm} in the last step, and $L$ denotes the maximal number of macroelements containing an element $T\in\mathcal{T}(\Gh)$.
	For the divergence term, we estimate
	\begin{align*}
		\Inner{\div_{\Gh}\PPh\bm{v}}{q}{\Gh}
		&= \sum_{\gh\in\Omega_h} \Inner{\div_{\Gh}\PPh\bm{v}^0_{\gh}}{q}{\Gh}
		= \sum_{\gh\in\Omega_h} \Inner{\div_{\Gh}\PPh\bm{v}_{\gh}}{p_{\gh}}{\gh} \\
		&= \sum_{\gh\in\Omega_h} \Inner{\div_{\Gh}\PPh\bm{v}_{\gh}}{q_{\gh}}{\gh} + \Inner{\div_{\Gh}\PPh\bm{v}_{\gh}}{\avg{q}{\gh}}{\gh} \\
		&\geq \sum_{\gh\in\Omega_h} \beta'\Norm{q_{\gh}}{\gh}\NormOne[\gh]{\bm{v}_{\gh}} + \Inner{\div_{\Gh}\PPh\bm{v}_{\gh}}{\avg{q}{\gh}}{\gh}\,.
	\end{align*}
	The second term on the right-hand side can further be estimated by \eqref{eq:difference-discrete-b-constq} and \Cref{lem:averaged-norm} as
	\begin{align*}
		\Inner{\div_{\Gh}\PPh\bm{v}_{\gh}}{\avg{q}{\gh}}{\gh}
		&\leqC h^{\kg+1} \Abs{\avg{q}{\gh}}\NormOne[\gh]{\bm{v}_{\gh}}
		\leqC h^{\kg+1} \Norm{q}{\gh}\NormOne[\gh]{\bm{v}_{\gh}}\\
		&\leqC h^{\kg+1} \Norm{q_{\gh}}{\gh}\NormOne[\gh]{\bm{v}_{\gh}}.
	\end{align*}
	Setting $\beta_2\colonequals\frac{1}{2L}\beta'$, we obtain for $h\leq h_0$ small enough
	\begin{align*}
		\Inner{\div_{\Gh}\PPh\bm{v}}{q}{\Gh} \geq \beta' \sum_{\gh\in\Omega_h} \Norm{q_{\gh}}{\gh}\NormOne[\gh]{\bm{v}_{\gh}} \geq \beta_2 \NormOne{\bm{v}}\Norm{q}{\Gh}\,.
	\end{align*}
\end{proof}

%% file: 04_geometric_errors.tex
% ============================================================================
\section{Geometric Errors}\label{sec:geometric_errors}

In this section we will discuss the geometric errors that arise from discretizing the linear forms $a$ and $b$ over the discrete surface $\Gh$. Since this has already been done for the vector Laplace equation \cite{HLL2020Analysis, HP2022Tangential}, we will only briefly state the results for $a$ and concentrate on those for $b$. Throughout this section we denote by $(\VelocitySpace,\PressureSpace)$ a pair of discrete spaces that satisfies \Cref{prop:interpolation,prop:inverse-estimates}.

We begin by stating the following basic estimate for the difference between the tangential $H^1$-norms on $\G$ and $\Gh$.

\begin{lemma}[$H^1$-estimates]\label{lem:basicH1Estimate}
	Let $\bm{v}_h\in \HSpace[1]{\Gh}$. Then we have
	\begin{align*}
		\Norm{(\GradG\PP\bm{v}_h^\ell)^e - \GradGh\PPh\bm{v}_h}{\Gh} &\leqC h^{\kg}\GNorm{\bm{v}_h}\,.
	\end{align*}
\end{lemma}
\begin{proof}
	See, e.g., \cite{HLL2020Analysis,HP2022Tangential}.
\end{proof}

For $i,j\in\{1,2\}$, vector fields $\bm{v},\bm{w}\in\HSpaceAmb{\G}$, and functions $q \in \HSpaceAmb[j-1]{\G}$ we introduce the difference forms
\begin{align*}
	\delta_{a_i}(\bm{v},\bm{w})
    	&\colonequals a(\PP\bm{v},\PP\bm{w})-a_{i}(\PPh\bm{v}^{e},\PPh\bm{w}^{e}),\\
    \delta_{b_j}(\bm{v},q)
    	&\colonequals b(\PP\bm{v},q)-b_j(\PPh\bm{v},q),\\
    \delta_{l}(\bm{v})
    	&\colonequals \Inner{\bm{f}}{\bm{v}}{\G} - \Inner{\bm{f}^e}{\PPh\bm{v}^e}{\Gh}\,.
\end{align*}

Furthermore, we set
\begin{align}\label{eq:geom-error-delta}
	\delta_{ij}(\bm{v},q)(\bm{w},r) &\colonequals \delta_{l}(\bm{w}) - \delta_{a_i}(\bm{v},\bm{w}) - \delta_{b_j}(\bm{v},r) - \delta_{b_j}(\bm{w},q)\,.
\end{align}

\begin{lemma}\label{lem:geometricQlError}
	For $\bm{v}_h,\bm{w}_h\in \LSpaceAmb[2]{\Gh}$, we have
	\begin{align}
    \Abs{\delta_{l}(\bm{v}_h^\ell)}
        &\leqC h^{\kg+1}\Norm{\bm{f}}{\G}\NormZero{\bm{v}_h}\,.
	\end{align}
\end{lemma}
\begin{proof}
  See \cite[Lemma 4.8]{HP2022Tangential}
\end{proof}

\begin{lemma}\label{lem:geometric_a_errors}
	For $\bm{v}_h,\bm{w}_h\in\HSpaceAmb[1]{\Gh}$, we have
	\begin{align}
	\Abs{\delta_{a_1}(\bm{v}_h^\ell,\bm{w}_h^\ell)} &\leqC h^{\kg} \GNorm{\bm{v}_{h}}\GNorm{\bm{w}_{h}}\,,\\
	\Abs{\delta_{a_2}(\bm{v}_h^\ell,\bm{w}_h^\ell)} &\leqC h^{\kg} \GNorm{\bm{v}_{h}}\GNorm{\bm{w}_{h}}+h^{\kk}\ahNorm{\bm{v}_h}\ahNorm{\bm{w}_h}\,.
	\end{align}
	%
%	If $\bm{v}\in\HSpaceTan[1]{\G}$ and $\bm{w}_h\in\VelocitySpace$ and $h<h_0$ small enough, then we obtain
%	\begin{align}
%	\Abs{\delta_{a_i}(\bm{v},\bm{w}_h^\ell)} \leqC & h^{\kg}
%	\HNormTan{\bm{v}}{\G}\NormOne{\bm{w}_h}+(i-1)h^{\kk}\HNormTan{\bm{v}}{\G}\ahNorm{\bm{w}_h}.
%	% \left(\ahNorm{\bm{w}_{h}} + h^{-1}\Norm{\QQh\bm{w}_h}{\Gh}\right)\,.
%	\end{align}
	%
	For $\bm{v},\bm{w}\in \HSpaceTan[2]{\G}$ we have
	\begin{align}\label{eq:geometric_a_errors_3}
	\Abs{\delta_{a_1}(\bm{v},\bm{w})} \leqC & h^{\kg+1}	\HNormTan[2]{\bm{v}}{\G}\HNormTan[2]{\bm{w}}{\G}\,,\\
	\Abs{\delta_{a_2}(\bm{v},\bm{w})} \leqC & h^{\kg+1}	\HNormTan[2]{\bm{v}}{\G}\HNormTan[2]{\bm{w}}{\G}+h^{\kk}\HNormTan{\bm{v}}{\G}\HNormTan{\bm{w}}{\G}\,.
	\end{align}
\end{lemma}
\begin{proof}
	A corresponding error estimate for $\tilde{a}(\bm{v},\bm{w})\colonequals \Inner{\PP\bm{v}}{\PP{\bm{w}}}{\HSpaceTan{\G}}$ is proven in \cite[Lemma 4.12]{HP2022Tangential}. The same proof can also be employed for the symmetric gradient $\bm{E}(\PP{\bm{v}})$ instead of the full gradient. Finally, the Gauss curvature term can be split in two parts
	\begin{align*}
	\Abs{\Inner{K\PP\bm{v}_h^\ell}{\PP\bm{w}_h^\ell}{\G}
		-\Inner{K^{e}\PP_h\bm{v}_h}{\PP_h\bm{w}_h}{\Gh}}
	+\Abs{\Inner{(K^{e}-\KhTilde)\PPh\bm{v}_h}{\PPh\bm{w}_h}{\Gh}}\,.
	\end{align*}
	The first term behaves like the mass matrix term in \cite[Lemma 4.9]{HP2022Tangential} and can be estimated by
	\begin{align*}
	\Abs{\Inner{K\PP\bm{v}_h^\ell}{\PP\bm{w}_h^\ell}{\G}
		-\Inner{K^{e}\PP_h\bm{v}_h}{\PP_h\bm{w}_h}{\Gh}}&\leqC h^{\kg+1}\NormZero{\bm{v}_h}\NormZero{\bm{w}_h}.
	\end{align*}
	The second term is estimated by \eqref{eq:order-assumptions-K}, the requirement on the Gaussian curvature approximation.
\end{proof}

From geometric errors, we can revisit the difference between the bilinear forms in \Cref{rem:difference-discrete-bilinear-forms} and derive
\begin{corollary}
	For $\bm{v}_h,\bm{w}_h\in\VelocitySpace$ we have
	\begin{align}
		\Abs{a_1(\bm{v}_h,\bm{w}_h) - a_2(\bm{v}_h,\bm{w}_h)} &\leqC h^{\kg}\NormOne{\bm{v}_h}\NormOne{\bm{w}_h} + h^{\kk}\ahNorm{\bm{v}_h}\ahNorm{\bm{w}_h}\,.
	\end{align}
\end{corollary}

\begin{lemma}[The divergence form]\label{lem:geometricQb1Error}
	For $\bm{v}_h\in\HSpaceAmb[1]{\Gh}$ and $q_h\in L^2(\Gh)$ the following estimate holds for $h<h_0$ small enough:
	\begin{align}\label{eq:geometricQb1Error}
		\delta_{b_1}(\bm{v}_h^\ell,q_h^\ell) & \leqC h^{\kg}\GNorm{\bm{v}_h}\Norm{q_h}{\Gh}.
	\end{align}
\end{lemma}
\begin{proof}
	We write
	\begin{align*}
		|\delta_{b_1}(\bm{v}_h^\ell,q_h^\ell)|
		&\leqC  |\Inner{q_{h}}{(\divG\PP\bm{v}_h^{\ell})^{e}-\divGh\PPh\bm{v}_h}{\Gh}|+ h^{\kg+1}\Norm{q_{h}}{\Gh}\Norm{(\divG\PP\bm{v}_h^{\ell})^{e}}{\Gh}\\
		&\leqC  \Norm{q_{h}}{\Gh}\Norm{(\GradG\PP\bm{v}_h^{\ell})^{e}-\GradGh\PPh\bm{v}_h}{\Gh}+ h^{\kg+1}\Norm{q_{h}}{\Gh}\Norm{\GradG\PP\bm{v}_h^{\ell}}{\G}\\
		& \leqC h^{\kg}\GNorm{\bm{v}_h}\Norm{q_h}{\Gh},
	\end{align*}
	where we use \Cref{lem:basicH1Estimate}.
\end{proof}

\begin{lemma}[The alternative divergence form]\label{lem:geometricQbError}
	For $\bm{v}_h\in\LSpaceAmb[2]{\Gh}$ and $q_h\in H^{1}(\Gh)$ the following estimate holds for $h<h_0$ small enough:
\begin{align}\label{eq:geometricQbError}
	 \Abs{\delta_{b_2}(\bm{v}_h^\ell,q_h^\ell)} &\leqC h^{\kg+1}
   \NormZero{\bm{v}_h}
   %(\Norm{\PPh\bm{v}_h}{\Gh}+h^{-1}\Norm{\QQh\bm{v}_h}{\Gh})
   \Norm{\GradGh q_h}{\Gh}.
\end{align}
\end{lemma}
\begin{proof}
  We introduce $\bm{B}\colonequals \D_{\Gh}\pi$ with $\Norm{1-|\det\bm{B}|}{L^\infty(\Gh)}\leqC h^{\kg+1}$, cf. \cite[Lemma 4.1]{HLL2020Analysis}. By the triangle inequality, we get
	\begin{align*}
		\Abs{\delta_{b_2}(\bm{v}_h^\ell,q_h^\ell)}
  	    &= \Abs{\Inner{\GradGh q_h}{\PPh\bm{v}_h}{\Gh} - \Inner{\GradG q_{h}^{\ell}}{\PP\bm{v}_h^{\ell}}{\G}}\\
				&\leqC \Abs{\Inner{\GradGh q_h}{(\PPh-\PP)\bm{v}_h}{\Gh}}
				+ \Abs{\Inner{\GradGh q_h-(\GradG q_{h}^{\ell})^{e}}{\PP\bm{v}_h}{\Gh}}
				+  h^{\kg+1}\Norm{(\GradG q_{h}^{\ell})^{e}}{\Gh}\Norm{\PP\bm{v}_h}{\Gh}.
	\end{align*}
We further estimate using $\Norm{\PP\PPh-\bm{B}}{L^\infty(\Gh)}\leqC h^{\kg+1}$
\begin{align*}
	\Abs{\Inner{\GradGh q_h-(\GradG q_{h}^{\ell})^{e}}{\PP\bm{v}_h}{\Gh}}
		& = \Abs{\Inner{\PP(\PPh-B)\GradGh q_h}{\PP\bm{v}_h}{\Gh}}\\
		& \leqC h^{\kg+1} \Norm{\GradGh q_h}{\Gh}\Norm{\PP\bm{v}_h}{\Gh}.
\end{align*}
We have also
\begin{align*}
	\Abs{\Inner{\GradGh q_h}{(\PPh-\PP)\bm{v}_h}{\Gh}}
	& \leq \Abs{\Inner{\GradGh q_h}{\PPh(\PPh-\PP)\PPh\bm{v}_h}{\Gh}}
	+  \Abs{\Inner{\GradGh q_h}{\PPh(\PPh-\PP)\QQh\bm{v}_h}{\Gh}}\\
	& \leqC h^{2\kg} \Norm{\GradGh q_h}{\Gh}\Norm{\bm{v}_h}{\Gh}
	+  h^{\kg}\Norm{\GradGh q_h}{\Gh}\Norm{\QQh\bm{v}_h}{\Gh}.
\end{align*}
Using $\Norm{\PP\bm{v}_h}{\Gh}\leq \Norm{\PPh\bm{v}_h}{\Gh}+ \Norm{\QQh\bm{v}_h}{\Gh}$
and $\Norm{(\GradG q_{h}^{\ell})^{e}}{\Gh}\leqC \Norm{\GradGh q_h}{\Gh}$ we obtain the estimate.
%
% Fuer q_h in H^2 und v_h in H^1 bekommt man über non-standard geometry estimates auch h^{kg+1},
% wenn man den Term mit dem \QQh entsprechend abschätzt
\end{proof}

\begin{remark}
	If $\bm{v}_h\in\LSpaceAmb[2]{\Gh}$ and $q_h\in \PressureSpace\subset H^{1}(\Gh)$, inverse estimates \Cref{prop:inverse-estimates} combined with \eqref{eq:geometricQbError} yield
	\begin{align}\label{eq:geometricQbErrorDiscrete}
		\Abs{\delta_{b_2}(\bm{v}_h^\ell,q_h^\ell)} &\leqC h^{\kg}
		\NormZero{\bm{v}_h}
		% (\Norm{\PPh\bm{v}_h}{\Gh}+h^{-1}\Norm{\QQh\bm{v}_h}{\Gh})
		\Norm{q_h}{\Gh}\,.
	\end{align}
\end{remark}

We see that the estimate for $b_2$ is better than the one for $b_1$ if the function $q_h$ is smooth enough. We can use \Cref{lem:difference-discrete-b} to improve the estimate for $b_1$ for tangential velocities.
\begin{lemma}\label{lem:geometricQb1Error_tangential}
	Let $\bm{v}_h\in\HSpaceAmb[1]{\Gh}$ with $\bm{v}_h\cdot \bm{n}=0$ and $q_h\in H^{1}(\Gh)$. For $\kg>1$ we have the estimate
	\begin{align}
		\Abs{\delta_{b_1}(\bm{v}_h^\ell,q_h^\ell)} & \leqC h^{\kg+1}\HNormTan{\bm{v}_h^{\ell}}{\G}\HNorm{q_h}{\Gh}\,.
	\end{align}
	If additionally $\bm{v}_h\in\HSpaceAmb[2]{\Gh}$, we have for $\kg=1$
	\begin{align}
		\Abs{\delta_{b_1}(\bm{v}_h^\ell,q_h^\ell)} & \leqC h^{2}\HNormTan[2]{\bm{v}_h^{\ell}}{\G}\HNorm{q_h}{\Gh}\,.
	\end{align}
\end{lemma}

\begin{proof}
	For $\kg>1$ this follows from \Cref{lem:difference-discrete-b} and \Cref{lem:geometricQbError}. For $\kg=1$ we revisit the proof of \Cref{lem:geometricQb1Error}:
	It is a.e. on $\G$
	\begin{align*}
		\divG\PP\bm{v}_h^{\ell}-(\divGh\PPh\bm{v}_h)^{\ell}
		& = \tr ((\PP-\bm{B}^{-1})\GradG  \bm{v}_h^{\ell}) +H_h\langle \bm{v}_h^{\ell},\bm{n}_h\rangle-H \langle \bm{v}_h^{\ell} ,\bm{n}\rangle,
	\end{align*}
	where $H$ and $H_h$ denotes the mean curvature of $\G$ and $\Gh$, respectively.
	%	Nebenrechnung:
	%	\begin{align*}
	%		(\DivG\PP\bm{v})\circ \pi-\DivGh\PPh(\bm{v}\circ \pi)
	%		& = (\tr \PP D \bm{v})\circ \pi-(\tr \PP D \QQ\bm{v})\circ \pi-\tr \PPh D(\bm{v}\circ \pi) +\tr \PPh D\QQh(\bm{v}\circ \pi)\\
	%		& = (\tr \PP D \bm{v})\circ \pi-(\tr \PP D \langle \bm{v},\bm{n}\rangle \bm{n})\circ \pi-\tr \PPh D(\bm{v}\circ \pi) +\tr \PPh D\langle \bm{v}\circ \pi,\bm{n}_h\rangle \bm{n}_h\\
	%		& = (\tr \PP D \bm{v})\circ \pi-(H \langle \bm{v},\bm{n}\rangle )\circ \pi-\tr \PPh D(\bm{v}\circ \pi) +H_h\langle \bm{v}\circ \pi,\bm{n}_h\rangle\\
	%	\end{align*}
	For $\kg=1$ and $\bm{v}_h\cdot \bm{n}=0$, the last two terms vanish and we can write
	\begin{align*}
		\divG\PP\bm{v}_h^{\ell}-(\divGh\PPh\bm{v}_h)^{\ell}
		& = \tr ((\PP-\bm{B}^{-1})\GradG  \bm{v}_h^{\ell})
		= \tr (\QQh \GradG  \bm{v}_h^{\ell})+\tr ((\PPh\PP-\bm{B}^{-1})\GradG  \bm{v}_h^{\ell})
	\end{align*}
	As we have for $\bm{v}=\bm{v}_h^{\ell}$ that $\tr \QQh \GradG  \bm{v} = n_{h}^{i}n_{h}^{j}\P_{i}^{k}\P_{j}^{l}\D_{k}v_{l}$, we get the estimate,
	\begin{align*}
		\Abs{\Inner{q_{h}}{(\divG\PP\bm{v}_h^{\ell})-(\divGh\PPh\bm{v}_h)^{\ell}}{\G}}
		&\leqC \Abs{\Inner{q_{h}}{\tr (\QQh \GradG  \bm{v}_h^{\ell})}{\G}}
		+ \Abs{\Inner{q_{h}}{\tr ((\PPh\PP-\bm{B}^{-1})\GradG  \bm{v}_h^{\ell})}{\G}}\\
		&\leqC h^{2}\HNormTan[2]{\bm{v}_h^{\ell}}{\G}\HNorm{q_h}{\Gh}\,,
	\end{align*}
	see \cite[Lemma 2.6]{HP2022Tangential}.
\end{proof}

We can now state the combined error estimates for $\delta_{ij}$ as defined in \eqref{eq:geom-error-delta}.

\begin{lemma}\label{lem:basic_delta_estimate}
	For $\bm{v}_h,\bm{w}_h \in \HSpaceAmb{\Gh}$, and $q_h,r_h\in H^{j-1}(\Gh)$, we have for $h<h_0$ small enough
	\begin{align}\label{eq:basic_delta_estimate}
	\Abs{\delta_{ij}(\bm{v}_h^\ell,q_h^\ell)(\bm{w}_h^\ell,r_h^\ell)}
	&\leqC h^{\kg}\left(\GNorm{\bm{w}_h} + h^{j-1}\HNormTan[j-1]{r_h}{\Gh}\right)
	\left(\GNorm{\bm{v}_h} + h^{j-1}\HNormTan[j-1]{q_h}{\Gh} + h\Norm{\bm{f}}{\G}\right)\notag \\
	&\quad  + (i-1)h^{\kk}\ahNorm{\bm{w}_h}\ahNorm{\bm{v}_h} \,.
	\end{align}
	For $\bm{u} \in \HSpaceTan{\G}$, $p\in H^{1}(\G)$, and $\bm{v}_{h}\in\VelocitySpace, q_h\in\PressureSpace\subset H^{j-1}(\Gh)$, we have for $h<h_0$ small enough
	\begin{align}
	\Abs{\delta_{ij}(\bm{u},p)(\bm{v}_h^{\ell},q_h^\ell)}
	&\leqC h^{\kg}\left(\NormOne{\bm{v}_h}+ \Norm{q_h}{\Gh}\right)
	\left(\HNormTan[1]{\bm{u}}{\G} + \HNorm{p}{\G} + \Norm{\bm{f}}{\G}\right) \label{eq:discrete-galerkin-ortho_1} \\
	& \qquad +(i-1)h^{\kk}\ahNorm{\bm{v}_h}\HNormTan[1]{\bm{u}}{\G} \notag \\
	&\leqC \left(h^{\kg-\frac{1}{2}}+(i-1)h^{\kk}\right) \EnergyNorm{(\bm{v}_h,q_h)}{}
	\left(\HNormTan[1]{\bm{u}}{\G} + \HNorm{p}{\G} + \Norm{\bm{f}}{\G}\right)\,.\label{eq:discrete-galerkin-ortho_2}
	\end{align}
\end{lemma}
\begin{proof}
	The estimate \eqref{eq:basic_delta_estimate} follows from \Cref{lem:geometricQlError,lem:geometric_a_errors,lem:geometricQb1Error,lem:geometricQbError}.
	For \eqref{eq:discrete-galerkin-ortho_1}, we use additionally \Cref{lem:relation-between-norms}.
	The last estimate \eqref{eq:discrete-galerkin-ortho_2} then follows from inverse estimates for $\bm{v}_h$ and $q_h$.
\end{proof}

In order make use of the stronger estimate \eqref{eq:geometric_a_errors_3} of \Cref{lem:geometric_a_errors}, both vector fields need to be in $\HSpaceTan[2]{\G}$. Since we are considering discrete vector fields, these will not be as smooth. However, the following lemma will be used for approximations of $\HSpaceTan[2]{\G}$-vector fields.

\begin{lemma}\label{lem:contGalerkinOrtho}
	Let $\bm{v},\bm{w}\in \HSpaceTan[2]{\G}$, $q,r\in H^{1}(\G)$, and $\bm{v}_{h},\bm{w}_{h}\in \VelocitySpace$, and $q_h,r_h\in \PressureSpace\subset H^{j-1}(\Gh)$
	discrete approximations of $\bm{v},\bm{w}$ and $q,r$, respectively. Then
	\begin{align*}
	\Abs{\delta_{ij} (\bm{v}_h^{\ell},q_h^{\ell})(\bm{w}_h^{\ell},r_h^{\ell})}
		&\leqC h^{\kg+1} \left(\theta_1(\bm{v},\bm{v}_h,q,q_h) + \Norm{\bm{f}}{\G}\right)\theta_1(\bm{w},\bm{w}_h,r,r_h)
		 +(i-1)h^{\kk}\theta_2(\bm{v},\bm{v}_h)\theta_2(\bm{w},\bm{w}_h)\,,
	\end{align*}
	with
	\begin{align*}
		\theta_1(\bm{v},\bm{v}_h,q,q_h)
			&\colonequals \HNormTan[2]{\bm{v}}{\G} + \HNormTan[1]{q}{\G} +h^{-1}\GNorm{\bm{v}_h-\bm{v}^{e}} + h^{j-2}\HNormTan[j-1]{q^{e}-q_h}{\G} \\
		\theta_2(\bm{v},\bm{v}_h)
			&\colonequals \ahNorm{(\bm{v}_h-\bm{v}^{e})}+\HNormTan{\bm{v}}{\G}
	\end{align*}
\end{lemma}
\begin{proof}
	We split the term by writing
	\begin{align*}
	\delta_{ij} ((\bm{v}_h^{\ell},q_h^{\ell}),(\bm{w}_h^{\ell},r_h^{\ell}))
	& =
	\delta_{ij} ((\bm{v},q),(\bm{w},r))
	+ \delta_{ij} ((\bm{v}_h^{\ell}-\bm{v},q_h^{\ell}-q),(\bm{w},r))\\
	&\quad + \delta_{ij} ((\bm{v},q),(\bm{w}_h^{\ell}-\bm{w},r_h^{\ell}-r))
	+ \delta_{ij} ((\bm{v}_h^{\ell}-\bm{v},q_h^{\ell}-q),(\bm{w}_h^{\ell}-\bm{w},r_h^{\ell}-r)).
	\end{align*}
	We use \eqref{eq:geometric_a_errors_3} of \Cref{lem:geometric_a_errors}, and \Cref{lem:geometricQlError,lem:geometricQbError,lem:geometricQb1Error_tangential}
	to estimate the first term by
	\begin{align*}
	\Abs{\delta_{ij} ((\bm{v},q),(\bm{w},r))}
	&\leqC h^{\kg+1}\left(\HNormTan[2]{\bm{v}}{\G}+ \HNormTan[1]{q}{\G}+ \Norm{\bm{f}}{\G}\right)\left(\HNormTan[2]{\bm{w}}{\G}+ \HNormTan[1]{r}{\G}\right)\\
	&\qquad + (i-1)h^{\kk}\HNormTan{\bm{v}}{\G}\HNormTan{\bm{w}}{\G}\;.
	\end{align*}
	Similarly, we can estimate the remaining terms to obtain the assertion.
\end{proof}

%% file: 05_error_analysis.tex
% ============================================================================
\section{A Priori Error Analysis}\label{sec:error_estimates}

We use the inf-sup stability from \Cref{sec:inf_sups} and geometric estimates from \Cref{sec:geometric_errors} to show well-posedness and a priori error estimates in the energy norm for the four different discrete problems introduced in \Cref{sec:discrete_problem_formulation}. Afterwards, we show improvements in the tangential $H^1$- and $L^2$-norms of the errors in the velocity.

We simultaneously analyze the discrete problem with the bilinear forms $a_i$ and $b_j$. Each discrete problem $i$, $j$ can have a pair of discrete solutions $\bm{u}_h^{ij}$, $p_h^{ij}$. To keep the notation short, we omit the indices and abbreviate the discrete solutions as $\bm{u}_h$ and $p_h$.

% -----------------------------------------------------------------------------
\subsection{Well-Posedness and Stability of the Discrete Problem}\label{sec:well-posedness-discrete-prob}

We have proven in \Cref{sec:inf_sups} that as long as we assume a suitable partitioning of $\Gh$ into macroelements such that \Cref{prop:local-super-approximation,prop:stable-reference-spaces} are satisfied for the pair of discrete spaces $(\VelocitySpace,\PressureSpace)$ as defined in \Cref{def:stable-mixed-spaces}, the lifted stability estimate \eqref{eq:discrete-inf-sup-V} holds. Combined with coercivity this yields well-posedness and stability of \Cref{prob:DiscreteProblem}.

\begin{lemma}[Coercivity]
		Let $(\VelocitySpace,\PressureSpace)$ be a pair of discrete spaces. For $i=1,2$, and $\kk>0$ if $i=2$, the form $a_i+s_h$ is coercive with respect to $\AhNorm{\cdot}$, i.e., there exists a $h_0$, and $\eta_i>0$ independent of $h$, such that for all $h\leq h_0$
		\begin{align*}
			a_{i}(\bm{v}_h,\bm{v}_h)+s_h(\bm{v}_h,\bm{v}_h)\geq \eta_i \AhNorm{\bm{v}_h}^2\,,
		\end{align*}
		for all $\bm{v}_h\in \VelocitySpace$.
	\end{lemma}
	\begin{proof}
		Let $\bm{v}_h\in \VelocitySpace$. Then by \Cref{lem:basicH1Estimate}, we have
	\begin{align*}
		\ahNorm{\bm{v}_h}^2
		&\leqC \HNormTan{\PP\bm{v}_h^{\ell}}{\G}^2 + h^{\kg} \ahNorm{\bm{v}_h}\NormOne{\bm{v}_h}+ h^{2\kg}\NormOne{\bm{v}_h}^2\\
		&\leqC \HNormTan{\PP\bm{v}_h^{\ell}}{\G}^2 + h^{\kg-\frac{1}{2}} \AhNorm{\bm{v}_h}^2.
	\end{align*}
	By \Cref{lem:geometric_a_errors}, we further have
	\begin{align*}
		\HNormTan{\PP\bm{v}_h^{\ell}}{\G}^2 & \leq a_i(\bm{v}_h,\bm{v}_h) + h^{\kg-\frac{1}{2}} \AhNorm{\bm{v}_h}^2 & \text{ for }i &= 1\,, \\
		\HNormTan{\PP\bm{v}_h^{\ell}}{\G}^2 & \leq a_i(\bm{v}_h,\bm{v}_h) + (h^{\kg-\frac{1}{2}}+h^{\kk}) \AhNorm{\bm{v}_h}^2 & \text{ for }i &= 2\,.
	\end{align*}
	This implies coercivity.
\end{proof}

\begin{lemma}[Well-posedness]
	Let $(\VelocitySpace,\PressureSpace)$ be a stable pair of discrete spaces. For $i,j\in\{1,2\}$, $\PressureSpace\subset H^{j-1}(\Gh)$, $\kk>0$ if $i=2$, and $h\leq h_0$ small enough,
	each discrete saddle point problem in \Cref{prob:DiscreteProblem} has a unique solution. Moreover, there exists a constant $\beta_3 > 0$ independent of $h$ such that
	\begin{align}\label{eq:discrete-inf-sup-Brezzi}
	\sup_{0\neq(\bm{v},q)\in\VelocitySpace\times\PressureSpaceAvg}
	\frac{\mathcal{B}_{ij}\big((\bm{v},q),(\bm{w},r)\big)}{\EnergyNorm{(\bm{v},q)}{}} &\geq \beta_3 \EnergyNorm{(\bm{w},r)}{} & \text{for all }(\bm{w},r)\in\VelocitySpace\times\PressureSpaceAvg\,,
	\end{align}
	where $\PressureSpaceAvg = \PressureSpace\cap \LSpaceAvg{\Gh}$.
\end{lemma}

\begin{proof}
	Note that as $\AhNorm{\bm{v}}\leqC \NormOne{\bm{v}}$ for $\bm{v}\in \VelocitySpace$ indepent of $h$,
	\eqref{eq:discrete-inf-sup-V} implies
	\begin{equation}\label{eq:V-stability-bj}
	\sup_{\bm{v}\in\VelocitySpace}
	\frac{b_{j}\big(\bm{v},q)}{\AhNorm{\bm{v}}} \geq \beta_j \Norm{q}{\Gh} \quad\text{ for all }q\in \PressureSpaceAvg\,.
	\end{equation}
	Well-posedness and Brezzi-inf-sup stability \eqref{eq:discrete-inf-sup-Brezzi} then follow by coercivity using standard arguments, see, e.g., \cite[Lemma 5.2]{ESW2005Finite}.
\end{proof}

% -----------------------------------------------------------------
\subsection{Error Estimates in the Discrete Energy Norm}

From the stability condition and the perturbed Galerkin orthogonality we obtain directly an error estimate in the energy norm by standard arguments, cf. \cite{HP2022Tangential,HLL2020Analysis}.
%Therefore, we emphasize that $(\VelocitySpace,\PressureSpace)$ denotes a stable pair of discrete spaces of order $\ku$.

%In the following, we simulatanteously analyze the discrete problem with the bilinear forms $a_i$ and $b_j$. Each discrete problem $i,j$ can have a pair of discrete solutions $\bm{u}_h^{ij},p_h^{ij}$. To keep the notation short, we omit the indices and and abbreviate the discrete solutions as $\bm{u}_h$ and $p_h$.

% TODO: Sollten wir hier einmal noch die Polynomordnung \ku von \Gh ausschreiben und dass wir ein Kh der Ordnung \kk nutzen?

\begin{theorem}\label{thm:energy-norm-estimate}
	Let $(\VelocitySpace,\PressureSpace)$ be a stable pair of discrete spaces of order $\ku$ as in \Cref{def:stable-mixed-spaces}. For $i,j\in\{1,2\}$ we assume $\PressureSpace\subset H^{j-1}(\Gh)$ and denote by $\bm{u}_{h}\in \VelocitySpace, p_h\in\PressureSpace\cap \LSpaceAvg{\Gh}$ the discrete solution of \eqref{eq:discrete-variational-form}, and by $\bm{u}\in \HSpaceTan[\ku+1]{\G}, p\in H^{\ku}(\G)$ the continuous solution of \eqref{eq:variational-form}. For $h<h_0$ small enough, we have the estimates
	\begin{align*}
		\EnergyNorm{(\bm{u}^{e}-\IIVh\bm{u}^{e},p^e - \IIQh p^e)}{}
			&\leqC\;h^{\ku}\left(\HNormTan[\ku+1]{\bm{u}}{\G} + \HNormTan[\ku]{p}{\G}\right)\,, \\
		\EnergyNorm{(\IIVh\bm{u}^{e}-\bm{u}_{h}, \IIQh p^e - p_h)}{}
			&\leqC\;h^{m_i}\left(\HNormTan[\ku+1]{\bm{u}}{\G} + \HNormTan[\ku]{p}{\G} + \Norm{\bm{f}}{\G}\right)
	\end{align*}
	with $m_1\colonequals \min\{\ku, \kg-\frac{1}{2}\}$ and $m_2\colonequals \min\{\ku, \kg-\frac{1}{2},\kk\}$, and hence
	\begin{align*}
		\EnergyNorm{(\bm{u}^{e}-\bm{u}_{h}, p^e - p_h)}{}
			&\leqC\;h^{m_i}\left(\HNormTan[\ku+1]{\bm{u}}{\G} + \HNormTan[\ku]{p}{\G} + \Norm{\bm{f}}{\G}\right).
	\end{align*}
\end{theorem}
\begin{proof}
	The first estimate is just interpolation, \Cref{lem:interpolation-properties}. For the second estimate we note that the discrete inf-sup condition \eqref{eq:discrete-inf-sup-Brezzi}, \Cref{lem:basic_delta_estimate} \eqref{eq:discrete-galerkin-ortho_2}, and the interpolation error estimate imply that
	\begin{align*}
		\beta_3 \EnergyNorm{(\IIVh\bm{u}^{e}-\bm{u}_{h}, \IIQh p^e - p_h)}{} & \leq
		\sup_{(\bm{v}_h,q_h)\in\VelocitySpace\times\PressureSpaceAvg}
		\frac{\mathcal{B}_{ij}\big((\IIVh\bm{u}^{e}-\bm{u}_{h}, \IIQh p^e - p_h),(\bm{v}_h,q_h)\big)}{\EnergyNorm{(\bm{v}_h,q_h)}{}}\\
%		&\leq
%		\sup_{(\bm{v}_h,q_h)\in\VelocitySpace\times\PressureSpace}
%		\frac{\mathcal{B}_h\big((\IIVh\bm{u}^{e}-\bm{u}, \IIQh p^e - p),(\bm{v}_h,q_h)\big)}{\EnergyNorm{(\bm{v}_h,q_h)}{}}\\
%		&\quad
%		+
%		\sup_{(\bm{v}_h,q_h)\in\VelocitySpace\times\PressureSpace}
%		\frac{\mathcal{B}_h\big((\bm{u}^{e}-\bm{u}_{h}, p^e - p_h),(\bm{v}_h,q_h)\big)}{\EnergyNorm{(\bm{v}_h,q_h)}{}}\\
		&\leq \|\mathcal{B}_{ij}\|\EnergyNorm{(\IIVh\bm{u}^{e}-\bm{u}, \IIQh p^e - p)}{} \\
		&\qquad +
		\sup_{(\bm{v}_h,q_h)\in\VelocitySpace\times\PressureSpaceAvg}
		\frac{\delta_{ij}(\bm{u},p)(\bm{v}_h^{\ell},q_h^\ell) + s_{h}(\bm{u}^{e},\bm{v}_h)}{\EnergyNorm{(\bm{v}_h,q_h)}{}}\\
		&\leqC\;h^{m_i}\left(\HNormTan[\ku+1]{\bm{u}}{\G} + \HNormTan[\ku]{p}{\G}+\Norm{\bm{f}}{\G}\right),
	\end{align*}
	where we have used that $\EnergyNorm{\bm{u}^e}{s_h}\leqC h^{\kg-\frac{1}{2}} \Norm{\bm{u}}{\G}$ for all $\bm{u}\in \HSpaceTan[1]{\G}$.
\end{proof}

The following estimates for the normal parts of the solution and estimates in a $H^1$-norm follow directly from the energy norm estimates.

\begin{corollary}\label{lem:velocity_error_norms}
	The same conditions apply as in \Cref{thm:energy-norm-estimate}.
	For $h<h_0$ small enough, we have the estimates
	\begin{align*}
		\Norm{\QQ(\bm{u}-\bm{u}_{h}^{\ell})}{\G} & \leqC h^{m_i+\frac{1}{2}} \left(\HNormTan[\ku+1]{\bm{u}}{\G} + \HNormTan[\ku]{p}{\G}+\Norm{\bm{f}}{\G}\right)\\
		\GNorm{\bm{u}^{e}-\bm{u}_{h}}& \leqC h^{m_i-\frac{1}{2}}\left(\HNormTan[\ku+1]{\bm{u}}{\G} + \HNormTan[\ku]{p}{\G}+\Norm{\bm{f}}{\G}\right)\,,
	\end{align*}
	with $m_i$ as in \Cref{thm:energy-norm-estimate}.
\end{corollary}

%\begin{proof}
%	As we can estimate the $\GNorm{\cdot}$-norm for discrete maps by the discrete energy norm, the assertion follows from \Cref{thm:energy-norm-estimate} and interpolation error estimates.
%\end{proof}

In the following proofs (but not in the statements of the theorems) we will use the abbreviation
\begin{align*}
	\rhsConst\colonequals \HNormTan[\ku+1]{\bm{u}}{\G} + \HNormTan[\ku]{p}{\G}+\Norm{\bm{f}}{\G}.
\end{align*}

% ------------------------------------------------------
\subsection{Error Estimates in Tangential Norms}

For $\ku\geq \kk\geq \kg \geq 2$, we computationally observe a convergence order in the tangential $H^1$-norm that is higher than it is implied by the energy norm estimate \Cref{thm:energy-norm-estimate}. This was also observed and proven for the tensor Laplace equation \cite{HP2022Tangential}.

The idea for obtaining a better tangential estimate in  \cite{HP2022Tangential} is to use the  Galerkin orthogonality for the continuous bilinear and the error estimates already established. For the Stokes problem we have to account for the additional pressure error as well. This is why we start proving error estimates on the pressure that depend on the tangential error of the velocities.

%This is why we start by proving a gradient estimate on this term. In order to do this, we assume that \eqref{eq:discrete-inf-sup-BP} holds.

\begin{lemma}\label{lem:p-ph-estimate}
	The same conditions apply as in \Cref{thm:energy-norm-estimate}. For $h<h_0$ small enough, we have the estimate
	\begin{align*}
		\Norm{p-p_h^{\ell}}{\G} + h \Norm{\gradG(p-p_h^{\ell})}{\Tri[\G]} \leqC \HNormTan{\bm{u}-\PP\bm{u}_h^{\ell}}{\G} + h^{\hat{m}_i}\left(\HNormTan[\ku+1]{\bm{u}}{\G} + \HNormTan[\ku]{p}{\G}+\Norm{\bm{f}}{\G}\right)
	\end{align*}
	with $\hat{m}_1\colonequals \min\{\ku,\kg\}$ and $\hat{m}_2\colonequals \min\{\hat{m}_1,\kk\}$.
\end{lemma}
\begin{proof}
	We set $e_p= p-p_h^{\ell}$, $\tilde{e}_{p,h}=\IIQh p-p_h$, and $e_{p,h}= \tilde{e}_{p,h}- \avg{\IIQh p}{\Gh} \in \PressureSpace\cap \LSpaceAvg{\Gh}$.
	By interpolation error estimates \Cref{lem:interpolation-properties} and inverse estimates, we have
	\begin{align*}
		\Norm{p-p_h^{\ell}}{\G} +h \Norm{\gradG(p-p_h^{\ell})}{\Tri[\G]}
		%\leqC h \Norm{\gradG\tilde{e}_{p,h}}{\G} + h \Norm{\gradG(p-\IIQh p)}{\G}
		\leqC \Norm{e_{p,h}}{\G} +\Abs{\avg{\IIQh p}{\Gh}} + h^{\ku} \HNormTan[\ku]{p}{\G}.
	\end{align*}
	Further we have by interpolation error estimates
	\begin{align*}
		\Abs{\avg{\IIQh p}{\Gh}}
%		& = \avg{\IIQh p-p^{e}}{\Gh} + \frac{1}{\Abs{\Gh}}\left(\int_{\Gh}p^{e} - \int_{\G}p\right)\\
		&\leqC \Abs{\avg{\IIQh p-p^{e}}{\Gh}} + h^{\kg+1}\Norm{p}{\G}
		\leqC \left(h^{\ku} + h^{\kg+1}\right)\HNormTan[\ku]{p}{\G}.
	\end{align*}
	By V-stability \eqref{eq:discrete-inf-sup-V}, we have
	\begin{align*}
		\Norm{e_{p,h}^{\ell}}{\G}
		&\leqC \sup_{\bm{v}_h\in \VelocitySpace} \frac{b_{j}(\bm{v}_h, e_{p,h})}{\NormOne{\bm{v}_h}}.
	\end{align*}
	From \Cref{lem:difference-discrete-b} follows
	\begin{align*}
		b_{j}(\bm{v}_h, e_{p,h})
	%	&\leqC b_{j}(\bm{v}_h, \tilde{e}_{p,h})  + \Abs{ b_{j}(\bm{v}_h, \avg{I_h{} p}{\Gh})}\\
		&\leqC b_{j}(\bm{v}_h, \tilde{e}_{p,h}) + h^{\kg+1}\Abs{\avg{\IIQh p}{\Gh}}\NormZero{\bm{v}_h}\\
%		&\leqC b_{j}(\bm{v}_h, \tilde{e}_{p,h}) + h^{\kg+1}\left(h^{\ku} + h^{\kg+1}\right)\HNormTan[\ku]{p}{\G}\NormZero{\bm{v}_h}\\
		&\leqC  b_{j}(\bm{v}_h, e_{p}^{e}) + b_{j}(\bm{v}_h,\IIQh p-p^{e}) + h^{\kg+1}\left(h^{\ku} + h^{\kg+1}\right)\HNormTan[\ku]{p}{\G}\NormZero{\bm{v}_h}.
	\end{align*}
	By \Cref{lem:difference-discrete-b} and interpolation error estimates, we have
	\begin{align*}
		b_{j}(\bm{v}_h,\IIQh p-p^{e})
		&\leqC b_{1}(\bm{v}_h,\IIQh p-p^{e}) + (j-1)h^{\kg}\left(h\Norm{\Grad(\IIQh p-p^{e})}{\Gh}+ \Norm{\IIQh p-p^{e}}{\Gh}\right)\NormOne{\bm{v}_h}\\
		&\leqC \left((j-1)h^{\kg+1}\Norm{\Grad(\IIQh p-p^{e})}{\Gh}+ \Norm{\IIQh p-p^{e}}{\Gh}\right)\NormOne{\bm{v}_h}\\
		&\leqC h^{\ku}\HNormTan[\ku]{p}{\G}\NormOne{\bm{v}_h}.
	\end{align*}
	Thus, we have
	\begin{align*}
		b_{j}(\bm{v}_h, e_{p,h}) & \leqC b_{j}(\bm{v}_h, e_{p}^{e}) + (h^{\ku}+h^{\kg+1})\HNormTan[\ku]{p}{\G}\NormOne{\bm{v}_h}.
	\end{align*}
	We write
	\begin{align*}
		b_{j}(\bm{v}_h, e_{p}^{e})
		& = \delta_{ij}(\bm{u}, p)(\bm{v}_h,0)
		+s_h(\bm{u}_h,\bm{v}_h)
		- a_i(\bm{v}_h, \bm{u}^{e}-\bm{u}_h).
	\end{align*}
	We have by \Cref{lem:basic_delta_estimate}  \eqref{eq:discrete-galerkin-ortho_1}
	\begin{align*}
		\Abs{\delta_{ij}(\bm{u},p)(\bm{v}_h^{\ell},0)}
		&\leqC	\left(h^{\kg} +(i-1)h^{\kk}\right)\rhsConst
		\NormOne{\bm{v}_h}.
	\end{align*}
	Further, we have by \Cref{thm:energy-norm-estimate}
	\begin{align*}
		\Abs{s_h(\bm{u}_h,\bm{v}_h)}
		&\leq h^{\frac{1}{2}}\EnergyNorm{\bm{u}_h}{s_h}\NormZero{\bm{v}_h}
		\leqC h^{\frac{1}{2}} \left(\EnergyNorm{\bm{u}_h-\bm{u}^{e}}{s_h} + h^{\kg-\frac{1}{2}}\Norm{\bm{u}}{\G}\right) \NormZero{\bm{v}_h}
		\leqC  h^{m_i+\frac{1}{2}} \rhsConst \NormZero{\bm{v}_h}.
	\end{align*}
	By \Cref{lem:geometric_a_errors,lem:basicH1Estimate,lem:velocity_error_norms}, we have
	\begin{align*}
		\Abs{\delta_{a_i}(\bm{v}_h, \bm{u}^{e}-\bm{u}_h)}
		&\leqC \left(h^{\kg} \GNorm{\bm{u}^{e}-\bm{u}_h} + (i-1)h^{\kk}\ahNorm{(\bm{u}^{e}-\bm{u}_h)}\right)\NormOne{\bm{v}_h}\\
		&\leqC \left(h^{\kg}(1+(i-1)h^{\kk}) \GNorm{\bm{u}^{e}-\bm{u}_h} + (i-1)h^{\kk}\HNormTan{\bm{u}-\PP\bm{u}_h^{\ell}}{\G}\right)\NormOne{\bm{v}_h}\\
		&\leq \left(h^{\kg+m_i-\frac{1}{2}} \rhsConst + (i-1)h^{\kk}\HNormTan{\bm{u}-\PP\bm{u}_h^{\ell}}{\G}\right)\NormOne{\bm{v}_h}.
	\end{align*}
	Thus, we obtain
	\begin{align*}
		\Abs{a_i(\bm{v}_h, \bm{u}^{e}-\bm{u}_h)}
		&\leqC \Abs{a(\PP\bm{v}_h^{\ell}, \bm{u}-\PP\bm{u}_h^{\ell})} + \Abs{\delta_{a_i}(\bm{v}_h, \bm{u}^{e}-\bm{u}_h)}
		\leqC \left(\HNormTan{\bm{u}-\PP\bm{u}_h^{\ell}}{\G}
		+ h^{\kg+m_i-\frac{1}{2}} \rhsConst\right) \NormOne{\bm{v}_h}.
	\end{align*}
	We have now proven that
	\begin{align*}
		\sup_{\bm{v}_h\in \VelocitySpace} \frac{b_{j}(\bm{v}_h, e_{p}^{e})}{\NormOne{\bm{v}_h}}
		&\leqC  \HNormTan{\bm{u}-\PP\bm{u}_h^{\ell}}{\G}
		+ \left(h^{m_i+\frac{1}{2}} +(i-1)h^{\kk}\right)\rhsConst\;.
	\end{align*}
	Using the definition of $m_i$, this yields the assertion.
\end{proof}

We cite the following lemma from \cite{HP2022Tangential}:

\begin{lemma}\label{lem:qip-lagrange}
	Let $\bm{v}\in \HSpaceTan[m]{\G}\cap C(\G,\R^3)$ for $m\geq 2$, and let $\I_{h,k}\colon C(\G)\to[S_h(\Gh;\FSpace_k)]^3$ denote Lagrange interpolation of order $k\geq m-1$.
	Let $\El\in\Tri[\G]$ be a curved element, and $h_{\El}$ the diameter of $\El$. Then for $0\leq l\leq \min\{m, k\}$
	\begin{align*}
	\Norm{\QQ \I_{h,k}\bm{v}}{\El}
		\leqC h_{\El}^{l+1} \HNorm[l]{\I_{h,k} \bm{v}}{\El}
	%	\leqC h_{\El}^{l+1} (\HNorm[l]{\bm{v}}{\El}+ \HNorm[l]{\I_L^{\ku} \bm{v}- \bm{v}}{\El})
		\leqC h_{\El}^{l+1}\HNormTan[l]{\bm{v}}{\El}+ h_{\El}^{m+1}\HNormTan[m]{\bm{v}}{\El}.
	\end{align*}
\end{lemma}

\begin{proof}
	See \cite[Lemma 4.5]{HP2022Tangential} and combine with interpolation error estimates \cite[Lemma 4.1]{HP2022Tangential} and norm equivalence \cite[Lemma 2.3]{HLL2020Analysis}.
\end{proof}

Using the estimates on the pressure and an Aubin--Nitsche type argument, we prove an estimate in the tangential $L^2$-norm, depending on the tangential energy norm.

\begin{theorem}\label{thm:tangential_l2_estimate}
	The same conditions apply as in \Cref{thm:energy-norm-estimate}. For $h<h_0$ small enough, we have the estimate
	\begin{align*}
		\Norm{\PP(\bm{u}-\bm{u}_{h}^ {\ell})}{\G} \leqC
		h\HNormTan{\bm{u}-\PP\bm{u}_{h}^{\ell}}{\G}+ h^{l_i}\left(\HNormTan[\ku+1]{\bm{u}}{\G} + \HNormTan[\ku]{p}{\G}+\Norm{\bm{f}}{\G}\right)
	\end{align*}
	with $l_{1}\colonequals \min\{\ku+1,\kg+1,2\kg-1\}$ and $l_2\colonequals \min\{l_1, \kk\}$.
\end{theorem}
\begin{proof}
	We set $\bm{e}\colonequals \bm{u} - \bm{u}_h^{\ell}$
	We consider the adjoint problem defined by
	\begin{align}\label{eq:adjoint_problem}
		(\bm{\psi},\phi)\in \HSpaceTan[1]{\G} \times \LSpaceAvg[2]{\G}:\quad  \mathcal{B}\left((\bm{v},q),(\bm{\psi},\phi)\right) = \Inner{\PP\bm{e}}{\bm{v}}{\G}\quad \forall (\bm{v},q)\in \HSpaceTan[1]{\G} \times \LSpaceAvg[2]{\G}
	\end{align}
	By \Cref{lem:regularity}, the solution $(\bm{\psi},\phi)$ fulfills the regularity estimates
	\begin{align*}
		\HNormTan[2]{\bm{\Psi}}{\G} + \Norm{\phi}{H^1(\G)} &\leqC \Norm{\PP\bm{e}}{\G}.
	\end{align*}
	Note that $\bm{\psi}$ is continuous.
	We interpolate the solution using in particular Lagrange interpolation for $\bm{\psi}$, and set
	$\bm{\psi}_{h} \colonequals \I_{h,\min\{\ku,2\}} \bm{\psi} \in [S_h(\Gh;\FSpace_{\min\{\ku,2\}})]^3\subset \VelocitySpace$ and $\phi_h \colonequals \IIQh{\phi} \in \PressureSpace$.
	By interpolation error estimates and the regularity, we have
	\begin{align}\label{eq:psi-psih}
		\EnergyNorm{\left(\bm{\psi}-\bm{\psi}_{h}^{\ell},\phi-\phi_{h}^{\ell}\right)}{\mathcal{B}}
		&\leqC \GNorm{\bm{\psi}^{e}-\bm{\psi}_{h}} + \Norm{\phi-\phi_{h}^{\ell}}{\G}
		\leqC  h\left(\HNormTan[2]{\bm{\Psi}}{\G} + \Norm{\phi}{H^1(\G)}\right)
		\leqC h \Norm{\PP\bm{e}}{\G},
	\end{align}
	and by \Cref{lem:qip-lagrange} with $m=2$, and $l=k=\min\{2,\ku\}$, as well as regularity, we have
	\begin{align*}
	\Norm{\QQ \bm{\psi}_h}{\Gh}
	\leqC  h^{\min\{2,\ku\}+1}\HNormTan[\min\{2,\ku\}]{\bm{\psi}}{\Gh}+ h^{3}\HNormTan[2]{\bm{\psi}}{\Gh} \leqC h^{\min\{3,\ku+1\}} \Norm{\PP\bm{e}}{\G}.
	\end{align*}
	This implies
	\begin{align}\label{eq:Qh_psih_estimate}
		\EnergyNorm{\bm{\psi}_{h}}{s_h}
		&\leqC h^{-\frac{1}{2}} \Norm{\QQ\bm{\psi}_{h}}{\Gh} + h^{\kg-\frac{1}{2}} \Norm{\bm{\psi}}{\Gh}
		\leqC h^{\min\{3,\ku+1,\kg\}-\frac{1}{2}}\Norm{\PP\bm{e}}{\G}.
	\end{align}
	We insert $\bm{v}= \PP\bm{e}$ and $q=p-p_h^{\ell}-\avg{p_h^{\ell}}{\G}$ as a test function pair into \eqref{eq:adjoint_problem} and obtain
	\begin{align}\label{eq:l2-pf-1}
		\Norm{\PP\bm{e}}{\G}^2
		& = \mathcal{B}\left((\PP\bm{e},p-p_h^{\ell}-\avg{p_h^{\ell}}{\G}),(\bm{\psi},\phi)\right) \notag \\
%		& =  \mathcal{B}\left((\PP\bm{e},p-p_h^{\ell}),(\bm{\psi}-\bm{\psi}_h,\phi-\phi_h)\right)
%		+\mathcal{B}\left((\PP\bm{e},p-p_h^{\ell}),(\bm{\psi}_h,\phi_h)\right)\notag\\
		& =  \mathcal{B}\left((\PP\bm{e},p-p_h^{\ell}),(\bm{\psi}-\bm{\psi}_h,\phi-\phi_h)\right)
		+\delta_{ij}\left(\bm{u}_{h}^{\ell},p_h^{\ell}\right)(\bm{\psi}_h,\phi_h)
		+ s_h\left(\bm{u}_{h},\bm{\psi}_h\right).
	\end{align}
	We estimate by \eqref{eq:psi-psih} and \Cref{lem:p-ph-estimate}
	\begin{align*}
		\mathcal{B}\left((\PP\bm{e},p-p_h^{\ell}),(\bm{\psi}-\bm{\psi}_h,\phi-\phi_h)\right)
		&\leqC \EnergyNorm{(\PP\bm{e},p-p_h^{\ell})}{\mathcal{B}}
		\EnergyNorm{(\bm{\psi}-\bm{\psi}_h,\phi-\phi_h)}{\mathcal{B}}\\
		&\leqC h \left(\HNormTan{\PP\bm{e}}{\G} + \Norm{p-p_h^{\ell}}{\G} \right)\Norm{\PP\bm{e}}{\G}\\
		&\leqC h \left(\HNormTan{\PP\bm{e}}{\G} + h^{\hat{m}_i}\rhsConst \right)\Norm{\PP\bm{e}}{\G}.
	\end{align*}
	The $s_h$-term is estimated by \Cref{thm:energy-norm-estimate} and \eqref{eq:Qh_psih_estimate}
	\begin{align*}
		\Abs{s_h\left(\bm{u}_{h},\bm{\psi}_h\right)}
		& \leqC \EnergyNorm{\bm{u}_h}{s_h}\EnergyNorm{\bm{\psi}_h}{s_h} \leqC h^{m_i+\min\{3,\ku+1,\kg\}-\frac{1}{2}}\rhsConst \Norm{\PP\bm{e}}{\G}.
	\end{align*}
	We use \Cref{lem:contGalerkinOrtho} to estimate
%	\begin{align*}
%		\delta_{ij} ((\bm{u}_h^{\ell},p_h^{\ell}),(\bm{\psi}_h^{\ell},\phi_h^{\ell}))
%		&\leqC h^{\kg+1} \left(\HNormTan[2]{\bm{u}}{\G} + \HNormTan[1]{p}{\G} +h^{-1}\GNorm{\bm{u}_h-\bm{u}^{e}} + \HNormTan[1]{p-p_h}{\G} + \Norm{\bm{f}}{\G}\right)\\
%		&\quad \cdot\left(\HNormTan[2]{\bm{\psi}}{\G} + \HNormTan[1]{\phi}{\G}+h^{-1}\GNorm{\bm{\psi}_h-\bm{\psi}^{e}} + \HNormTan[1]{\phi-\phi_h}{\G}\right)\\
%		&\quad  +(i-1)h^{\kk}\left(\ahNorm{\bm{u}_h-\bm{u}^{e}}+\HNormTan{\bm{u}}{\G}\right)\left(\ahNorm{\bm{\psi}_h-\bm{\psi}^{e}}+\HNormTan{\bm{\psi}}{\G}\right).
%	\end{align*}
	\begin{align*}
		\delta_{ij} ((\bm{u}_h^{\ell},p_h^{\ell}),(\bm{\psi}_h^{\ell},\phi_h^{\ell}))
		&\leqC h^{\kg+1} \left(\rhsConst +h^{-1}\GNorm{\bm{u}_h-\bm{u}^{e}} + h^{j-2}\HNormTan[j-1]{p^{e}-p_h}{\Gh}\right)\Norm{\PP\bm{e}}{\G}\\
		&\quad  +(i-1)h^{\kk}\left(\ahNorm{(\bm{u}_h-\bm{u}^{e})}+\rhsConst\right)\Norm{\PP\bm{e}}{\G}.
	\end{align*}
	By \Cref{lem:velocity_error_norms}, we have
	\begin{align*}
		\GNorm{\bm{u}_h-\bm{u}^{e}}\leqC h^{m_i-\frac{1}{2}}\rhsConst,
	\end{align*}
	%and by interpolation error estimates, inverse estimates, and \Cref{thm:energy-norm-estimate},
	%\begin{align*}
	%	 \HNormTan[1]{p-p_h}{\G} \leqC h^{m_i-1}\rhsConst.
	%\end{align*}
	and by \Cref{lem:p-ph-estimate}
	\begin{align*}
		h^{j-2}\HNormTan[j-1]{p^{e}-p_h}{\Gh} \leqC h^{-1}\HNormTan{\PP\bm{e}}{\G} +\rhsConst.
	\end{align*}
%
%	Further, we use
%	\begin{align*}
%		\left(\HNormTan[2]{\bm{\psi}}{\G} + \HNormTan[1]{\phi}{\G}+h^{-1}\GNorm{\bm{\psi}_h-\bm{\psi}^{e}} + \HNormTan[1]{\phi-\phi_h}{\G}\right)\leqC \Norm{\PP\bm{e}}{\G}.
%	\end{align*}
%
	W.l.o.g. we can assume $m_i-\frac{1}{2}\geq 0$, and obtain
	\begin{align*}
		\delta_{ij} ((\bm{u}_h^{\ell},p_h^{\ell}),(\bm{\psi}_h^{\ell},\phi_h^{\ell}))
		&\leqC h^{\kg} \HNormTan{\PP\bm{e}}{\G}\Norm{\PP\bm{e}}{\G}+  \left(h^{\kg+1} +h^{m_i+\kg-\frac{1}{2}} \right)\rhsConst\Norm{\PP\bm{e}}{\G}\\
		&\quad  +(i-1)h^{\kk}\rhsConst\Norm{\PP\bm{e}}{\G}.
	\end{align*}

	Using these estimates in \eqref{eq:l2-pf-1} yields
	\begin{align*}
		\Norm{\PP\bm{e}}{\G}^2
		&\leqC
		h \HNormTan{\PP\bm{e}}{\G}\Norm{\PP\bm{e}}{\G}
		+ \left( h^{\hat{m}_i+1} +h^{m_i+\min\{3,\ku+1,\kg\}-\frac{1}{2}} \right)\rhsConst\Norm{\PP\bm{e}}{\G}\\
		&\quad  +(i-1)h^{\kk}\rhsConst\Norm{\PP\bm{e}}{\G}.
	\end{align*}
	Using $m_1=\min\{\ku,\kg-\frac{1}{2}\}$ and $\hat{m}_1=\min\{\ku,\kg\}$ yields the assertion.
\end{proof}

We will now prove a better estimate (than \Cref{thm:energy-norm-estimate} provides) for the tangential $H^1$-error. We proceed analogously to similar estimates for the tensor Poisson equation in \cite{HP2022Tangential}.

\begin{theorem}\label{thm:H1-norm-estimate}
	The same conditions apply as in \Cref{thm:energy-norm-estimate}. For $h<h_0$ small enough, we have the estimate
	\begin{align*}
	\HNormTan{\bm{u}-\PP\bm{u}_{h}^{\ell}}{\G}
	&\leqC\;h^{\hat{m}_i}\left(\HNormTan[\ku+1]{\bm{u}}{\G} + \HNormTan[\ku]{p}{\G} + \Norm{\bm{f}}{\G}\right)
	\end{align*}
	with $\hat{m}_1= \min\{\ku, \kg\}$ and $\hat{m}_2= \min\{\hat{m}_1,\kk\}$.
\end{theorem}

\begin{proof}
	Let $\bm{e}=\bm{u}-\bm{u}_{h}^{\ell}$. As in the proof of \Cref{thm:energy-norm-estimate} we split the error $\bm{e}$ into the interpolation error $\bm{e}_{I} = \bm{u}-\IIVh\bm{u}^{\ell}$ and the discrete remainder $\bm{e}_{d} = \IIVh\bm{u}-\bm{u}_{h}$.

	Note that we have by inverse estimates, interpolation error estimates, \Cref{thm:tangential_l2_estimate}, and \Cref{lem:velocity_error_norms}
	\begin{align*}
	\Norm{\bm{e}_d}{\Gh} + h\NormOne{\bm{e}_{d}}
	&\leqC \Norm{\bm{e}}{\G} + h^{\ku}\rhsConst
	\leqC \Norm{\PP\bm{e}}{\G} + \Norm{\QQ\bm{e}}{\G} + h^{\ku}\rhsConst
	\leqC h\HNormTan{\PP\bm{e}}{\G} + h^{\hat{m}_i}\rhsConst.
	\end{align*}
	By \Cref{lem:velocity_error_norms} we have for $\bm{e}$ the estimates
	\begin{align*}
	\GNorm{\bm{e}^{e}}&\leqC h^{m_i-\frac{1}{2}}\rhsConst\leqC h^{\hat{m}_i-1}\rhsConst\,,\\
	\intertext{and}
	\ahNorm{\bm{e}^{e}}&\leqC \HNormTan{\PP\bm{e}}{\G} + h^{\kg}\GNorm{\bm{e}^{e}}
	\leqC \HNormTan{\PP\bm{e}}{\G} + h^{\hat{m}_i}\rhsConst\,.
	\end{align*}

	Consider the product $[\PP\bm{e}_{d}]^i = \sum_j P^i_j e_d^j$ and apply the interpolation operator $\II_{h,\FSpace_V}$ summandwise. As $\PP$ is smooth, we have by discrete commutator property of $\II_{h,\FSpace_V}$, see \cite{Bertoluzza1999DiscreteCommutator,ErnGuermond2004Theory},

	\begin{align*}
	\Norm{P^i_j e_d^j - \II_{h,\FSpace_V}[P^i_j e_d^j]}{H^m(\Gh)}
	&\leqC h \Norm{e_d^j}{H^m(\Gh)},\text{ for }m=0,1\,.
	\end{align*}
	We set $[\bm{e}_h]^i = \sum_j \II_{h,\FSpace_V}[P^i_j e_d^j] \in V_h$ . From the componentwise estimate we deduce
	\begin{align*}
	\Norm{\PP\bm{e}_d - \bm{e}_h}{\Gh} &\leqC h \Norm{\bm{e}_d}{\Gh}\,,\\
	\HNorm[1]{\PP\bm{e}_d - \bm{e}_h}{\Gh} &\leqC h \NormOne{\bm{e}_d}\,.
	\end{align*}
	We estimate the normal part of $\bm{e}_h$,
	\begin{align*}
	\Norm{\QQh\bm{e}_h}{\Gh}
	&\leqC \Norm{\QQh\bm{e}_h - \QQh\PP\bm{e}_d}{\Gh} + \Norm{\QQh\PP\bm{e}_d}{\Gh}
	\leqC h \Norm{\bm{e}_d}{\Gh}\leqC h^2\HNormTan{\PP\bm{e}}{\G} + h^{\hat{m}_i+1}\rhsConst.
	\end{align*}
	Thus, we have
	\begin{align*}
	\HNormTan{\PP\bm{e}_h^\ell}{\G}
	& \leqC \HNorm{\bm{e}_h^\ell-\PP\bm{e}_d}{\G}+  \HNormTan{\PP\bm{e}_d^{\ell}}{\G}
	\leqC \HNormTan{\PP\bm{e}}{\G} + h\NormOne{\bm{e}_d} + h^{\ku}\rhsConst\leqC \HNormTan{\PP\bm{e}}{\G} + h^{\hat{m}_i}\rhsConst,
	\\
	\NormOne{\bm{e}_{h}}
	&\leqC \HNormTan{\PP\bm{e}_h^\ell}{\G} + h^{-1}\Norm{\QQh\bm{e}_{h}}{\Gh}
	\leqC \HNormTan{\PP\bm{e}}{\G} +h^{\hat{m}_i}\rhsConst.
	\end{align*}
	We further estimate
	\begin{align*}
	\HNormTan{\PP(\bm{e} - \bm{e}_{h}^\ell)}{\G}
	\leqC h^{\ku}\rhsConst + h\NormOne{\bm{e}_d}
	\leqC h\HNormTan{\PP\bm{e}}{\G} + h^{\hat{m}_i}\rhsConst\,.
	\end{align*}

	To estimate $\HNormTan{\PP\bm{e}}{\G}$, we use the bilinear form $a$, add and subtract $\PP\bm{e}_h^{\ell}$, and set $q_h=p_h - \IIQh p - \avg{\IIQh p}{\Gh}\in \PressureSpace\cap \LSpaceAvg{\Gh}$ to obtain
	\begin{align*}
	\HNormTan{\PP\bm{e}}{\G}^2
	& \leqC a(\PP\bm{e},\PP(\bm{e}-\bm{e}_h^{\ell})) + a(\PP\bm{e},\PP\bm{e}_h^{\ell})\\
	& = a(\PP\bm{e},\PP(\bm{e}-\bm{e}_h^{\ell})) + \mathcal{B}((\bm{e},p-p_h),(\bm{e}_h^{\ell},q_h^{\ell})) + b(\bm{e}_h^{\ell},p-p_h)+ b(\bm{e},q_h^{\ell})\\
	& = a(\PP\bm{e},\PP(\bm{e}-\bm{e}_h^{\ell})) + s_h(\bm{u}_h,\bm{e}_h) + \delta_{ij}((\bm{u}_h^{\ell},p_h^{\ell}),(\bm{e}_h^{\ell},q_h^{\ell}))\\
	&\quad + b(\bm{e}_h^{\ell},p-(\IIQh p)^{\ell} -  \avg{\IIQh p}{\Gh}) +  b(\bm{e}-\bm{e}_h^{\ell},q_h^{\ell}).%\\
	\end{align*}
	We estimate
	\begin{align*}
	a(\PP\bm{e},\PP(\bm{e}-\bm{e}_h^{\ell})) & \leq \HNormTan{\PP\bm{e}}{\G}\HNormTan{\PP(\bm{e}-\bm{e}_h^{\ell})}{\G}
	\leqC h \HNormTan{\PP\bm{e}}{\G}^2 + h^{\hat{m}_i}\rhsConst \HNormTan{\PP\bm{e}}{\G}.
	\end{align*}
	It is by \Cref{thm:energy-norm-estimate}
	\begin{align*}
	s_h(\bm{u}_h,\bm{e}_h) & \leqC \EnergyNorm{\bm{u}_h}{s_h}\EnergyNorm{\bm{e}_h}{s_h}
	\leqC h^{m_i+\frac{1}{2}}\rhsConst (h\HNormTan{\PP\bm{e}}{\G} + h^{\hat{m}_i}\rhsConst)
	\leqC h^{\hat{m}_i+1}\rhsConst \HNormTan{\PP\bm{e}}{\G} + h^{2\hat{m}_i}\rhsConst^2.
	\end{align*}
	For the $\delta_{ij}$-Term, we split
	\begin{align*}
	\delta_{ij} ((\bm{u}_h^{\ell},p_h^{\ell}),(\bm{e}_h^{\ell},q_h^{\ell}))
	& = \delta_{ij} ((\bm{u},p),(\bm{e}_h^{\ell},q_h^{\ell}))
	- \delta_{ij} ((\bm{e},p_h^{\ell}-p),(\bm{e}_h^{\ell},q_h^{\ell}))
	\end{align*}
	and estimate  by \Cref{lem:basic_delta_estimate} \eqref{eq:discrete-galerkin-ortho_1}
	\begin{align*}
	\delta_{ij} ((\bm{u},p),(\bm{e}_h^{\ell},q_h^{\ell}))
	&\leqC h^{\kg}\rhsConst\left(\NormOne{\bm{e}_{h}} + \Norm{q_h}{\Gh}\right) + (i-1)h^{\kk}\rhsConst \ahNorm{\bm{e}_h}
	\end{align*}
	Note that by interpolation error estimates and \Cref{lem:p-ph-estimate}, we have
	\begin{align*}
	\Norm{q_{h}^{\ell}}{\G}
	& \leqC 	\Norm{p_h-\IIQh p}{\Gh} \leqC 	h^{\ku}\rhsConst + \Norm{p-p_{h}^{\ell}}{\G} \leqC \HNormTan{\PP\bm{e}}{\G} + h^{\hat{m}_i}\rhsConst.
	\end{align*}
	Hence, using the estimates for $\bm{e}_h$ established at the beginning of the proof, we have
	\begin{align*}
		\delta_{ij} ((\bm{u},p),(\bm{e}_h^{\ell},q_h^{\ell}))
		&\leqC h^{\kg}\rhsConst\left( \HNormTan{\PP\bm{e}}{\G} +h^{\hat{m}_i}\rhsConst\right) + (i-1)h^{\kk}\rhsConst ( \HNormTan{\PP\bm{e}}{\G} +  h^{\hat{m}_i}\rhsConst).
	\end{align*}
	Analogously, we have using \Cref{lem:basic_delta_estimate} \eqref{eq:basic_delta_estimate}
		\begin{align*}
	\delta_{ij} ((\bm{e},p_h^{\ell}-p),(\bm{e}_h^{\ell},q_h^{\ell}))
		&\leqC h^{\kg}\left(\NormOne{\bm{e}_h} + \Norm{q_h}{\Gh}\right)\left(\GNorm{\bm{e}^{e}} + \Norm{p_h-p^{e}}{\Gh} + h\rhsConst\right)\\
		&\quad +(i-1)h^{\kk}\ahNorm{\bm{e}^{e}} \ahNorm{\bm{e}_h}\\
	&\leqC h^{\kg}\left(\HNormTan{\PP\bm{e}}{\G} +h^{\hat{m}_i}\rhsConst\right)\left(\HNormTan{\PP\bm{e}}{\G} + \big(h^{\hat{m}_i-1} + h\big)\rhsConst\right)\\
	&\quad +(i-1)h^{\kk}(\HNormTan{\PP\bm{e}}{\G} + h^{\hat{m}_i}\rhsConst)^2.
	\end{align*}
	Thus, we obtain
	\begin{align*}
	\delta_{ij} ((\bm{u}_h^{\ell},p_h^{\ell}),(\bm{e}_h^{\ell},q_h^{\ell}))
	&\leqC h^{\kg}\left(\HNormTan{\PP\bm{e}}{\G} +h^{\hat{m}_i}\rhsConst\right)\left(\HNormTan{\PP\bm{e}}{\G} + \big(h^{\hat{m}_i-1} + 1\big)\rhsConst\right)\\
	&\quad +(i-1)h^{\kk}(\HNormTan{\PP\bm{e}}{\G} +\rhsConst)\left(\HNormTan{\PP\bm{e}}{\G} +h^{\hat{m}_i}\rhsConst\right).
	\end{align*}

	For the $b$-terms, we estimate
	using interpolation error estimates
	\begin{align*}
	b(\bm{e}_h^{\ell},p-(\IIQh p)^{\ell})
	&\leqC \HNormTan{\PP\bm{e}_h^{\ell}}{\G}\Norm{p-(\IIQh p)^{\ell}}{\G}
	\leqC h^{\ku}\rhsConst \left(\HNormTan{\PP\bm{e}}{\G} +h^{\hat{m}_i}\rhsConst\right),
	\end{align*}
	and
	\begin{align*}
	b(\bm{e}-\bm{e}_h^{\ell},q_h^{\ell})
	&\leqC \HNormTan{\PP(\bm{e}-\bm{e}_h^{\ell})}{\G}\Norm{q_h^{\ell}}{\G}
	\leqC\left(h\HNormTan{\PP\bm{e}}{\G} +h^{\hat{m}_i}\rhsConst\right)\left(\HNormTan{\PP\bm{e}}{\G} + h^{\hat{m}_i}\rhsConst\right).
	\end{align*}

	We summarize
	\begin{align*}
		\HNormTan{\PP\bm{e}}{\G}^2
		& \leqC
		h \HNormTan{\PP\bm{e}}{\G}^2 + h^{\hat{m}_i}\rhsConst \HNormTan{\PP\bm{e}}{\G} + h^{2\hat{m}_i}\rhsConst^2 \\
		&\qquad +(i-1)h^{\kk}(\HNormTan{\PP\bm{e}}{\G}^2 +\rhsConst\HNormTan{\PP\bm{e}}{\G} +  h^{\hat{m}_i}\rhsConst^2).
	\end{align*}
	For $i=2$, we can assume w.l.o.g. $\kk>0$. Thus for $h$ small enough, we obtain the assertion.
\end{proof}

\begin{corollary}\label{thm:tangential-L2-errors}
	The same conditions apply as in \Cref{thm:energy-norm-estimate}. For $h<h_0$ small enough, we have the estimates
	\begin{align*}
		\Norm{p-p_h^{\ell}}{\G} & \leqC h^{\hat{m}_i}\left(\HNormTan[\ku+1]{\bm{u}}{\G} + \HNormTan[\ku]{p}{\G}+\Norm{\bm{f}}{\G}\right),\\
		\Norm{\PP(\bm{u}-\bm{u}_{h}^{\ell})}{\G}
		&\leqC h^{l_i} \left(\HNormTan[\ku+1]{\bm{u}}{\G} + \HNormTan[\ku]{p}{\G}+\Norm{\bm{f}}{\G}\right)\,,
	\end{align*}
	with $\hat{m}_1= \min\{\ku, \kg\}$, $\hat{m}_2 = \min\{\hat{m}_1,\kk\}$ as in \Cref{lem:p-ph-estimate} and
	$l_{1}=\min\{\ku+1,\kg+1,2\kg-1\}$, $l_2= \min\{l_1, \kk\}$ as in \Cref{thm:tangential_l2_estimate}.
\end{corollary}

%% file: 06_numerical_experiments.tex
% =============================================================================
\section{Numerical Experiments and Results}\label{sec:numerical-experiments}
In the numerical experiments we parametrized a sphere $\G=\mathcal{S}^2$ with surface finite elements by starting from a coarse reference triangulation and projecting refined grid vertices to $\G$ using the closest-point projection $\pi(\bm{x})=\bm{x}/\Norm{\bm{x}}{}$. A red-refinement of the grid elements gives a sequence of piecewise flat reference surfaces $\hat{\G}_h$. The higher-order parametrization is then obtained by piecewise interpolation of $\pi$ on the flat elements. The numerical experiments below are implemented in the finite element framework AMDiS/Dune \cite{WLPV2015Software,BBD2021Dune} using dune-foamgrid \cite{SKSF2017Dune} to represent the reference grid and dune-curvedgrid \cite{PS2020DuneCurvedGrid} to implement the higher-order geometry mappings.

\begin{figure}[ht!]
  \centering
  \begin{subfigure}[b]{0.32\textwidth}
      \centering
      \includegraphics[width=\textwidth]{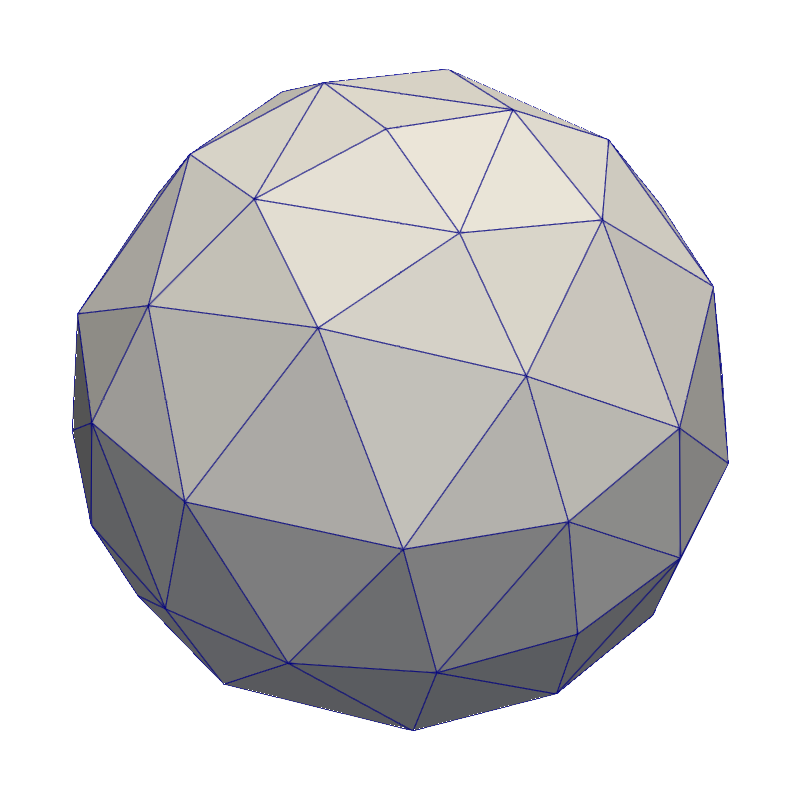}
      \caption{Reference grid $\hat{\G}_h$}
  \end{subfigure}
  \hfill
  \begin{subfigure}[b]{0.32\textwidth}
      \centering
      \includegraphics[width=\textwidth]{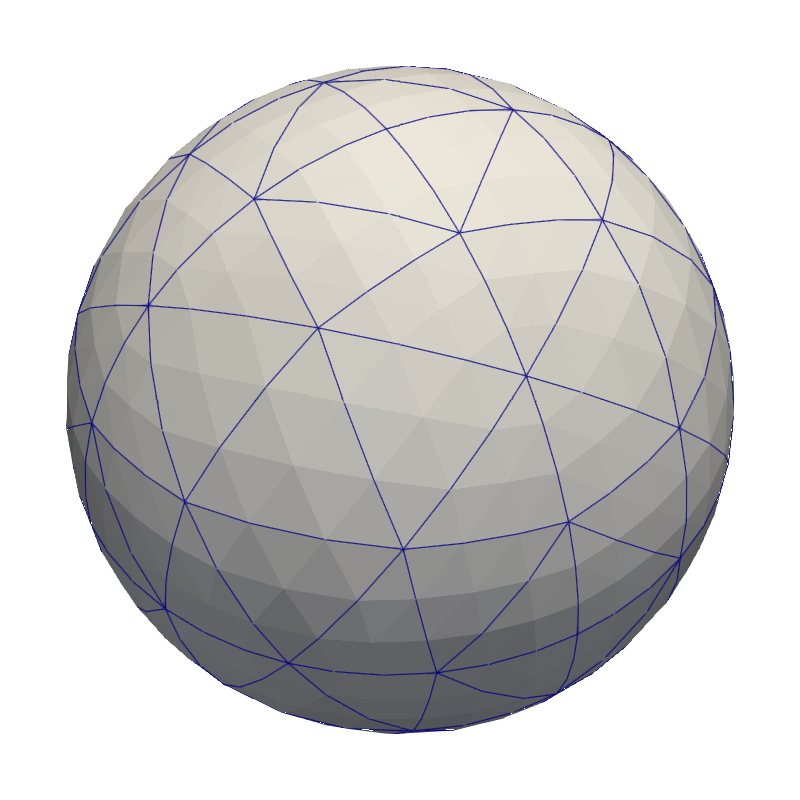}
      \caption{Curved triangulation $\Gh$}
  \end{subfigure}
  \hfill
  \begin{subfigure}[b]{0.32\textwidth}
      \centering
      \includegraphics[width=\textwidth]{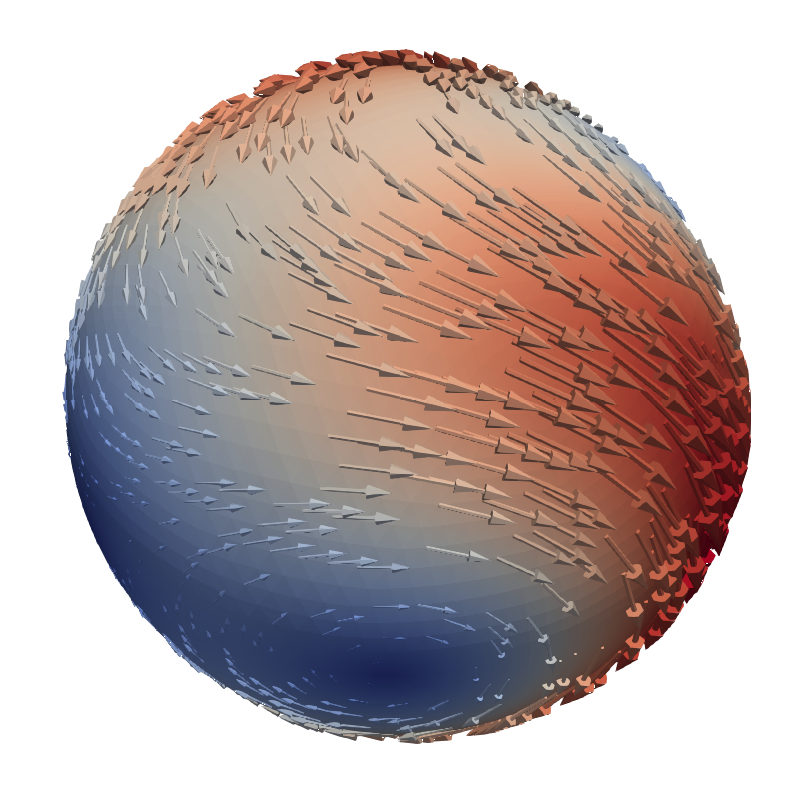}
      \caption{Velocity solution $\bm{u}^*$}
  \end{subfigure}
  \caption{(Colours online) The surface grid of the \textit{sphere} geometry and velocity solution of the surface Stokes equation. In the curved triangulation surface Lagrange parametrization of order $\kg=3$ is used. The arrows in the solution plot indicate the direction of the vector field and the colours its magnitude.}\label{fig:sphere-shape}
\end{figure}

We introduce a reference solution $\{\bm{u}^*,p^*\}$ to the Stokes problem \eqref{eq:stokes-problem} defined in the embedding space $\R^3$ as
\begin{align}
  p^*(x,y,z)      &\colonequals x\,, &
  \bm{u}^*(x,y,z) &\colonequals \CurlG(z - x^2)\,,
\end{align}
with $\CurlG$ the surface curl resulting in a tangential vector field. From these reference solutions we analytically compute the right-hand side load function $\bm{f}$ for the Stokes problem,
\begin{align}
  \bm{f}(x,y,z) &= \transposed{\big(-x^2 - y + 1, x(6z - y + 1), -x(6y+z)\big)}\,,
\end{align}
and use the given solutions to compute discretization errors of the discrete solutions to \eqref{eq:discrete-variational-form}. A numerical solution, the reference grid, and a higher-order parametrized surface are visualized in \Cref{fig:sphere-shape}.

In the following we test four stable pairs of spaces to discretize the Stokes equation, namely
\begin{enumerate}
  \item Taylor--Hood (TH) element $(S_{h,\ku}, S_{h,\ku-1})$, with continuous velocity and pressure, \cite{TH1973numerical},
  \item MINI element $(S_h(\Gh;\FSpace_1\oplus\mathbb{B}_{3}), S_{h,1})$, which enriches the piecewise linear velocity space with element bubble functions, \cite{ArnoldEtAl1984Stable},
  \item Conforming Crouzeix--Raviart (CR) element $(S_h(\Gh;\FSpace_2+\mathbb{B}_{3}), S_{h,1}^\Broken)$, with discontinuous piecewise linear pressure, \cite{CR1973Conforming}, and
  \item P2P0 element, with continuous piecewise quadratic velocity and piecewise constant pressure, $(S_{h,2}, S_{h,0}^\Broken)$, \cite{BoffiEtAl2013Mixed}.
\end{enumerate}
Note that the elements with discontinuous pressure, CR and P2P0, are not suitable for the $b_2$ bilinear form and indeed result in non-converging or indefinite systems.

We plot in the following only the tangential $L^2$-error to compare with the results shown in \Cref{thm:tangential-L2-errors}. In \Cref{fig:a1-b2-bases} we show the measured errors and corresponding experimental convergence orders for three surface discretization parameters $\kg\in\{1,2,3\}$. For the $\kg=1$ case the bound $2\kg-1$ limits the convergence for all finite elements. Also the absolute errors are very close. For piecewise quadratic surfaces with $\kg=2$, the bound in the convergence comes from the interpolation errors and the bound $\ku+1$. Here it is advantages to go for higher-order elements like the Taylor--Hood or Crouzeix--Raviart. For the low-order elements, the P2P0 element has lower absolute errors than the MINI element due to its better approximation of the velocity component. For the even higher order case $\kg=3$ only the Taylor-Hood element can increase the convergence order to the optimal order 4.

\begin{figure}[ht!]
  \centering
  \begin{subfigure}[b]{0.32\textwidth}
      \centering
      \includegraphics[width=\textwidth]{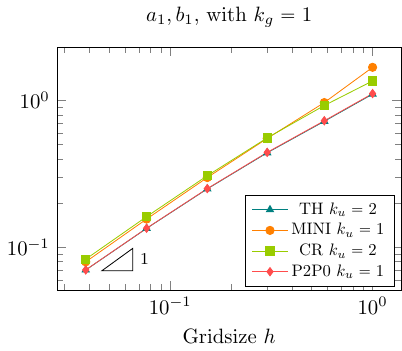}
  \end{subfigure}
  \hfill
  \begin{subfigure}[b]{0.32\textwidth}
      \centering
      \includegraphics[width=\textwidth]{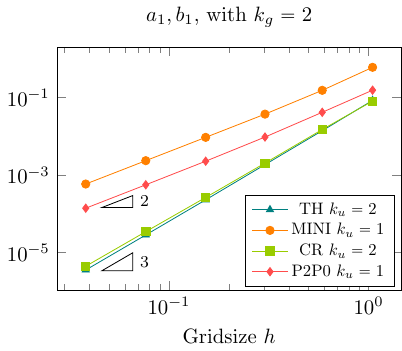}
  \end{subfigure}
  \hfill
  \begin{subfigure}[b]{0.32\textwidth}
      \centering
      \includegraphics[width=\textwidth]{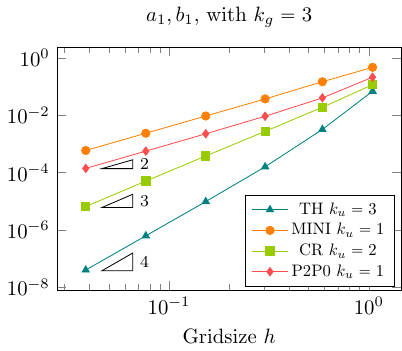}
  \end{subfigure}
  \caption{(Colours online) Tangential $L^2$-errors $\Norm{\PPh(\bm{u}^e-\bm{u}_{h})}{\Gh}$ on the discrete surface $\Gh$}\label{fig:a1-b2-bases}
\end{figure}

In a second study, we use the bilinear form $a_2$ for different curvature approximations $\kk$. The curvature $\KhTilde$ is thereby constructed from a higher order approximation of the surface $\G$ of order $\kg$ or $\kg+1$. Due to \Cref{rem:gauss-curv-approx}, in this case we get only even approximation orders $\kk$ of the continuous Gaussian curvature $\K$. \Cref{fig:a2-b1-kk} shows a comparison for different surface approximations. In the case of odd geometry orders, e.g., $\kg\in\{1,3\}$, the corresponding intrinsic curvature $\KhTilde=\Kh$ is only of order $\kk=\kg-1$. This is not sufficient to obtain the optimal convergence order in the tangential $L^2$-norm. Even in the  case $\kg=2$, where $\kk=\kg$, the curvature approximation limits the convergence order. We conclude that for the bilinear form $a_2$ the parametrization with $\KhTilde=\Kh$ and $\ku=\kg$ is suboptimal and a higher order curvature approximation is required. This must be one or even two orders better than the intrinsic curvature. A higher-order reconstruction of the mean curvature is discussed in \cite{HLZ2015Stabilized,FZ2019cut}, and an approach for a reconstruction of the Gaussian curvature in \cite{GawlikNeunteufel2023ScalarCurvature,GopalakrishnanEtAl2023DistributionalCurvature}. However, this might be an expensive task. Thus, additional knowledge about the surface often needs to be provided in order to use the bilinear form $a_2$.

\begin{figure}[ht!]
  \centering
  \begin{subfigure}[b]{0.32\textwidth}
      \centering
      \includegraphics[width=\textwidth]{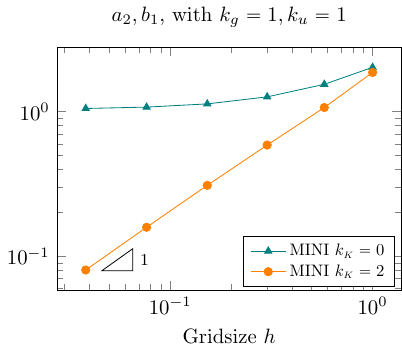}
  \end{subfigure}
  \hfill
  \begin{subfigure}[b]{0.32\textwidth}
      \centering
      \includegraphics[width=\textwidth]{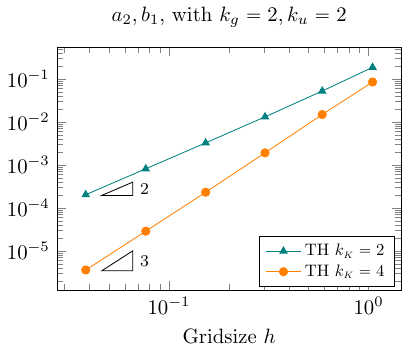}
  \end{subfigure}
  \hfill
  \begin{subfigure}[b]{0.32\textwidth}
      \centering
      \includegraphics[width=\textwidth]{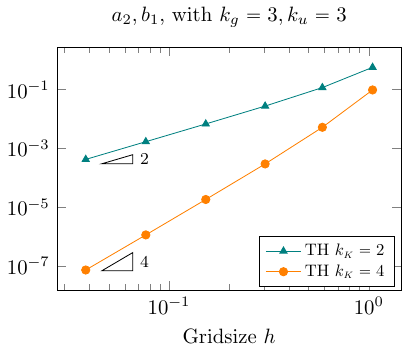}
  \end{subfigure}
  \caption{(Colours online) Tangential $L^2$-errors $\Norm{\PPh(\bm{u}^e-\bm{u}_{h})}{\Gh}$ on the discrete surface $\Gh$ for the bilinear form $a_2$}\label{fig:a2-b1-kk}
\end{figure}

The pressure $L^2$-norm error is visualized in \Cref{fig:p-a1-b1} for the three geometry orders $\kg=1,2,3$. Except for some minor deviations from the results in \Cref{thm:tangential-L2-errors}, the numerical data reflects the theory.

\begin{figure}[ht!]
  \centering
  \begin{subfigure}[b]{0.32\textwidth}
      \centering
      \includegraphics[width=\textwidth]{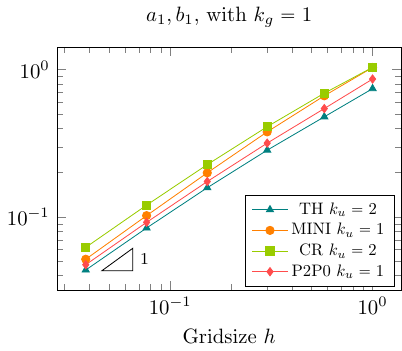}
  \end{subfigure}
  \hfill
  \begin{subfigure}[b]{0.32\textwidth}
      \centering
      \includegraphics[width=\textwidth]{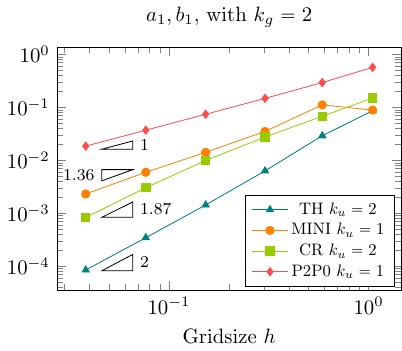}
  \end{subfigure}
  \hfill
  \begin{subfigure}[b]{0.32\textwidth}
      \centering
      \includegraphics[width=\textwidth]{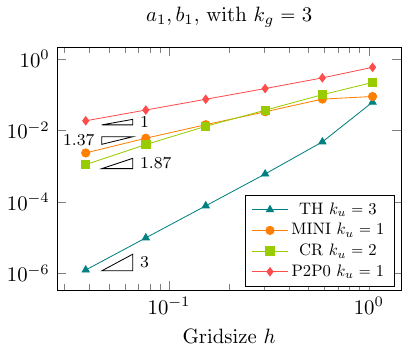}
  \end{subfigure}
  \caption{(Colours online) Pressure $L^2$-errors $\Norm{p^e-p_{h}}{\Gh}$ on the discrete surface $\Gh$}\label{fig:p-a1-b1}
\end{figure}

%% file: 09_appendix.tex
\appendix
\appendixpage
\addappheadtotoc

\section{Properties of macroelements}

\subsection{Proof of Lemma \ref{lem:prop-flattening}}\label{proof:prop-flattening}
\begin{proof}
	Let $\hat{x}\in\hat{\El}_{j}$.
	If $\hat{\El}_j=\hat{\El}$, the selected element, we have $\bar{\mu}=1$ and the assertion follow directly.
	If $\hat{\El}_i$ is direct neighbor of $\hat{\El}_j$ we can estimate the jump in the normal by comparing against the normal $\bm{n}$ of the continuous surface $\G$ in a point on the edge $\hat{x}_{ij}\in\hat{\El}_i\cap\hat{\El}_j$,
	\begin{align*}%\label{eq:normal-jump-estimate}
		\Norm{\hat{\bm{n}}_i - \hat{\bm{n}}_j}{}
		\leq \Norm{\hat{\bm{n}}_i - \bm{n}(\hat{x}_{ij})}{} + \Norm{\bm{n}(\hat{x}_{ij}) - \hat{\bm{n}}_j}{} \leq C h\,,
	\end{align*}
	using standard geometric estimates \cite{DE2013Finite}. Since the elements in $\hat{\g}$ are connected via edges we can find a chain of direct neighboring elements from all $\hat{\El}_j$ to $\hat{\El}$.
	Thus, we have
	\begin{align*}
		\Abs{1- \|\hat{\PP}_j\hat{\bm{t}}^k\|} & \leq \|\hat{\bm{t}}^k -\hat{\PP}_j\hat{\bm{t}}^k\|  \leq \Abs{\langle\hat{\bm{t}}^k,\hat{\bm{n}}_j \rangle}\leq \|\hat{\bm{n}}_j-\hat{\bm{n}}\|\leq Ch \quad \textrm{for }k=1,2,%\\
		%		\intertext{and}
		%		\Abs{1- \|\hat{\PP}_j\hat{\bm{t}}^1\times\hat{\PP}_j\hat{\bm{t}}^2\|} &
		%		\leq 1-\|\hat{\PP}_j\hat{\bm{t}}^1\|\|\hat{\PP}_j\hat{\bm{t}}^2\|=
		%		\left(1-\|\hat{\PP}_j\hat{\bm{t}}^1\|\right)\|\hat{\PP}_j\hat{\bm{t}}^2\|
		%		+ (1-\|\hat{\PP}_j\hat{\bm{t}}^2\|)
		%		\leq Ch,
	\end{align*}
	This immediately implies \eqref{eq:property-flattening-map-length}. Moreover, we have
	\begin{align*}
		\restr{\bar{\mu}}{\hat{\El}_{j}}=\operatorname{det}(\D_{\hat{x}}\restr{\bar{\pi}}{\hat{\El}_j}) &=
			\langle\hat{\bm{t}}^1, \hat{\PP}_j\hat{\bm{t}}^1\rangle\cdot\langle\hat{\bm{t}}^2, \hat{\PP}_j\hat{\bm{t}}^2\rangle - \langle\hat{\bm{t}}^1, \hat{\PP}_j\hat{\bm{t}}^2\rangle\cdot\langle\hat{\bm{t}}^2, \hat{\PP}_j\hat{\bm{t}}^1\rangle = \|\hat{\PP}_j\hat{\bm{t}}^1\times\hat{\PP}_j\hat{\bm{t}}^2\| \langle\hat{\bm{n}}, \hat{\bm{n}}_j\rangle\,,
	\end{align*}
	and
	\begin{align*}
	\Abs{\|\hat{\PP}_j\hat{\bm{t}}^1\times\hat{\PP}_j\hat{\bm{t}}^2\|\langle\hat{\bm{n}}, \hat{\bm{n}}_j\rangle - 1}
%	&= 1- \|\hat{\PP}_j\hat{\bm{t}}^1\times\hat{\PP}_j\hat{\bm{t}}^2\|\langle\hat{\bm{n}}, \hat{\bm{n}}_j\rangle\\
%	&= \frac{1}{2}\left(1-\|\hat{\PP}_j\hat{\bm{t}}^1\times\hat{\PP}_j\hat{\bm{t}}^2\|^2\right)  + \frac{1}{2}\|\hat{\bm{n}} -\|\hat{\PP}_j\hat{\bm{t}}^1\times\hat{\PP}_j\hat{\bm{t}}^2\|\hat{\bm{n}}_j \|^2\\
	&= \frac{1}{2}\left(\langle \hat{\bm{t}}^1, \hat{\bm{n}}_j\rangle^2+\langle \hat{\bm{t}}^2, \hat{\bm{n}}_j\rangle^2\right)  + \frac{1}{2}\|\hat{\bm{n}} -\|\hat{\PP}_j\hat{\bm{t}}^1\times\hat{\PP}_j\hat{\bm{t}}^2\|\hat{\bm{n}}_j \|^2\\
%	&= \frac{1}{2}\left(\langle \hat{\bm{t}}^1, \hat{\bm{n}}_j-\hat{\bm{n}}\rangle^2+\langle \hat{\bm{t}}^2, \hat{\bm{n}}_j-\hat{\bm{n}}\rangle^2\right)  + \frac{1}{2}\left(\|\hat{\bm{n}} - \hat{\bm{n}}_j +(1-\|\hat{\PP}_j\hat{\bm{t}}^1\times\hat{\PP}_j\hat{\bm{t}}^2\|)\hat{\bm{n}}_j \|\right)^2\\
	%
	&\leq 2\|\hat{\bm{n}}_j-\hat{\bm{n}}\|^2  +(1-\|\hat{\PP}_j\hat{\bm{t}}^1\times\hat{\PP}_j\hat{\bm{t}}^2\|)^2\\
	&\leq Ch^2.
	\end{align*}
\end{proof}

\subsection{Proof of Lemma \ref{lem:prop-parametric-macroelement}}\label{proof:prop-parametric-macroelement}
\begin{proof}
  We restrict the dicussion to a single element $\bar{T}$. With $F_h=\pi_h\circ\bar{\pi}^{-1}$ and $\hat{x}=\bar{\pi}^{-1}(\bar{x})$ we get by chain rule $\D_{\bar{x}} F_h(\bar{x}) = \D_{\hat{x}}\pi_h(\hat{x}) \cdot \D_{\hat{x}}\bar{\pi}(\hat{x})^{-1}$.
	As we have from \Cref{lem:prop-flattening} $\|\D_{\hat{x}}\bar{\pi}(\hat{x})-\Id\|\leq Ch$, for $h<h_0$ small enough we can express $\D_{\hat{x}}\bar{\pi}(\hat{x})^{-1}$ as a Neumann series
	\begin{equation*}
		\D_{\hat{x}}\bar{\pi}(\hat{x})^{-1} = \sum_{n=0}^{\infty}(\Id - \D_{\hat{x}}\bar{\pi}(\hat{x}))^n.
	\end{equation*}
	Thus, we have from \Cref{lem:prop-flattening} and the definition of $\bar{\pi}$ as piecewise affine
	\begin{align*}
		\|	\D_{\hat{x}}\bar{\pi}(\hat{x})^{-1}\|&\leq 1+Ch,\\
		\|	\D_{\hat{x}}\bar{\pi}(\hat{x})^{-1}-\Id\|&\leq \|	\D_{\hat{x}}\bar{\pi}(\hat{x})^{-1}\|\|	\D_{\hat{x}}\bar{\pi}(\hat{x})-\Id\| \leq Ch,\\
		\Abs{1 -\det	\D_{\hat{x}}\bar{\pi}(\hat{x})^{-1}}& \leq \Abs{\det	(\D_{\hat{x}}\bar{\pi}(\hat{x}))^{-1}}\Abs{1-\det	(\D_{\hat{x}}\bar{\pi}(\hat{x}))}\leq Ch^2, \textrm{and}\\
		\D_{\hat{x}}^k\bar{\pi}(\hat{x})^{-1}& = 0\quad \textrm{for } k\geq 2.
	\end{align*}
	This implies for a jump across an edge $\hat{e}=\bar{\pi}^{-1}(\bar{e})$ with $\bar{e}\in\bar{\mathcal{E}}$
	\begin{align*}
		\LInfNorm{\Jump{\D_{\hat{x}} \bar{\pi}^{-1}}}{\hat{e}}
			\leq\LInfNorm{\D_{\hat{x}}\bar{\pi}_{+}^{-1}-\Id}{\hat{e}} + \LInfNorm{\D_{\hat{x}}\bar{\pi}_{-}^{-1}-\Id}{\hat{e}} \leq Ch.
	\end{align*}
	For the piecewise polynomial approximation $\pi_h=\pi_{h,\kg}$ of $\pi$, we have
	\begin{alignat*}{3}
		\|\D_{\hat{x}}^{k}\pi_h (\hat{x})\|&\leq \|\D_{\hat{x}}^{k}\pi (\hat{x})\| + \|\D_{\hat{x}}^{k}\pi_h (\hat{x}) - \D_{\hat{x}}^{k}\pi (\hat{x}) \| \leq C (1+h^{\kg+1-k})\|\pi\|_{C^{\kg+1}(\hat{\El})}\quad &\textrm{for } 0&\leq k\leq \kg,\\
		\D_{\hat{x}}^{k}\pi_h (\hat{x})& = 0\quad &\textrm{for } k&\geq \kg+1.
	\end{alignat*}
	For the determinant, we have (see e.g. \cite[Lemma 4.4]{DE2013Finite})
	\begin{align*}
		\|{1-\sqrt{\det(\transposed{(\D_{\hat{x}}\pi_h)} \D_{\hat{x}}\pi_h)}}\|_{L^{\infty}(\hat{\El})}\leq Ch^{\kg+1}.
	\end{align*}
	For jump terms on an edge $\hat{e} = \hat\El_{+}\cap \hat\El_{-}$, we conclude
	\begin{align*}
		\LInfNorm{\Jump{\D_{\hat{x}} {\pi}_{h}}}{\hat{e}}
		&\leq\LInfNorm{\D_{\hat{x}} {\pi}_{h}^{+}-\D_{\hat{x}} {\pi}}{\hat{e}}
			 + \LInfNorm{\D_{\hat{x}} {\pi}_{h}^{-}-\D_{\hat{x}} {\pi}}{\hat{e}}\\
		&\leq Ch^{\kg}\|\pi\|_{C^{\kg+1}}.
	\end{align*}
	The estimates on the norm of $F_h$ and on the jump of $\D F_h$ now follow from the chain and product rules of differentiation. The estimate on the determinant follow from
	\begin{align*}
		\Abs{1 - \mu_h}\leq \Abs{\det(\D_{\hat{x}}\bar{\pi})^{-1}}\Abs{1-\sqrt{\det(\transposed{(\D_{\hat{x}}\pi_h)}\D_{\hat{x}}\pi_h)}}
		+ \sqrt{\det(\transposed{(\D_{\hat{x}}\pi_h)}\D_{\hat{x}}\pi_h)}\Abs{\det(\D_{\hat{x}}\bar{\pi})^{-1}-1}\leq Ch^2.
	\end{align*}
	For the derivative of the determinant, we use
	\begin{align*}
		\LInfNorm{\D_{\hat{x}}\sqrt{\operatorname{det}(\D_{\hat{x}}\pi_h(\hat{x}))^{T}\D_{\hat{x}}\pi_h(\hat{x}))}}{\hat{T}}
		&= \LInfNorm{[\operatorname{tr}\left(\operatorname{adj}(\D_{\hat{x}}\pi_h(\hat{x}))\cdot\D_{\hat{\theta}^i}\D_{\hat{x}}\pi_h(\hat{x})\right)]_i}{\hat{T}}\\
		& \leq C\frac{\LInfNorm{\D_{\hat{x}}\pi_h}{\hat{T}}^3 }{\LInfNorm{\sqrt{\operatorname{det}(\D_{\hat{x}}\pi_h(\hat{x}))^{T}\D_{\hat{x}}\pi_h(\hat{x}))}}{\hat{T}}} \LInfNorm{\D_{\hat{x}}^2\pi_h}{\hat{T}}\;.
	\end{align*}
	To obtain the estimate, we use that $\Abs{\D_{\hat{x}}^2\pi}\leq C h$, almost everywhere. This follows from the representation $\pi(x) = x - d(x)\bm{n}(\pi(x))$, boundedness of the curvature of $\G$ and the estimate on the distance and normal, cf. \cite[Lemma 4.1]{DE2013Finite}. With $\LInfNorm{\D_{\hat{x}}^2\pi - \D_{\hat{x}}^2\pi_h}{\hat{T}}\leq C h^{\kg-1}$ for $\kg > 1$ \cite[Proposition 2.3]{Demlow2009Higher} and $\D_{\hat{x}}^2\pi_{h,1}=0$ we then obtain
	\begin{align*}
		\LInfNorm{\D_{\hat{x}}^2\pi_h}{\hat{T}} &\leq \LInfNorm{\D_{\hat{x}}^2\pi - \D_{\hat{x}}^2\pi_h}{\hat{T}} + \LInfNorm{\D_{\hat{x}}^2\pi}{\hat{T}} \leq C h\,.
	\end{align*}
	 With $\D_{\hat{x}}\bar{\pi}^{-1}$ bounded and $\D_{\hat{x}}^2\bar{\pi}^{-1}=0$ the assertion follows.
\end{proof}